\documentclass[a4paper,10pt,twoside]{amsart}

\usepackage{a4wide}
\usepackage[latin1]{inputenc}
\usepackage{amssymb}
\usepackage{mathrsfs}
\usepackage{paralist}
\usepackage{url}
\usepackage[bookmarks=false,pdfborder={0 0 0.05}]{hyperref}
\usepackage{color}
\usepackage{tikz}
\usepackage{pgfplots}

\usepackage{amsmath,amssymb,amsfonts}

\usepackage{color}
\usepackage{tikz}
\usepackage{pgfplots}
\usetikzlibrary{calc}
\usetikzlibrary{shapes}
\usetikzlibrary{intersections}
\usetikzlibrary{patterns}
\usetikzlibrary{matrix}

\usepackage{enumerate}

\newtheorem{theorem}{Theorem}[section]
\newtheorem{proposition}[theorem]{Proposition}
\newtheorem{corollary}[theorem]{Corollary}
\newtheorem{lemma}[theorem]{Lemma}

\theoremstyle{definition}
\newtheorem{definition}[theorem]{Definition}
\newtheorem{example}[theorem]{Example}

\newtheorem{remark}[theorem]{Remark}

\def\R{\mathbb{R}}
\def\Z{\mathbb{Z}}
\def\N{\mathbb{N}}
\def\TwoCov{\mathcal{C}^2_{\c}}
\def\tTwoCov{\mathcal{C}^2_{\tc}}
\def\c{{\bf c}}
\def\tc{{\bf {\tilde c}}}
\def\sameelement{\approx}
\def\equptocomm{\sim}

\definecolor{darkblue}{rgb}{0,0,0.7} % darkblue color
\newcommand{\darkblue}{\color{darkblue}} % darkblue command
\newcommand{\Dfn}[1]{\emph{\darkblue #1}} % emphasis of a definition

\newcommand{\Perm}{{\mathsf {Perm}}}
\newcommand{\Asso}{{\mathsf {Asso}}}
\newcommand{\set}[2]{\left\{#1\vphantom{#2}\right.\;\left|\;\vphantom{#1}#2\right\}}

\newcommand\uword{\mathbf u}
\newcommand\vword{\mathbf v}
\newcommand\wword{\mathbf w}
\newcommand\wc{\pmb{w}^\c}
\newcommand\woword{\mathbf w _\circ}
\newcommand\wo{w_\circ}
\newcommand\woc{{\wo^\c}}
\newcommand\bwoc{{\pmb{w}_{\pmb \circ}^\c}}
\newcommand\sym{\Sigma}
\newcommand\precdot{\prec\negmedspace\negmedspace\cdot}

\DeclareMathOperator{\conv}{\mathsf{conv}}

\title{Cambrian acyclic domains: counting $c$-singletons}

\author[J.-P.~Labb\'e]{Jean-Philippe Labb\'e$^{1,2}$}
\address[J.-P.~Labb\'e]{Institut f\"ur Mathematik, Freie Universit\"at Berlin, Arnimallee 2, 14195 Berlin, Germany}
\email{labbe@math.fu-berlin.de}
\urladdr{http://page.mi.fu-berlin.de/labbe}
\thanks{$^1$supported by by the DFG Collaborative Research Center TRR~109 ``Discretization in Geometry and Dynamics''.}
\thanks{$^2$partially supported by a FQRNT Doctoral scholarship.}

\author[C.~Lange]{Carsten E. M. C. Lange$^{1}$}
\address[C.E.M.C.~Lange]{Fakult\"at f\"ur Mathematik, Technische Universit\"at M\"unchen, D-85748, Garching, Germany}
\email{lange@ma.tum.de}
\urladdr{https://www-m10.ma.tum.de/bin/view/Lehrstuhl/CarstenLange}


\begin{document}

\begin{abstract}
We study the size of certain acyclic domains that arise from geometric and combinatorial constructions. These acyclic domains consist 
of all permutations visited by commuting equivalence classes of maximal reduced decompositions if we consider the symmetric group and, 
more generally, of all $c$-singletons of a Cambrian lattice associated to the weak order of a finite Coxeter group. For this reason, 
we call these sets \emph{Cambrian acyclic domains}. Extending a closed formula of Galambos--Reiner for a particular acyclic domain 
called Fishburn's alternating scheme, we provide explicit formulae for the size of any Cambrian acyclic domain and characterize the 
Cambrian acyclic domains of minimum or maximum size.

\end{abstract}

\maketitle

\section{Introduction}

Examples of $c$-singletons include certain acyclic domains in social choice theory, natural partial orders of crossings 
in pseudoline arrangements as well as certain vertices of particular convex polytopes called permutahedra and associahedra 
in discrete geometry. We first describe these objects and outline the relationship between these incarnations.

\medskip
Acyclic domains are of great interest in social choice theory because of their importance for the following voting process: 
voters choose among a given collection of linear orders on~$m$ candidates and the result of the ballot obeys the order imposed 
by the majority for each pair of candidates. As already mentioned by the Marquis de Condorcet in~1785~\cite{condorcet_1785}, 
not every collection of linear orders yields a transitive order on the candidates in every election. Collections that do 
guarantee transitivity are called \emph{acyclic domains} or \emph{Condorcet domains}. According to 
Fishburn~\cite[Introduction]{fishburn_acyclic_sets_a_progress_report}, the fundamental problem to determine the maximum 
cardinality of an acyclic domain for a given number of candidates is one of most fascinating and intractable combinatorial 
problems in social choice theory. Abello as well as Chameni-Nembua describe different constructions of ``large'' acyclic sets. 
They use maximal chains of the weak order on the symmetric group~$\sym_m$ \cite{abello_the_weak_order_1991} and study covering 
distributive sublattices of the weak order on~$\sym_m$ \cite{chameniNembua_regle_majoritaire_1989}.

\medskip
Galambos and~Reiner \cite{galambos_acyclic_2008} prove that maximal acyclic domains constructed by Abello coincide with those 
of Chameni-Nembua and describe them in terms of higher Bruhat orders. Moreover, they show that acyclic domains obtained from 
Fishburn's alternating scheme \cite{fishburn_acyclic_sets} are a special case of Chameni-Nembua's construction and prove that 
the cardinality of Fishburn's acyclic domain is given by
\begin{equation}\label{formula:fishburn}
	\mathsf{fb}(m)=
	2^{m-3}(m+3)-
	\begin{cases}
		\frac{m-1}{2}\binom{m-1}{\tfrac{m-1}{2}} & \text{ for odd $m$,}\\[3mm]
		\frac{2m-3}{2}\binom{m-2}{\tfrac{m-2}{2}} & \text{ for even $m$.}\\
	\end{cases}
\end{equation}
Weakening a conjecture of Fishburn \cite[Conjecture~2]{fishburn_acyclic_sets}, Galambos and Reiner conjecture that
$\mathsf{fb}(m)$ is a tight upper bound on the cardinality of acyclic sets described in terms of higher Bruhat 
orders~\cite[Conjecture~1]{galambos_acyclic_2008}. 
We notice that Knuth had a conjecture related to the one of Galambos and Reiner discussing his Equation~(9.8)~\cite[p.~39]{knuth_axioms_1992}.
Felsner and Valtr as well as Danilov, Karzanov and Koshevoy mention counterexamples to these conjectures~\cite{felsner_coding_2011,danilov_condorcet_2012}. 
Galambos and Reiner base the formula for $\mathsf{fb}(m)$ and their conjecture on 
counting extensions of a certain pseudoline arrangement by adding a new pseudoline and relate these extensions to elementarily 
equivalent maximal chains in the weak order on~$\sym_m$. 

\medskip
Planar pseudoline arrangements with contact points as well as pseudo- and multitriangulations are systematically studied by Pilaud and Pocchiola 
using the framework of \emph{networks}~\cite{pilaud_multitriangulations_2012}. Subsequently, Pilaud and Santos construct polytopes from a given 
network and relate their combinatorics to the combinatorics of triangulations of point configurations \cite{pilaud_brick_2012}. For well-chosen 
networks, they construct associahedra (or Stasheff polytopes) which essentially coincide with a family of realizations obtained from the 
permutahedron by Hohlweg and Lange~\cite{hohlweg_realizations_2007}. This family provides a geometric interpretation of Reading's Cambrian 
lattices \cite{reading_cambrian_2006}. Cambrian lattices are remarkable as they generalize the Tamari lattice as lattice quotient of the weak 
order on~$\sym_m$ in two ways. First, distinct lattice quotients are obtained by choosing different Coxeter elements~$c$ and yield distinct 
realizations of the associahedron from the permutahedron. Second, the construction of distinct lattice quotients extends from the symmetric 
group~$\sym_m$ to the weak order of any finite Coxeter group~$W$. Hohlweg, Lange and Thomas then identify \emph{$c$-singletons} as fundamental 
objects of Cambrian lattices and use them to derive distinct polytopal realizations of generalized associahedra from 
$W$-permutahedra~\cite{hohlweg_permutahedra_2011}. Generalized associahedra are CW-complexes defined in the context of cluster algebras of 
finite type \cite{fomin_y_systems_2003} that coincide with associahedra in type~$A$. Finally, Pilaud and Stump extend the construction of 
polytopes from Pilaud and Santos to any finite Coxeter group and, analogous to type~$A$, essentially reobtain realizations of generalized 
associahedra discovered by Hohlweg, Lange and Thomas \cite{pilaud_brick_2011}.

\medskip
Two interpretations of $c$-singletons described in~\cite{hohlweg_permutahedra_2011} are fundamental for our work. First, the geometric construction 
of generalized associahedra from $W$-permutahedra exhibits $c$-singletons as the common vertices of both polytopes and, second, $c$-singletons are 
combinatorially described as prefixes of a certain reduced expression for the longest element~$\wo\in W$ up to commutations. For Coxeter groups of 
type~$A$, the latter interpretation translates to higher Bruhat orders: the set of $c$-singletons for a fixed Coxeter element~$c\in \sym_m$ is 
precisely the set of all elements~$w\in \sym_m$ visited by the maximal chains contained in a certain equivalence class of elementarily equivalent 
maximal chains determined by~$c$. Galambos and Reiner showed in type~$A$ that these elements coincide with certain maximal acyclic domains and for 
this reason we define a Cambrian acyclic domain as the set of $c$-singletons for a given Coxeter element~$c$ of a finite Coxeter group~$W$. The 
main results of this article are
\begin{compactitem}
	\item Theorem~\ref{thm:singleton_cuts} that provides a combinatorial description for the cardinality of a Cambrian acyclic domain for any 
		finite Coxeter system $(W,S)$ and any Coxeter element~$c$.
	\item Theorem~\ref{thm:bounds_acyc_domain} that characterizes the possible choices of~$c$ to minimize and maximize the 
		cardinality of an Cambrian acyclic domain for any finite Coxeter system~$(W,S)$. 
\end{compactitem}  
These results solve Problem~3.1 of~\cite[Chapter~8]{mueller_associahedra_2012}. Even though we mentioned above that the conjecture of Galambos and 
Reiner is not true in general, Theorem~\ref{thm:bounds_acyc_domain} proves that the conjecture holds if it is restricted to the large subclass of 
acyclic domains: Fishburn's alternating scheme yields the maximum cardinality for Cambrian acyclic domains of type~$A$.  

\medskip
The article is organized as follows. In Section~\ref{sec:prelim}, we discuss the results in type~$A$. 
Sections~\ref{ssec:Assoc_and_singletons}--\ref{sec:counting_c_singletons_type_A} provide a unified description of Cambrian 
acyclic domains as geometric entities in terms of vertices of convex polytopes, as pseudoline arrangements, and as certain 
order ideals for type~$A$. Moreover, we derive formulae for the cardinality of Cambrian acyclic domains and give a new proof 
of Equation~\eqref{formula:fishburn} using hypergeometric series in Section~\ref{sec:counting_in_type_a}.
Section~\ref{sec:general} generalizes the discussion from type~$A$ to other finite types. We introduce and discuss relevant 
notions in Sections~\ref{subsec:nat_part_order}--\ref{ssec:crossingcutpaths} before proving Theorem~\ref{thm:singleton_cuts} 
in Section~\ref{subsec:enumerate_c_singletons}. More precisely, a poset called \emph{natural partial order} by Galambos and 
Reiner~\cite{galambos_acyclic_2008} as well as \emph{heap} by Viennot~\cite{viennot_1986} and Stembridge~\cite{stembridge_1996} 
is introduced in Section~\ref{subsec:nat_part_order}. In Section~\ref{ssec:equivalence_classes} we introduce $c$-singletons 
of a finite Coxeter system~$(W,S)$ and show that the weak order on $c$-singletons is isomorphic to the lattice of order ideals 
of a well-chosen natural partial order. In Section~\ref{subsec:2covers_cut_paths}, Hasse diagrams of natural partial orders 
are embedded in a cylindrical oriented graph that we call $2$-cover. The $2$-cover replaces the network used in type~$A$ as 
framework to count $c$-singletons in arbitrary type. The extension of a pseudoline arrangements considered by Galambos and 
Reiner in type~$A$ is replaced by cut paths introduced in Section~\ref{subsec:2covers_cut_paths}. It turns out that the total 
number of cut paths in the $2$-cover exceeds the size~$\mathsf S_c$ of Cambrian domains and the difference can be expressed in 
terms of ``crossing'' cut paths discussed in Section~\ref{ssec:cutpaths}. In Section~\ref{sec:examples}, we illustrate 
Theorem~\ref{thm:singleton_cuts}: we explicitly compute the cardinality of Cambrian acyclic domains for various finite types and 
different choices of Coxeter elements. In Section~\ref{sec:bounds}, we finally derive lower and upper bounds for the cardinality 
of Cambrian acyclic domains. The examples discussed in Section~\ref{sec:examples} cover all possibilities to minimize and 
maximize the size~$\mathsf S_c$ of Cambrian domains.

\medskip
We assume familiarity with basic notions of convex polytopes and of Coxeter group theory and refer to~\cite{ziegler_lectures_1995} as well as~\cite{humphreys_reflection_1992} for details.
 
\section{Associahedra, pseudoline arrangements and $c$-singletons in type~$A$}\label{sec:prelim}

\subsection{Associahedra and $c$-singletons}\label{ssec:Assoc_and_singletons}

An associahedron is a simple convex polytope of a particular combinatorial type. The underlying combinatorial structure relates to various 
branches of mathematics as mentioned in \cite{tamari_thesis_1951,mueller_associahedra_2012} or \cite{stasheff_homotopy_1963}. We follow Lee 
and consider triangulations of a convex $(n+3)$-gon to define the combinatorics of an \Dfn{$n$-dimensional associahedron}~\cite{lee_assoc_1989}. 
A plethora of distinct polytopal realizations is known for the associahedron, e.g. \cite{haiman_associahedron_1984, billera_secondary_1990,
gelfand_kapranov_zelevinsky_1994, hohlweg_realizations_2007, devadoss_realization_2009, pilaud_brick_2012}, we focus on a family of realizations 
described by Hohlweg and Lange \cite{hohlweg_realizations_2007, lange_using_2013} that generalizes~\cite{loday_realization_2004} and relates 
directly to triangulations of a labeled $(n+3)$-gon~$P$ and to the symmetric group~$\sym_{n+1}$. The resulting $n$-dimensional associahedra 
and the labelings of~$P$ are parametrized by Coxeter elements~$c\in \sym_{n+1}$. We refer to the labeled polygons as~$P_c$ and to the various 
polytopal realizations of associahedra as~$\Asso_c$.

Assume that~$P$ is a convex $(n+3)$-gon in the plane with no two vertices on a vertical line. To obtain the labeled polygon~$P_c$, we label 
the vertices of~$P$ from smallest to greatest $x$-coordinate using the integers~$0$ to~$n+2$. Without loss of generality, we assume that the 
vertices labeled $0$ and $n+2$ lie on a horizontal line. This induces a partition of $[n+1]:=\{1,2,\ldots, n+1\}$ into a down set 
$D_c = \{ d_1 <d_2 <\cdots <d_k \}$ and an up set $U_c = \{ u_1 < u_2 < \cdots <u_\ell\}$ where the vertices in~$U_c$ lie in the upper hull 
of~$P_c$ and the vertices in~$D_c$ in the lower hull. We have special notation in the following two situations. If $U_c=\varnothing$ then we 
replace the subscript~$c$ by~$Lod$ to remind of Loday who gave a combinatorial interpretation of the vertex coordinates of~$\Asso_{Lod}$ 
\cite{loday_realization_2004}. If $U_c=\set{d\in\N}{0<d< n+2\text{ and $d$ even}}$ then we replace the subscript~$c$ by~$alt$. This reminds 
of alternating (or bipartite) Coxeter elements and relates to Fishburn's alternating scheme. 

The labeled $(n+3)$-gons~$P_c$ are characterized by certain permutations~$\pi_c$ with one peak
which describe a relabeling to obtain~$P_c$ from~$P_{Lod}$, see Figure~\ref{fig:downup} for two examples.

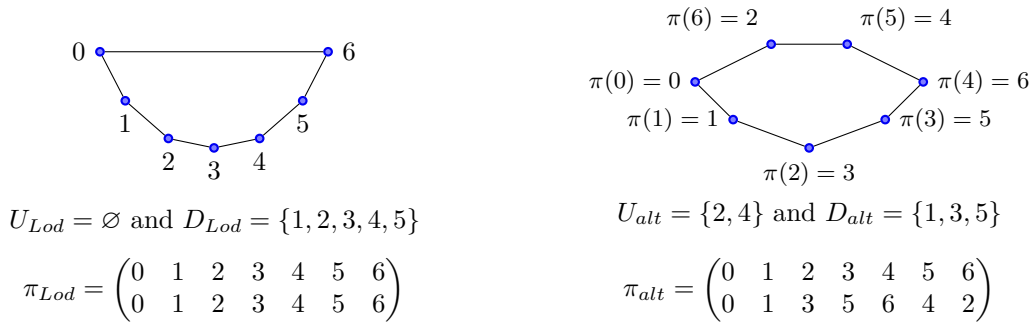
\begin{figure}[!htp]
\begin{center}
	\begin{tabular}{c@{\hspace{2cm}}c}
	\begin{tikzpicture}[spec/.style={gray},
	pointille/.style={dashed},
	axe/.style={color=black, very thick},
	sommet/.style={inner sep=1pt,circle,draw=blue!95!black,fill=blue!50,thick,anchor=base}]

	\node[sommet, label=left:$0$] (0) at (0,0.8/2) {};
	\node[sommet, label=below:$1$] (1) at (1/3,-1/4) {};
	\node[sommet, label=below:$2$] (2) at (1.8/2,-1.5/2) {};
	\node[sommet, label=below:$3$] (3) at (3/2,-1.75/2) {};
	\node[sommet, label=below:$4$] (4) at (4.2/2,-1.5/2) {};
	\node[sommet, label=below:$5$] (5) at (16/6,-1/4) {};
	\node[sommet, label=right:$6$] (6) at (6/2,0.8/2) {};
	
	\draw (0) -- (1) -- (2) -- (3) -- (4) -- (5) -- (6) -- (0);
	
	\node at (1.5,-1.85) {$U_{Lod}=\varnothing$ and $D_{Lod}=\{1,2,3,4,5\}$};
	\node at (1.5,-2.75) {$\pi_{Lod}=\left(\begin{matrix}
										0 & 1 & 2 & 3 & 4 & 5 & 6\\
										0 & 1 & 2 & 3 & 4 & 5 & 6
									 \end{matrix}\right)$};
	\end{tikzpicture}
	
	&
	
	\begin{tikzpicture}[spec/.style={gray},
	pointille/.style={dashed},
	axe/.style={color=black, very thick},
	sommet/.style={inner sep=1pt,circle,draw=blue!95!black,fill=blue!50,thick,anchor=base}]

	\node[sommet, label=left:{\small$\pi(0)=0$}] (0) at (0,0) {};
	\node[sommet, label=left:{\small$\pi(1)=1$}] (1) at (1/2,-1/2) {};
	\node[sommet, label=above left:{\small$\pi(6)=2$}] (2) at (2/2,1/2) {};
	\node[sommet, label=below:{\small$\pi(2)=3$}] (3) at (3/2,-1.75/2) {};
	\node[sommet, label=above right:{\small$\pi(5)=4$}] (4) at (4/2,1/2) {};
	\node[sommet, label=right:{\small$\pi(3)=5$}] (5) at (5/2,-1/2) {};
	\node[sommet, label=right:{\small$\pi(4)=6$}] (6) at (6/2,0) {};

	\draw (0) -- (1) -- (3) -- (5) -- (6) -- (4) -- (2) -- (0);
	
	\node at (1.5,-1.75) {$U_{alt}=\{2,4\}$ and $D_{alt}=\{1,3,5\}$};
	\node at (1.5,-2.75) {$\pi_{alt}=\left(\begin{matrix}
										0 & 1 & 2 & 3 & 4 & 5 & 6\\
										0 & 1 & 3 & 5 & 6 & 4 & 2
									 \end{matrix}\right)$};		
	\end{tikzpicture}	
\end{tabular}
	\caption{\label{fig:downup} Two examples of labeled heptagons $P_c$.}
\end{center}
\end{figure}

Moreover, the set of labeled $(n+3)$-gons~$P_c$ is in bijection with orientations of Coxeter graphs of type~$A$ and all Coxeter elements~$c$ 
of~$\sym_{n+1}$ \cite{shi_enumeration_1997}. This bijection is crucial to extend the construction of associahedra to generalized associahedra 
for arbitrary finite Coxeter groups \cite{hohlweg_permutahedra_2011}.

Any proper diagonal~$\delta$ of~$P_c$ yields a facet-defining inequality $H^\delta_\geq$ for~$\Asso_c$ as follows. Let~$B_\delta$ be the label 
set of vertices of $P_c$ which lie strictly below the line supporting $\delta$ and include the endpoints of $\delta$ which are in $U_c$ and set 
$H^\delta_\geq := \set{x\in \R^{n+1}}{\sum_{i\in B_\delta}x_i \geq {{n+1}\choose{|B_\delta|}}}$.

\begin{theorem}[{\cite[Proposition~1.3]{hohlweg_realizations_2007},\cite[Corollary~7]{lange_using_2013}}]\hfill\break
	For every labeled $(n+3)$-gon~$P_c$, the polytope
	\[
		\Asso_c = \set{x\in\R^{n+1}}
						{\begin{matrix}
							\sum_{i\in [n+1]} x_i = \frac{(n+1)(n+2)}{2} \text{ and}\\[2mm]
							x\in H^\delta_\geq \text{ for any proper diagonal $\delta$ of $P_c$}
						\end{matrix}}
	\]
	is a particular realization of an $n$-dimensional associahedron.
\end{theorem}

Each associahedron~$\Asso_c$ is an instance of a generalized permutahedron, introduced by Postnikov,  as it is obtained from the classical permutahedron
\begin{align*}
	\Perm_n &= \conv\set{\left(\pi(1),\dots,\pi(n+1)\right)^\top\in\R^{n+1}}{\pi\in \sym_{n+1}}\\
			&= \set{x\in\R^{n+1}}
					{\begin{matrix}
						\sum_{i\in [n+1]} x_i = \frac{(n+1)(n+2)}{2} \text{ and}\\[2mm]
						\sum_{i\in I}x_i \geq {{n+1}\choose{|I|}} \text{ for any nonempty $I\subset [n+1]$}
					\end{matrix}}
\end{align*}
by discarding some facet-defining inequalities \cite{postnikov_faces_2008,postnikov_permutahedra_2009}. Following~\cite[Section~2.3]{hohlweg_permutahedra_2011}, 
a \Dfn{$c$-singleton} is a common vertex of~$\Perm_n$ and of~$\Asso_c$. 

\begin{figure}[!htbp]
\begin{center}
	\input{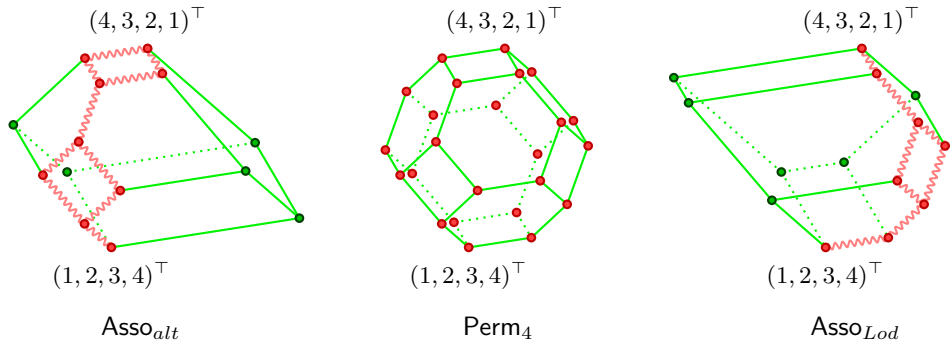}
	\caption{\label{fig:asso} The permutahedron~$\Perm_4$ and the associahedra $\Asso_{alt}$ and $\Asso_{Lod}$. Red vertices of the associahedra indicate $c$-singletons and maximal paths from $(1,2,3,4)^\top$ to $(4,3,2,1)^\top$ along red zig-zag edges of associahedra correspond to elementarily equivalent maximal chains in the weak order. }
\end{center}
\end{figure}

From Figure~\ref{fig:asso}, where the $3$-dimensional polytopes~$\Perm_4$, $\Asso_{Lod}$ and $\Asso_{alt}$ are shown, it is immediate that the 
number of $c$-singletons as well as the number of paths from $(1,2,3,4)^\top$ to $(4,3,2,1)^\top$ in the $1$-skeleton of~$\Asso_c$ visiting only 
$c$-singletons depends on~$c$. For later use, we remark that the realization of~$\Asso_c$ is completely determined by $U_c \cap \{2,3,\ldots,n+1\}$
and assume without loss of generality
\[
	D_c = \{ d_1 =1 < d_2 < \cdots < d_k \}
	\qquad\text{and}\qquad
	U_c = \{ u_1 < u_2 < \cdots <u_\ell\}.
\]

\subsection{Pseudoline arrangements and $c$-singletons}

Using the duality of points and lines in the Euclidean plane, we now describe~$\Asso_c$ and $c$-singletons in terms of pseudoline arrangements, 
see \cite{pilaud_multitriangulations_2012} and \cite{pilaud_brick_2012} for details.  We visualize pseudoline arrangements on an alternating and 
sorting  network~$\mathcal N_c$ that encodes the combinatorics of the point configuration of~$P_c$: $\mathcal N_c$ consists of $n+3$ horizontal 
lines and $\binom{n+3}{2}=\tfrac{(n+3)(n+2)}{2}$ \Dfn{commutators} which are vertical line segments connecting consecutive horizontal lines in 
an alternating way. A commutator is at \Dfn{level}~$i$ if it connects the horizontal lines~$i$ and~$i+1$ of~$\mathcal N_c$ (counted from bottom 
to top starting with~$0$). Additionally, we label the ends of horizontal lines at the left end of~$\mathcal N_c$ from~$0$ to~$n+2$ bottom to top 
and from~$0$ to~$n+2$ top to bottom at the right end. Figure~\ref{fig:networks} illustrates these notions for~$\mathcal N_{Lod}$ and~$\mathcal N_{alt}$ 
which correspond to~$P_{Lod}$ and~$P_{alt}$ of Figure~\ref{fig:downup}. 

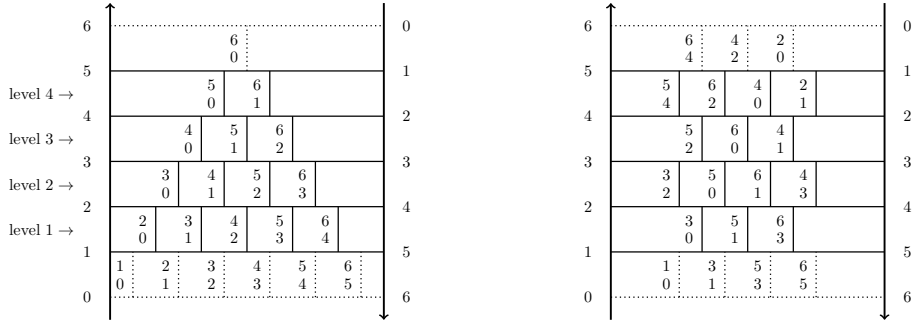
\begin{figure}[!htbp]
\begin{center}
	\begin{tabular}{c@{\hspace{2cm}}c}
	\scalebox{0.6}{
	\begin{tikzpicture}[spec/.style={dotted}]
	
	\draw[thick,spec] (0,-0.5) -- (6,-0.5);
	\draw[thick] (0,0.5) -- (6,0.5);
	\draw[thick] (0,1.5) -- (6,1.5);
	\draw[thick] (0,2.5) -- (6,2.5);
	\draw[thick] (0,3.5) -- (6,3.5);
	\draw[thick] (0,4.5) -- (6,4.5);
	\draw[thick,spec] (0,5.5) -- (6,5.5);
	
	\node at (-1.5,1) {level $1\rightarrow$};
	\node at (-1.5,2) {level $2\rightarrow$};
	\node at (-1.5,3) {level $3\rightarrow$};
	\node at (-1.5,4) {level $4\rightarrow$};

	\node at (-0.5,-0.5) {0};	
	\node at (-0.5,0.5) {1};
	\node at (-0.5,1.5) {2};
	\node at (-0.5,2.5) {3};
	\node at (-0.5,3.5) {4};
	\node at (-0.5,4.5) {5};
	\node at (-0.5,5.5) {6};

	\node at (6.5,-0.5) {6};
	\node at (6.5,0.5) {5};
	\node at (6.5,1.5) {4};
	\node at (6.5,2.5) {3};
	\node at (6.5,3.5) {2};
	\node at (6.5,4.5) {1};
	\node at (6.5,5.5) {0};

	\draw[thick,spec] (0.5,-0.5) node[label={[label distance=-0.1cm]135:$0$}]  {} -- (0.5,0.5) node[label={[label distance=-0.1cm]225:$1$}]  {};
	\draw[thick] (1,0.5) node[label={[label distance=-0.1cm]135:$0$}]  {} -- (1,1.5) node[label={[label distance=-0.1cm]225:$2$}]  {};
	\draw[thick] (1.5,1.5) node[label={[label distance=-0.1cm]135:$0$}]  {} -- (1.5,2.5) node[label={[label distance=-0.1cm]225:$3$}]  {};
	\draw[thick] (2,2.5) node[label={[label distance=-0.1cm]135:$0$}]  {} -- (2,3.5) node[label={[label distance=-0.1cm]225:$4$}]  {};
	\draw[thick] (2.5,3.5) node[label={[label distance=-0.1cm]135:$0$}]  {} -- (2.5,4.5) node[label={[label distance=-0.1cm]225:$5$}]  {};
	\draw[thick,spec] (3,4.5) node[label={[label distance=-0.1cm]135:$0$}]  {} -- (3,5.5) node[label={[label distance=-0.1cm]225:$6$}]  {};

	\draw[thick,spec] (1.5,-0.5) node[label={[label distance=-0.1cm]135:$1$}]  {} -- (1.5,0.5) node[label={[label distance=-0.1cm]225:$2$}]  {};
	\draw[thick] (2,0.5) node[label={[label distance=-0.1cm]135:$1$}]  {} -- (2,1.5) node[label={[label distance=-0.1cm]225:$3$}]  {};
	\draw[thick] (2.5,1.5) node[label={[label distance=-0.1cm]135:$1$}]  {} -- (2.5,2.5) node[label={[label distance=-0.1cm]225:$4$}]  {};
	\draw[thick] (3,2.5) node[label={[label distance=-0.1cm]135:$1$}]  {} -- (3,3.5) node[label={[label distance=-0.1cm]225:$5$}]  {};
	\draw[thick] (3.5,3.5) node[label={[label distance=-0.1cm]135:$1$}]  {} -- (3.5,4.5) node[label={[label distance=-0.1cm]225:$6$}]  {};

	\draw[thick,spec] (2.5,-0.5) node[label={[label distance=-0.1cm]135:$2$}]  {} -- (2.5,0.5) node[label={[label distance=-0.1cm]225:$3$}]  {};
	\draw[thick] (3,0.5) node[label={[label distance=-0.1cm]135:$2$}]  {} -- (3,1.5) node[label={[label distance=-0.1cm]225:$4$}]  {};
	\draw[thick] (3.5,1.5) node[label={[label distance=-0.1cm]135:$2$}]  {} -- (3.5,2.5) node[label={[label distance=-0.1cm]225:$5$}]  {};
	\draw[thick] (4,2.5) node[label={[label distance=-0.1cm]135:$2$}]  {} -- (4,3.5) node[label={[label distance=-0.1cm]225:$6$}]  {};

	\draw[thick,spec] (3.5,-0.5) node[label={[label distance=-0.1cm]135:$3$}]  {} -- (3.5,0.5) node[label={[label distance=-0.1cm]225:$4$}]  {};
	\draw[thick] (4,0.5) node[label={[label distance=-0.1cm]135:$3$}]  {} -- (4,1.5) node[label={[label distance=-0.1cm]225:$5$}]  {};
	\draw[thick] (4.5,1.5) node[label={[label distance=-0.1cm]135:$3$}]  {} -- (4.5,2.5) node[label={[label distance=-0.1cm]225:$6$}]  {};

	\draw[thick,spec] (4.5,-0.5) node[label={[label distance=-0.1cm]135:$4$}]  {} -- (4.5,0.5) node[label={[label distance=-0.1cm]225:$5$}]  {};
	\draw[thick] (5,0.5) node[label={[label distance=-0.1cm]135:$4$}]  {} -- (5,1.5) node[label={[label distance=-0.1cm]225:$6$}]  {};

	\draw[thick,spec] (5.5,-0.5) node[label={[label distance=-0.1cm]135:$5$}]  {} -- (5.5,0.5) node[label={[label distance=-0.1cm]225:$6$}]  {};
		
	\draw[very thick,->] (0,-1) -> (0,6) {};
	\draw[very thick,<-] (6,-1) -> (6,6) {};
		
	\end{tikzpicture}
	}
	
	&
	
	\scalebox{0.6}{
	\begin{tikzpicture}[spec/.style={dotted}]
	
	\draw[thick,spec] (0,-0.5) -- (6,-0.5);
	\draw[thick] (0,0.5) -- (6,0.5);
	\draw[thick] (0,1.5) -- (6,1.5);
	\draw[thick] (0,2.5) -- (6,2.5);
	\draw[thick] (0,3.5) -- (6,3.5);
	\draw[thick] (0,4.5) -- (6,4.5);
	\draw[thick,spec] (0,5.5) -- (6,5.5);

	\node at (-0.5,-0.5) {0};	
	\node at (-0.5,0.5) {1};
	\node at (-0.5,1.5) {2};
	\node at (-0.5,2.5) {3};
	\node at (-0.5,3.5) {4};
	\node at (-0.5,4.5) {5};
	\node at (-0.5,5.5) {6};

	\node at (6.5,-0.5) {6};
	\node at (6.5,0.5) {5};
	\node at (6.5,1.5) {4};
	\node at (6.5,2.5) {3};
	\node at (6.5,3.5) {2};
	\node at (6.5,4.5) {1};
	\node at (6.5,5.5) {0};

	\draw[thick,spec] (1.5,-0.5) node[label={[label distance=-0.1cm]135:$0$}]  {} -- (1.5,0.5) node[label={[label distance=-0.1cm]225:$1$}]  {};
	\draw[thick] (1.5,1.5) node[label={[label distance=-0.1cm]135:$2$}]  {} -- (1.5,2.5) node[label={[label distance=-0.1cm]225:$3$}]  {};
	\draw[thick] (1.5,3.5) node[label={[label distance=-0.1cm]135:$4$}]  {} -- (1.5,4.5) node[label={[label distance=-0.1cm]225:$5$}]  {};

	\draw[thick] (2,0.5) node[label={[label distance=-0.1cm]135:$0$}]  {} -- (2,1.5) node[label={[label distance=-0.1cm]225:$3$}]  {};
	\draw[thick] (2,2.5) node[label={[label distance=-0.1cm]135:$2$}]  {} -- (2,3.5) node[label={[label distance=-0.1cm]225:$5$}]  {};
	\draw[thick,spec] (2,4.5) node[label={[label distance=-0.1cm]135:$4$}]  {} -- (2,5.5) node[label={[label distance=-0.1cm]225:$6$}]  {};

	\draw[thick,spec] (2.5,-0.5) node[label={[label distance=-0.1cm]135:$1$}]  {} -- (2.5,0.5) node[label={[label distance=-0.1cm]225:$3$}]  {};
	\draw[thick] (2.5,1.5) node[label={[label distance=-0.1cm]135:$0$}]  {} -- (2.5,2.5) node[label={[label distance=-0.1cm]225:$5$}]  {};
	\draw[thick] (2.5,3.5) node[label={[label distance=-0.1cm]135:$2$}]  {} -- (2.5,4.5) node[label={[label distance=-0.1cm]225:$6$}]  {};

	\draw[thick] (3,0.5) node[label={[label distance=-0.1cm]135:$1$}]  {} -- (3,1.5) node[label={[label distance=-0.1cm]225:$5$}]  {};
	\draw[thick] (3,2.5) node[label={[label distance=-0.1cm]135:$0$}]  {} -- (3,3.5) node[label={[label distance=-0.1cm]225:$6$}]  {};
	\draw[thick,spec] (3,4.5) node[label={[label distance=-0.1cm]135:$2$}]  {} -- (3,5.5) node[label={[label distance=-0.1cm]225:$4$}]  {};
			
	\draw[thick,spec] (3.5,-0.5) node[label={[label distance=-0.1cm]135:$3$}]  {} -- (3.5,0.5) node[label={[label distance=-0.1cm]225:$5$}]  {};
	\draw[thick] (3.5,1.5) node[label={[label distance=-0.1cm]135:$1$}]  {} -- (3.5,2.5) node[label={[label distance=-0.1cm]225:$6$}]  {};
	\draw[thick] (3.5,3.5) node[label={[label distance=-0.1cm]135:$0$}]  {} -- (3.5,4.5) node[label={[label distance=-0.1cm]225:$4$}]  {};

	\draw[thick] (4,0.5) node[label={[label distance=-0.1cm]135:$3$}]  {} -- (4,1.5) node[label={[label distance=-0.1cm]225:$6$}]  {};
	\draw[thick] (4,2.5) node[label={[label distance=-0.1cm]135:$1$}]  {} -- (4,3.5) node[label={[label distance=-0.1cm]225:$4$}]  {};
	\draw[thick,spec] (4,4.5) node[label={[label distance=-0.1cm]135:$0$}]  {} -- (4,5.5) node[label={[label distance=-0.1cm]225:$2$}]  {};

	\draw[thick,spec] (4.5,-0.5) node[label={[label distance=-0.1cm]135:$5$}]  {} -- (4.5,0.5) node[label={[label distance=-0.1cm]225:$6$}]  {};
	\draw[thick] (4.5,1.5) node[label={[label distance=-0.1cm]135:$3$}]  {} -- (4.5,2.5) node[label={[label distance=-0.1cm]225:$4$}]  {};
	\draw[thick] (4.5,3.5) node[label={[label distance=-0.1cm]135:$1$}]  {} -- (4.5,4.5) node[label={[label distance=-0.1cm]225:$2$}]  {};

	\draw[very thick,->] (0,-1) -> (0,6) {};
	\draw[very thick,<-] (6,-1) -> (6,6) {};
		
	\end{tikzpicture}
	}
\end{tabular}
	\caption{\label{fig:networks} Two networks for pseudoline arrangements with labeled commutators:~$\mathcal N_{Lod}$ (left) and~$\mathcal N_{alt}$ (right) correspond to the labeled heptagons of Figure~\ref{fig:downup}. The $1$-kernels~$\mathcal K_{Lod}$ and~$\mathcal K_{alt}$ are obtained by deletion of the dotted line segments.}
\end{center}
\end{figure}

Now, pseudoline~$i$ (supported by~$\mathcal N_c$) is an abscissa monotone path on $\mathcal N_c$ starting on the left at label~$i$ and ending 
right at label~$i$ and a pseudoline arrangement on $\mathcal N_c$ is a collection of pseudolines such that any pair intersects precisely along 
one commutator called \Dfn{crossing}. A commutator that is touched by two pseudolines and traversed by none is called \Dfn{contact}. There is a 
unique pseudoline arrangement with $n+3$ pseudolines on $\mathcal N_c$ which induces a labeling of all commutators by the two unique pseudolines 
that traverse along it, see  Figure~\ref{fig:networks}. The reader may prove the following facts:
\begin{compactenum}[i)]
	\item labeled commutators of~$\mathcal N_c$ at levels~$0$ and~$n+1$ are in bijection to boundary diagonals of~$P_c$.
	\item $\mathcal N_c$ is determined by~$\mathcal N_{Lod}$ and~$\pi_c$. The inversions of~$\pi_c$ 
		label commutators in the bottom right of~$\mathcal N_{Lod}$ which can be moved to the upper left part 
		if we temporarily consider~$\mathcal N_{Lod}$ as M\"obius strip by identifying its sides. Relabeling 
		the commutators by $\pi_c$ yields~$\mathcal N_c$.
\end{compactenum}

The \Dfn{$1$-kernel}~$\mathcal K_c$ of the network~$\mathcal N_c$ is the network obtained from~$\mathcal N_c$ by deletion of the horizontal 
lines~$0$ and~$n+2$ as well as all commutators touching these lines. On~$\mathcal K_c$, we use notions induced by~$\mathcal N_c$, for example, 
the level of a commutator or its label are inherited from~$\mathcal N_c$. Triangulations of~$P_c$ are now in bijection to pseudoline arrangements 
with $n+1$ pseudolines supported by~$\mathcal K_c$: diagonals of a triangulation correspond to the contacts of a unique pseudoline arrangement 
on~$\mathcal K_c$ (\cite[Theorem~23]{pilaud_multitriangulations_2012}). The simple fact that a commutator labeled by the endpoints of a proper 
diagonal~$\delta$ of~$P_c$ is at level~$|B_\delta|$ extends~\cite[Proposition~1.4]{hohlweg_realizations_2007} and~\cite[Proposition~20]{lange_using_2013} 
by statement~iii) below:

\begin{proposition}%[{\cite[Proposition~1.4]{hohlweg_realizations_2007},\cite[Proposition~20]{lange_using_2013}}]
	\label{prop:singleton}
	Let $v$ be a vertex of $\Asso_c$ with corresponding triangulation~$T_v$ of~$P_c$ and let~$C_v$ be the set of commutators of~$\mathcal N_c$ labeled by proper diagonals of $T_v$. The following statements are equivalent:
	\begin{compactenum}[i)]
		\item $v$ is a vertex of $\Perm_n$.
		\item The proper diagonals $\delta_i$ of $T_v$ can be ordered such that 
				$\varnothing \subset B_{\delta_1} \subset \ldots \subset B_{\delta_{n}} \subset [n+1]$.
		\item $C_v$ contains one commutator from each level of~$\mathcal K_c$ and commutators from consecutive levels are adjacent.
	\end{compactenum}
\end{proposition}

Proposition~\ref{prop:singleton} shows that a $c$-singleton for~$\Asso_c$ corresponds to a path which traverses the $1$-kernel~$\mathcal K_c$ 
from bottom to top and ascents whenever possible. We call such a path a \Dfn{greedy ordinate monotone path} on~$\mathcal K_c$. In the theory of 
pseudoline arrangements, a greedy ordinate monotone path is known as a pseudoline from the south pole to the north pole which extends the original 
arrangement by a new pseudoline. 

\subsection{Order ideals and $c$-singletons}\label{sec:counting_c_singletons_type_A}

It is possible to give another description of $c$-singletons that uses neighbouring transpositions $s_i=(i\ \ i+1)$, $1\leq i\leq n$, if we 
combine Proposition~\ref{prop:singleton} with~\cite[Theorem~2.2]{hohlweg_permutahedra_2011}: a permutation $\pi\in\Sigma_{n+1}$ is a $c$-singleton 
of~$\Asso_c$ if and only if there is a reduced word for~$\pi$ in $s_1,\ldots,s_n$ that is a prefix up to commutation of a particular reduced 
expression~$\bwoc$ of the reverse permutation~$\wo=[n+1, n, \ldots, 1]$ (given here in one-line notation). Although this point of view will be 
used in Section~\ref{sec:general} to define $c$-singletons for arbitrary irreducible finite Coxeter systems~$(W,S)$, we  directly describe prefixes 
up to commutations of~$\bwoc$ in type~$A$ using order ideals of a poset $(\mathcal S,\prec_c)$ associated to~$\mathcal K_c$. 

Let~$\mathcal S$ be the set that contains one copy of the transposition $s_i$ for each bounded region of~$\mathcal K_c$ at level~$i$ and distinguish 
copies of the same transposition by their associated region. Define the partial order~$\prec_{c}$ on $\mathcal S$ as the transitive closure of the 
covering relation $s_i\rightarrow s_{j}$ that satisfies
\begin{compactenum}[i)]
	\item $|i-j|=1$;
	\item the bounded regions associated to~$s_i$ and~$s_j$ intersect in a (nonempty) horizontal line segment;
	\item the commutator bounding the region of~$s_i$ to the left is left of the region associated to~$s_{j}$.
\end{compactenum} 

Two Hasse diagrams for~$(\mathcal S, \prec_{c})$ associated to~$\mathcal K_{alt}$ are illustrated in Figure~\ref{fig:dualpicture}. Each greedy ordinate 
monotone path~$\frak p$ in~$\mathcal K_c$ is a cut of the Hasse diagram~$(\mathcal S, \prec_{c})$ that partitions the vertex set~$\mathcal S$ of this 
oriented graph in two sets and the set~$\mathcal S_{\frak p}\subseteq \mathcal S$ below~$\frak p$ is an order ideal of~$(\mathcal S, \prec_{c})$. 
Therefore~$\mathcal S_{\frak p}$ corresponds to a $c$-singleton.
 
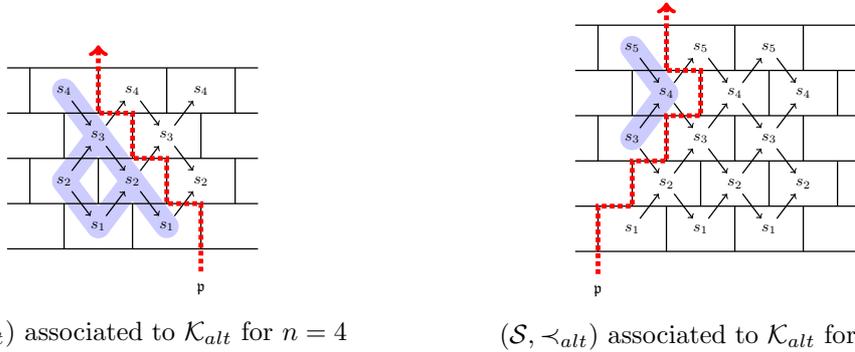
\begin{figure}[!ht]
\begin{center}
	\begin{tabular}{c@{\hspace{2cm}}c}
	\begin{minipage}{0.38\linewidth}
	\rule{0cm}{5mm}
	
	\noindent
	$ $\hspace{10mm}
	\scalebox{0.6}{
	\begin{tikzpicture}[spec/.style={dotted},
				decoration={snake,amplitude=.4mm,segment length=1mm},
				copies/.style={line width=15pt,color=blue!20,draw opacity=1,line cap=round,line join=round}]
	
	\draw[copies] (1.5,3) -- (3.75,0);
	\draw[copies] (1.5,1) -- (2.25,2);
	\draw[copies] (1.5,1) -- (2.25,0);
	\draw[copies] (2.25,0) -- (3.0,1);
		
	\draw[thick] ( 0.25,-0.5) -- (5.75,-0.5);
	\draw[thick] ( 0.25, 0.5) -- (5.75, 0.5);
	\draw[thick] ( 0.25, 1.5) -- (5.75, 1.5);
	\draw[thick] ( 0.25, 2.5) -- (5.75, 2.5);
	\draw[thick] ( 0.25, 3.5) -- (5.75, 3.5);

	\draw[thick]			 	(0.75, 0.5)	{} -- (0.75,1.5)   {};
	\draw[thick] 				(0.75, 2.5)	{} -- (0.75,3.5)   {};
	
	\draw[thick] 				(1.5,-0.5)	{} -- (1.5,0.5) 	  {};
	\draw[thick] 				(1.5, 1.5)	{} -- (1.5,2.5) 	  {};
	
	\draw[thick]	 			(2.25,0.5)	{} -- (2.25,1.5)   {};
	\draw[thick] 				(2.25,2.5) 	{} -- (2.25,3.5)   {};

	\draw[thick] 				(3,-0.5) 	{} -- (3,0.5) 	  {};
	\draw[thick] 				(3, 1.5) 	{} -- (3,2.5) 	  {};

	\draw[thick]		 		(3.75,0.5)	{} -- (3.75,1.5)   {};
	\draw[thick] 				(3.75,2.5) 	{} -- (3.75,3.5)   {};

	\draw[thick] 				(4.5,-0.5) 	{} -- (4.5,0.5) 	  {};
	\draw[thick] 				(4.5, 1.5) 	{} -- (4.5,2.5) 	  {};

	\draw[thick]				(5.25,0.5)	{} -- (5.25,1.5)   {};
	\draw[thick] 				(5.25,2.5)	{} -- (5.25,3.5)   {};

	\node (s4-1) at (1.5,3) {$s_4$};
	\node (s4-2) at (3.0,3) {$s_4$};
	\node (s4-3) at (4.5,3) {$s_4$};
	
	\node (s3-1) at (2.25,2) {$s_3$} edge[thick,<-] (s4-1) edge[thick,->] (s4-2);
	\node (s3-2) at (3.75,2) {$s_3$} edge[thick,<-] (s4-2) edge[thick,->] (s4-3);

	\node (s2-1) at (1.5,1) {$s_2$} edge[thick,->] (s3-1);
	\node (s2-2) at (3.0,1) {$s_2$} edge[thick,<-] (s3-1) edge[thick,->] (s3-2);
	\node (s2-3) at (4.5,1) {$s_2$} edge[thick,<-] (s3-2);

	\node (s1-1) at (2.25,0) {$s_1$} edge[thick,<-] (s2-1) edge[thick,->] (s2-2);
	\node (s1-2) at (3.75,0) {$s_1$} edge[thick,<-] (s2-2) edge[thick,->] (s2-3);
	
	\draw[line width=3pt, densely dashed,red,->] (4.5,-1) -- (4.5,0.5) -- (3.75,0.5) -- (3.75,1.5) -- (3.0,1.5) -- (3.0,2.5) -- (2.25,2.5) -- (2.25,4);
	
	\node[label=below:$\frak p$] at (4.5,-1) {};
	
	\end{tikzpicture}
	}\\[-3mm]
	
	$(\mathcal S,\prec_{alt})$ associated to $\mathcal K_{alt}$ for $n=4$
	
	\end{minipage}
	&
	
	\begin{minipage}{0.38\linewidth}
	\hspace{8mm}
	\scalebox{0.6}{
	\begin{tikzpicture}[spec/.style={dotted},
				decoration={snake,amplitude=.4mm,segment length=1mm},
				copies/.style={line width=15pt,color=blue!20,draw opacity=1,line cap=round,line join=round}]
	
	\draw[copies] (1.5,2) -- (2.25,3) -- (1.5,4);
	
	\draw[thick] ( 0.25,-0.5) -- (6.50,-0.5);
	\draw[thick] ( 0.25, 0.5) -- (6.50,0.5);
	\draw[thick] ( 0.25, 1.5) -- (6.50,1.5);
	\draw[thick] ( 0.25, 2.5) -- (6.50,2.5);
	\draw[thick] ( 0.25, 3.5) -- (6.50,3.5);
	\draw[thick] ( 0.25, 4.5) -- (6.50,4.5);

	\draw[thick] (0.75,-0.5)	 {} -- (0.75,0.5)   {};
	\draw[thick] (0.75, 1.5)	 {} -- (0.75,2.5)   {};
	\draw[thick] (0.75, 3.5)	 {} -- (0.75,4.5)   {};

	\draw[thick] (1.5,0.5)		{} -- (1.5,1.5)   {};
	\draw[thick] (1.5,2.5)		{} -- (1.5,3.5)   {};

	\draw[thick] (2.25,-0.5)	{} -- (2.25,0.5)  {};
	\draw[thick] (2.25, 1.5)	{} -- (2.25,2.5)  {};
	\draw[thick] (2.25, 3.5)	{} -- (2.25,4.5)  {};

	\draw[thick] (3.00,0.5)     {} -- (3.00,1.5)    {};
	\draw[thick] (3.00,2.5)     {} -- (3.00,3.5)    {};
			
	\draw[thick] (3.75,-0.5)	{} -- (3.75,0.5)  {};
	\draw[thick] (3.75, 1.5)	{} -- (3.75,2.5)  {};
	\draw[thick] (3.75, 3.5)	{} -- (3.75,4.5)  {};

	\draw[thick] (4.50,0.5)		{} -- (4.50,1.5)    {};
	\draw[thick] (4.50,2.5)		{} -- (4.50,3.5)    {};

	\draw[thick] (5.25,-0.5)	{} -- (5.25,0.5)  {};
	\draw[thick] (5.25, 1.5)	{} -- (5.25,2.5)  {};
	\draw[thick] (5.25, 3.5)	{} -- (5.25,4.5)  {};

	\draw[thick] (6.00, 0.5) 	-- (6.00,1.5);
	\draw[thick] (6.00, 2.5) 	-- (6.00,3.5);

	\node (s5-1) at (1.5,4) {$s_5$};
	\node (s5-2) at (3.0,4) {$s_5$};
	\node (s5-3) at (4.5,4) {$s_5$};
	
	\node (s4-1) at (2.25,3) {$s_4$} edge[thick,<-] (s5-1) edge[thick,->] (s5-2);
	\node (s4-2) at (3.75,3) {$s_4$} edge[thick,<-] (s5-2) edge[thick,->] (s5-3);
	\node (s4-3) at (5.25,3) {$s_4$} edge[thick,<-] (s5-3);

	\node (s3-1) at (1.5,2) {$s_3$} edge[thick,->] (s4-1);
	\node (s3-2) at (3.0,2) {$s_3$} edge[thick,<-] (s4-1) edge[thick,->] (s4-2);
	\node (s3-3) at (4.5,2) {$s_3$} edge[thick,<-] (s4-2) edge[thick,->] (s4-3);

	\node (s2-1) at (2.25,1) {$s_2$} edge[thick,<-] (s3-1) edge[thick,->] (s3-2);
	\node (s2-2) at (3.75,1) {$s_2$} edge[thick,<-] (s3-2) edge[thick,->] (s3-3);
	\node (s2-3) at (5.25,1) {$s_2$} edge[thick,<-] (s3-3);

	\node (s1-1) at (1.5,0) {$s_1$} edge[thick,->] (s2-1);
	\node (s1-2) at (3.0,0) {$s_1$} edge[thick,<-] (s2-1) edge[thick,->] (s2-2);
	\node (s1-3) at (4.5,0) {$s_1$} edge[thick,<-] (s2-2) edge[thick,->] (s2-3);;

	\draw[line width=3pt, densely dashed,red,->] (0.75,-1) -- (0.75,0.5) -- (1.5,0.5) -- (1.5,1.5) -- (2.25,1.5) -- (2.25,2.5) -- (3.0,2.5) -- (3.0,3.5) -- (2.25,3.5) -- (2.25,5);
	
	\node[label=below:$\frak p$] at (0.75,-1) {};
			
	\end{tikzpicture}
	}\\[-3mm]
	
	$(\mathcal S,\prec_{alt})$ associated to $\mathcal K_{alt}$ for $n=5$
	\end{minipage}
\end{tabular}
	\caption{\label{fig:dualpicture} The Hasse diagram of $(\mathcal S,\prec_{alt})$. An ordinate monotone path~$\frak p$ (dashed red line) determines an order ideal~$\mathcal S_{\mathfrak p}$ of~$(\mathcal S,\prec_{alt})$ (shaded region).}
\end{center}
\end{figure}

\medskip
We briefly indicate the relation to the work of Galambos and Reiner as well as Manin and Schechtmann. For a given Coxeter element~$c$, each 
$c$-singleton~$w$ determines a unique greedy ordinate monotone path~$\frak p_w$ in the $1$-kernel~$\mathcal K_c$ as well as a unique pseudoline
arrangement~$\mathcal A_w$ supported by~$\mathcal K_c$ (where contacts are the commutators traversed by $\frak p_w$). The arrangements~$\mathcal A_w$ 
differ only by their embedding in $\mathcal K_c$. They describe a unique pseudoline arrangement~$\mathcal A_c$ that depends only on~$c$. Galambos 
and Reiner consider the ``natural partial order''~$\mathcal{P_A}$ on the crossings of~$\mathcal A_c$ which coincides with~$(\mathcal S,\prec_{c})$ 
described above. Moreover, they show that the order ideals of~$\mathcal{P_A}$ encode an equivalence class of elementarily equivalent maximal chains 
in the weak order on~$\sym_{n+1}$ as defined by Manin and Schechtmann~\cite{manin_HigherBraidGroups_1989}. Hence the set of $c$-singletons for a 
given Coxeter element~$c$ corresponds to an equivalence class of elementarily equivalent admissible permutations of ${[n+1]\choose 2}$ also studied 
by Ziegler~\cite{ziegler_higherBruhat_1993}.

\subsection{Counting $c$-singletons}\label{sec:counting_in_type_a}

We first derive a general formula for the number~${\mathsf S}_c$ of $c$-singletons by enumeration of greedy ordinate monotone paths of~$\mathcal K_c$ 
with $n+1$ horizontal lines. To this respect, we define the \Dfn{trapeze network}~$\mathcal{T}_c$ associated to~$\mathcal K_c$ as the maximal sorting 
network with $n+1$ horizontal lines with the following properties:
\begin{compactenum}[i)]
	\item $\mathcal{T}_c$ contains~$\mathcal{K}_c$,
	\item $\mathcal{T}_c$ is alternating,
	\item every commutator of~$\mathcal{T}_c$ is included in a greedy ordinate monotone path in $\mathcal{T}_c$ that contains a commutator 
		of~$\mathcal{K}_c$ at level~$n$.
\end{compactenum}
A commutator in~$\mathcal{T}_c$ which is not in~$\mathcal{K}_c$ is called a \Dfn{trapeze commutator}. If $\mathcal K_c=\mathcal K_{Lod}$ then 
$\mathcal T_{Lod}=\mathcal K_{Lod}$, so $\mathcal T_{Lod}$ contains no trapeze commutators. When~$\mathcal K_c=\mathcal K_{alt}$, the trapeze 
network~$\mathcal T_{alt}$ contains~$8$ trapeze commutators for $n=4$ and contains $10$ trapeze commutators for $n=5$, see Figure~\ref{fig:trapeze_A4}.

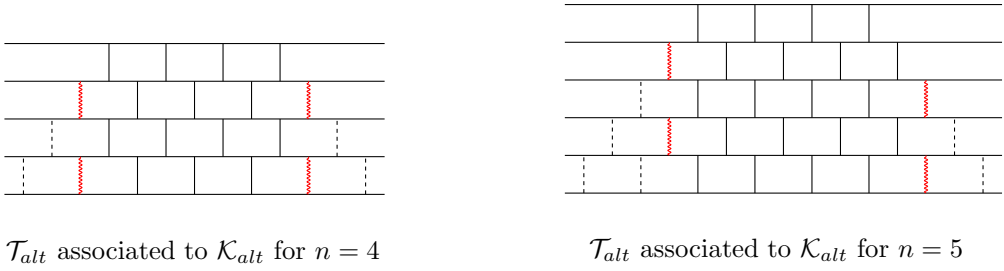
\begin{figure}[!ht]
\begin{center}
	\begin{tabular}{c@{\hspace{2cm}}c}
	\begin{minipage}{0.38\linewidth}
	\rule{0cm}{5mm}
	
	\noindent
	$ $\hspace{-2mm}
	\scalebox{0.5}{
	\begin{tikzpicture}[spec/.style={dotted},
				decoration={snake,amplitude=.4mm,segment length=1mm}]
	
	\draw[thick] (-3.0,-0.5) -- (7.0,-0.5);
	\draw[thick] (-3.0, 0.5) -- (7.0,0.5);
	\draw[thick] (-3.0, 1.5) -- (7.0,1.5);
	\draw[thick] (-3.0, 2.5) -- (7.0,2.5);
	\draw[thick] (-3.0, 3.5) -- (7.0,3.5);

	\draw[thick,dashed] 		(-2.5,-0.5) -- (-2.5,0.5);

	\draw[thick,dashed] 		(-1.75,0.5) -- (-1.75,1.5);

	\draw[thick,decorate,red]	(-1.00,-0.5) -- (-1.0,0.5);
	\draw[thick,decorate,red]	(-1.00, 1.5) -- (-1.0,2.5);

	\draw[thick]			 	(-0.25, 0.5)	{} -- (-0.25,1.5)   {};
	\draw[thick] 				(-0.25, 2.5)	{} -- (-0.25,3.5)   {};
	
	\draw[thick] 				(0.5,-0.5)	{} -- (0.5,0.5) 	  {};
	\draw[thick] 				(0.5, 1.5)	{} -- (0.5,2.5) 	  {};
	
	\draw[thick]	 			(1.25,0.5)	{} -- (1.25,1.5)   {};
	\draw[thick] 				(1.25,2.5) 	{} -- (1.25,3.5)   {};

	\draw[thick] 				(2,-0.5) 	{} -- (2,0.5) 	  {};
	\draw[thick] 				(2, 1.5) 	{} -- (2,2.5) 	  {};

	\draw[thick]		 		(2.75,0.5)	{} -- (2.75,1.5)   {};
	\draw[thick] 				(2.75,2.5) 	{} -- (2.75,3.5)   {};

	\draw[thick] 				(3.5,-0.5) 	{} -- (3.5,0.5) 	  {};
	\draw[thick] 				(3.5, 1.5) 	{} -- (3.5,2.5) 	  {};

	\draw[thick]				(4.25,0.5)	{} -- (4.25,1.5)   {};
	\draw[thick] 				(4.25,2.5)	{} -- (4.25,3.5)   {};

	\draw[thick,decorate,red]	(5.00,-0.5) -- (5.0,0.5);
	\draw[thick,decorate,red]	(5.00, 1.5) -- (5.0,2.5);

	\draw[thick,dashed] 		(5.75,0.5) -- (5.75,1.5);

	\draw[thick,dashed] 		(6.5,-0.5) -- (6.5,0.5);
	
	\end{tikzpicture}
	}\\
	
	$\mathcal T_{alt}$ associated to $\mathcal K_{alt}$ for $n=4$
	\end{minipage}
	
	&
	\begin{minipage}{0.4\textwidth}
	\hspace{-5mm}
	\scalebox{0.5}{
	\begin{tikzpicture}[spec/.style={dotted},
				decoration={snake,amplitude=.4mm,segment length=1mm}]
	
	\draw[thick] (-2.75,-0.5) -- (8.75,-0.5);
	\draw[thick] (-2.75, 0.5) -- (8.75,0.5);
	\draw[thick] (-2.75, 1.5) -- (8.75,1.5);
	\draw[thick] (-2.75, 2.5) -- (8.75,2.5);
	\draw[thick] (-2.75, 3.5) -- (8.75,3.5);
	\draw[thick] (-2.75, 4.5) -- (8.75,4.5);

	\draw[thick,dashed] 		(-2.25,-0.5) -- (-2.25,0.5);

	\draw[thick,dashed] 		(-1.5, 0.5) -- (-1.5,1.5);

	\draw[thick,dashed] 		(-0.75,-0.5) -- (-0.75,0.5);
	\draw[thick,dashed] 		(-0.75, 1.5) -- (-0.75,2.5);

	\draw[thick,decorate,red]	(0.00, 0.5) -- (0.00,1.5);
	\draw[thick,decorate,red]	(0.00, 2.5) -- (0.00,3.5);

	\draw[thick] (0.75,-0.5)	 {} -- (0.75,0.5)   {};
	\draw[thick] (0.75, 1.5)	 {} -- (0.75,2.5)   {};
	\draw[thick] (0.75, 3.5)	 {} -- (0.75,4.5)   {};

	\draw[thick] (1.5,0.5)		{} -- (1.5,1.5)   {};
	\draw[thick] (1.5,2.5)		{} -- (1.5,3.5)   {};

	\draw[thick] (2.25,-0.5)	{} -- (2.25,0.5)  {};
	\draw[thick] (2.25, 1.5)	{} -- (2.25,2.5)  {};
	\draw[thick] (2.25, 3.5)	{} -- (2.25,4.5)  {};

	\draw[thick] (3.00,0.5)     {} -- (3.00,1.5)    {};
	\draw[thick] (3.00,2.5)     {} -- (3.00,3.5)    {};
			
	\draw[thick] (3.75,-0.5)	{} -- (3.75,0.5)  {};
	\draw[thick] (3.75, 1.5)	{} -- (3.75,2.5)  {};
	\draw[thick] (3.75, 3.5)	{} -- (3.75,4.5)  {};

	\draw[thick] (4.50,0.5)		{} -- (4.50,1.5)    {};
	\draw[thick] (4.50,2.5)		{} -- (4.50,3.5)    {};

	\draw[thick] (5.25,-0.5)	{} -- (5.25,0.5)  {};
	\draw[thick] (5.25, 1.5)	{} -- (5.25,2.5)  {};
	\draw[thick] (5.25, 3.5)	{} -- (5.25,4.5)  {};

	\draw[thick] (6.00, 0.5) 	-- (6.00,1.5);
	\draw[thick] (6.00, 2.5) 	-- (6.00,3.5);

	\draw[thick,decorate,red]	(6.75,-0.5) -- (6.75,0.5);
	\draw[thick,decorate,red]	(6.75, 1.5) -- (6.75,2.5);

	\draw[thick,dashed] 		(7.5, 0.5) -- (7.5,1.5);
	
	\draw[thick,dashed] 		(8.25,-0.5) -- (8.25,0.5);
		
	\end{tikzpicture}
	}\\
	
	$\mathcal T_{alt}$ associated to $\mathcal K_{alt}$ for $n=5$
	\end{minipage}
\end{tabular}
	\caption{\label{fig:trapeze_A4} The trapeze network~$\mathcal T_{alt}$ with trapeze commutators drawn dashed or zig-zag. The trapeze commutators in $\Theta_{alt}$ are drawn zig-zag.}
\end{center}
\end{figure}

Clearly, any greedy ordinate monotone path in~$\mathcal{T}_c$ is either a greedy ordinate monotone path that traverses only commutators of~$\mathcal{K}_c$ 
or a greedy ordinate monotone path in~$\mathcal T_c$ that traverses at least one trapeze commutator. As there are~$|U_c|+2$ commutators of~$\mathcal K_c$ 
at level~$n$ and since there are~$2^{n-1}$ distinct ordinate monotone paths in~$\mathcal T_c$ that end at each commutator at level~$n$, we conclude that 
the number of greedy ordinate monotone paths in~$\mathcal{T}_c$ equals $(|U_c|+2)2^{n-1}$. It remains to count the ordinate monotone paths that traverse 
at least one trapeze commutator. Clearly, these paths are naturally partitioned by the last trapeze commutator traversed. Let~$\Theta_c$ denote the set 
of trapeze commutators that appear as last trapeze commutator of some greedy ordinate monotone path in~$\mathcal{T}_c$ and let~$\gamma_t$ denote the 
number of ordinate monotone paths in~$\mathcal{T}_c$ that stay in~$\mathcal K_c$ after traversing~$t\in \Theta_c$ at level~$\ell_t$. Then there
are~$\gamma_t2^{\ell_t-1}$ greedy ordinate monotone paths in~$\mathcal T_c$ with~$t\in\Theta_c$ as last trapeze commutator. We conclude
\[
	{\mathsf S}_c
	=
	(|U_c|+2)2^{n-1} - \sum_{t\in\Theta_c}\gamma_t2^{\ell_t-1}.
\]
A trivial consequence is ${\mathsf S}_{Lod}=2^n$ as $\Theta_{Lod}=U_{Lod}=\varnothing$ and $D_{Lod}=[n+1]$.

\medskip
\noindent
For a less trivial example, we prove ${\mathsf S}_{alt}=\mathsf{fb}(n+1)$. Assume first that~$n=2k$. Then
\[
	|U_{alt}|=\frac{n}{2} 
	\qquad\text{and}\qquad
	(|U_{alt}|+2)2^{n-1}=(n+4)2^{n-2}.
\]
For $0\leq r \leq k-1$ there are exactly two distinct commutators in~$\Theta_{alt}$ at odd level $\ell_t = 2r+1$ and no commutator at even 
level $\ell_t=2r+2$, see Figure~\ref{fig:trapeze_A4}. For $t\in \Theta_{alt}$ to the left of $\mathcal K_{alt}$ with $\ell_t = 2r+1$, any 
greedy ordinate monotone path~$\mathfrak p$ with last trapeze commutator~$t$ contains a greedy ordinate monotone path~$\tilde{\mathfrak p}$ 
from~$t$ to a commutator of $\mathcal K_{alt}$ at level~$n$ that uses only commutators in $\mathcal K_{alt}$. The path~$\tilde{\mathfrak p}$ 
traverses $n - \ell_t$ commutators, where the first one is determined since~$\tilde{\mathfrak p}$ must take an ``east'' step after~$t$. Moreover, 
at any position~$\tilde{\mathfrak p}$ must have taken strictly more ``east'' steps than ``west'' steps. The number of such paths~$\tilde{\mathfrak p}$ 
is $\binom{2(k-r-1)}{k-r-1}$, see \cite[Corollary~6]{nilsson_enumeration_2012} for details. 

A similar argument applies if $t\in \Theta_{alt}$ with $\ell_t = 2r+1$ is located the right of $\mathcal K_{alt}$, so 
$\gamma_{t} = \binom{2(k-r-1)}{k-r-1}=\binom{n-\ell_t-1}{(n-\ell_t-1)/2}$ for all $t\in \Theta_{alt}$ relates to the sequence of central binomial 
coefficients~{\tt{A000984}} of~\cite{oeis}. Setting $p:=k-r-1$ and using Gosper's algorithm to obtain a closed form for the hypergeometric sum 
\cite[Chapter~5]{petkovsek_ab_1996}, we conclude
\[
	\sum_{t\in\Theta_{alt}}
		\gamma_t2^{\ell_t-1}
	=
	2\sum_{r=0}^{k-1}
		\binom{2(k-r-1)}{k-r-1}
		2^{2r}
	=2^{n-1}
	\sum_{p=0}^{k-1}
		\binom{2p}{p}
		2^{-2p}
	=
	k\binom{2k}{k}
	=
	\frac{n}{2}\binom{n}{\frac{n}{2}}.
\]
If we now assume $n=2k-1$ then
\[
	|U_{alt}|=\frac{n-1}{2}
	\qquad\text{and}\qquad
	(|U_{alt}|+2)2^{n-1}=(n+3)2^{n-2}.
\]
For $0\leq r \leq k-2$ there is precisely one commutator in~$\Theta_{alt}$ at even level $\ell_t = 2r+2$. Since $n-\ell_t$ is odd, we get
$\gamma_t=\binom{n-\ell_t-1}{(n-\ell_t-1)/2}=\binom{2(k-r-2)}{k-r-2}$. Further there is precisely one commutator in~$\Theta_{alt}$ at odd 
level $\ell_t = 2r+1$, a similar argument yields $\gamma_t=\binom{2(k-r-2)+1}{(k-r-2)+1}$. Thus 
$\gamma_t=\binom{n-\ell_t-1}{\lfloor\frac{n-\ell_t-1}{2}\rfloor}$ relates to sequence~{\tt{A001405}} of~\cite{oeis}. Setting $p:=k-r-2$ and 
again using Gosper's algorithm we conclude 
\[
	\sum_{t\in\Theta_{alt}}
		\gamma_t2^{\ell_t-1}
	=
	2^{n-3}\sum_{p=0}^{k-2}
			\frac{(4p+3)}{p+1}\binom{2p}{p}2^{-2p}\\
	=
	-2^{n-2}+\frac{2n-1}{2}\binom{n-1}{\frac{n-1}{2}}.
\]
We now set $n=m-1$ and this proves the claim ${\mathsf S}_{alt}=\mathsf{fb}(n+1)$. This provides a new proof of Formula~\eqref{formula:fishburn}.

\section{Enumeration of $c$-singletons -- General Case}\label{sec:general}

In order to provide formulae to enumerate singletons in the general case of arbitrary finite Coxeter systems, we first generalize the poset 
$(\mathcal S, \prec_c)$ to general type, and present a planar embedding of its Hasse diagram in Section~\ref{subsec:nat_part_order}. In 
Section~\ref{ssec:equivalence_classes}, we describe two equivalence relations on words and class representatives indexed by Coxeter elements. 
In Section~\ref{subsec:2covers_cut_paths}, we present a graph called the \emph{2-cover} that we embed on a cylinder. In Section~\ref{ssec:cutpaths}, 
we count \emph{cut paths} in this embedding and provide a correspondence to Coxeter elements. In Section~\ref{ssec:crossingcutpaths}, we define 
when two cut paths are \emph{crossing}. Finally, in Section~\ref{subsec:enumerate_c_singletons}, starting with a Coxeter element $c$, we obtain 
a formula for the cardinality~$\mathsf S_c$ of a Cambrian acyclic domain by counting certain cut path that do not cross the cut path corresponding 
to~$c$.

Consider an \Dfn{irreducible finite Coxeter system}~$(W,S)$ of rank~$n$ with  generators~$s_1, \ldots , s_n$ and length function~$\ell$. A 
\Dfn{Coxeter element}~$c\in W$ is the product of~$n$ distinct generators of~$S$ in some order and $\mathsf{Cox}(W,S)$ is the set of all Coxeter 
elements of~$(W,S)$. The \Dfn{Coxeter number}~$h$ is the smallest positive integer such that $c^h$ is the identity of~$W$ and is independent of 
the choice of~$c$. As proposed by Shi~\cite{shi_enumeration_1997}, we identify Coxeter elements $c\in W$ and \emph{orientations $\Gamma_c$ of 
the Coxeter graph}~$\Gamma$ associated to~$W$: an edge $\{ s,t\}$ of~$\Gamma$ is directed from~$s$ to~$t$ if and only if~$s,t\in S$ do not commute 
and~$s$ comes before~$t$ in (any reduced expression of)~$c$. A \Dfn{word}~$\wword$ in~$S$ is a concatenation $\sigma_1\ldots\sigma_k$ for some 
nonnegative integer~$k$ and $\sigma_i\in S$, a \Dfn{subword} of $\wword=\sigma_1\ldots \sigma_k$ is a word $\sigma_{i_1}\ldots \sigma_{i_r}$ 
with $1\leq i_1 < \ldots < i_{r} \leq k$ and the \Dfn{support}~$\operatorname{supp}(\wword)$ of~$\wword$ is the set of generators that appear 
in~$\wword$. Of particular interest is the unique element~$\wo\in W$ of maximum length $\ell(\wo) = N := \tfrac{nh}{2}$ which is called 
\Dfn{longest element}. A reduced expression $\woword=\sigma_1\ldots\sigma_N$ is called \Dfn{longest word}.

\subsection{Natural partial order}\label{subsec:nat_part_order}

A \Dfn{Coxeter triple} $(W,S,\wword)$ is an irreducible finite Coxeter system~$(W,S)$ together with a word~$\wword$ in~$S$ such that 
$\operatorname{supp}(\wword)=S$. Any Coxeter triple~$(W,S,\wword)$ induces a unique reduced expression~$\c_{\wword}$ of a Coxeter 
element~$c_{\wword}$ where the elements of~$S$ appear according to their first appearance in~$\wword$. In particular, $\wword$ induces 
a canonical orientation on~$\Gamma$. We first define the natural partial order~$\prec_{\wword}$ on the disjoint union 
$\mathcal L_\wword := \{\sigma_1, \ldots, \sigma_k\}$ of letters of the word~$\wword=\sigma_1\ldots\sigma_k$. The map 
$\mathsf{g}:\mathcal L_\wword \rightarrow S$ assigns to each letter~$\sigma_i$ at position~$i$ the corresponding generator. The natural 
partial order appeared in \cite[Section~2]{viennot_1986}, \cite[Section~2.2]{stembridge_1996}, \cite[Definition~6]{galambos_acyclic_2008}.

\begin{definition}[Natural partial order]\label{def:infinite_natural_partial_order}
	\hfill\break
	The \Dfn{natural partial order}~$\prec_{\wword}$ on~$\mathcal L_\wword$ is defined for any Coxeter triple~$(W,S,\wword)$
	as follows: $\sigma_r \prec_{\wword} \sigma_s$ if and only if there is a subword $\sigma_{i_1}\ldots \sigma_{i_k}$
	of~$\wword$ such that $\sigma_{i_1}=\sigma_r$, $\sigma_{i_k}=\sigma_s$ and the elements $\mathsf{g}(\sigma_{i_j})$ 
	and $\mathsf{g}(\sigma_{i_{j+1}})$ do not commute for all $1\leq j\leq k-1$. The Hasse diagram 
	of~$(\mathcal L_\wword,\prec_\wword)$ is an oriented graph denoted by~$\mathcal G_\wword$.
\end{definition}

\begin{example}\label{expl:hasse_graph}
	Let~$(\Sigma_5,S,\wword)$ be the Coxeter triple with generators $s_i = (i\ \ i+1)$ and 
	\[
		\wword := s_3s_2s_1s_2s_3s_4s_2s_3s_2s_1s_2s_3s_4s_3s_2s_1s_3s_2s_3s_4.
	\]
	The induced reduced word for the Coxeter element~$c_\wword$ is $\c_\wword=s_3s_2s_1s_4$ and
	we have $|\mathcal L_\wword|=20$. Moreover, $(\mathcal L_\wword,\prec_\wword)$ consists of the following 
	$21$~covering relations:
	\begin{align*}
		\sigma_1 \precdot_\wword \sigma_2,		& &	\sigma_2 \precdot_\wword \sigma_3,		& &	\sigma_3 \precdot_\wword \sigma_4,		& &	\sigma_4 \precdot_\wword \sigma_5,		& &	\sigma_5    \precdot_\wword \sigma_6,	& &	\sigma_5    \precdot_\wword \sigma_7,	& &	\sigma_6 \precdot_\wword \sigma_8, \\
		\sigma_7 \precdot_\wword \sigma_8,		& &	\sigma_8 \precdot_\wword \sigma_9,		& &	\sigma_9 \precdot_\wword \sigma_{10},	& &	\sigma_{10} \precdot_\wword \sigma_{11},	& &	\sigma_{11} \precdot_\wword \sigma_{12},	& &	\sigma_{12} \precdot_\wword \sigma_{13},	& & \sigma_{13} \precdot_\wword \sigma_{14}, \\
		\sigma_{14} \precdot_\wword \sigma_{15},	& &	\sigma_{15} \precdot_\wword \sigma_{16},	& &	\sigma_{15}	\precdot_\wword \sigma_{17},	& &	\sigma_{16}	\precdot_\wword \sigma_{18},	& & \sigma_{17} \precdot_\wword \sigma_{18},	& &	\sigma_{18} \precdot_\wword \sigma_{19},	& &	\sigma_{19} \precdot_\wword \sigma_{20}.
	\end{align*}
\end{example}

\begin{example}
	Let $(W,S,\wword_1)$ and $(W,S,\wword_2)$ be Coxeter triples such that $\wword_1\neq\wword_2$, but $\wword_1$ is obtained from $\wword_2$ 
	by a sequence of braid relations of length~$2$ (and no deletions). Then $\c_{\wword_1}\neq\c_{\wword_2}$ are reduced expressions for the 
	same Coxeter element~$c\in W$ and the posets $(\mathcal L_{\wword_1},\prec_{\wword_1})$ and $(\mathcal L_{\wword_2},\prec_{\wword_2})$ are 
	isomorphic.
\end{example}

The next result gives a crossing-free straight-line planar embedding of the Hasse diagram~$\mathcal G_\wword$ for~$(W,S,\wword)$. 

\begin{proposition}\label{prop:drawing}
	Let~$k$ be a positive integer and~$(W,S,\c)$ be a Coxeter triple where~$\c$ is a reduced expression for $c\in\mathsf{Cox}(W,S)$. The graph 
	$\mathcal G_{\c^k}$ is connected, planar, and has a crossing-free straight-line planar embedding using integer vertex coordinates such that 
	the $x$-coordinate is strictly increasing in direction of every oriented edge.
\end{proposition}

\begin{proof}
	If $k=1$ then $\mathcal G_{\c}$ is isomorphic (as oriented graph) to the Coxeter graph~$\Gamma$ oriented
	according to~$\c$. By the classification of finite Coxeter groups, $\Gamma$ is a tree which is connected and 
	planar. Now label the vertices of~$\Gamma$ such that $s_1,\ldots,s_p$ are successive vertices of~$\Gamma$ 
	along a path of maximum length. We have $p=n-1$ if $(W,S)$ is of type~$D_n$, $E_6$, $E_7$ or~$E_8$ and
	$p=n$ otherwise. If $p=n-1$ we label the path such that the remaining vertex $s_n$ is connected to~$s_r$ 
	where $r=n-2$ (type~$D_n$) or $r=n-3$ (otherwise). To obtain the claimed drawing, locate $s_1$ at $(0,0)$ 
	and determine coordinates~$(x_j,y_j)$ for~$s_j$ with $j\leq p$ inductively from $(x_{j-1},y_{j-1})$ 
	for~$s_{j-1}$ via $x_j:=x_{j-1}\pm 1$ and $y_j:=y_{j-1}+1$ where the sign depends on the orientation 
	of~$\{ s_{j-1},s_j\}$. If $p=n-1$ then coordinates for the remaining point~$s_n$ are $x_n:=x_r\pm 1$ 
	and $y_n:=y_r$.
	
	\medskip
	\noindent
	We now inductively construct a planar drawing of~$\mathcal G_{\c^{k+1}}$ from a drawing of~$\mathcal G_{\c^k}$. 
	Observe that
	\begin{compactenum}[i)]
		\item $\mathcal L_{\c^k}$ and $\mathcal L_{\c^{k+1}}$ differ by a copy of~$S$. Denote the vertices 
			of~$\mathcal L_{\c^{k+1}}\setminus\mathcal L_{\c^k}$ by $s_i^{k+1}$ for $1\leq i \leq n$.
		\item The new covering relations of~$(\mathcal L_{\c^{k+1}},\prec_{\c^{k+1}})$ are of two types:
			\begin{compactitem}
				\item $s_j^{k+1}\precdot_{\c^{k+1}}s_{j'}^{k+1}$ if and only if $s_{j'}\rightarrow s_j$ in~$\Gamma$;
				\item $s_j^{k+1}\precdot_{\c^{k+1}}s_{j'}^{k}$ if and only if $s_j\rightarrow s_j$ in~$\Gamma$.
			\end{compactitem}
	\end{compactenum}
	Now set $x_j^{k+1}:=x_j^k+2$ and $y_j^{k+1}=y_j^k$ to obtain valid coordinates for~$s_j^{k+1}$ and include 
	oriented edges according to the covering relation of~$\prec_{\c^{k+1}}$.
\end{proof}
  
\begin{definition}[Tiles and their boundary]
	Let $(W,S,\wword)$ be a Coxeter triple.
	An (open) \Dfn{tile} of~$\mathcal G_\wword$ is a bounded connected component of~$\R^2\setminus \mathcal G_{\wword}$. 
	The \Dfn{boundary} of the closure of~$T$ is denoted by~$\partial T$.
\end{definition}

To simplify notation, the closure of a tile~$T$ is also denoted by~$T$.

\begin{corollary}\label{cor:Gw_in_Gck}
	Let~$T$ be a tile of~$\mathcal G_{\c^k}$ for a Coxeter triple $(W,S,\c^k)$. The boundary~$\partial T$ defines an 
	induced subgraph of~$\mathcal G_{\c^k}$ with four vertices; one vertex is a~source of out-degree~$2$ and one vertex
	is a~sink of in-degree~$2$. In particular, the source and sink of this subgraph are letters of~$\c^k$ of consecutive 
	copies of~$\c$ that represent the same generator of~$S$. 
\end{corollary}

\subsection{Equivalence classes and $\c$-sorting words}\label{ssec:equivalence_classes}

As in \cite{stembridge_1996}, we now define the equivalence relations $\sameelement$ and $\equptocomm$ on words in~$S$ as well as representatives 
for the equivalence classes $[\wword]_\sameelement$ and $[\wword]_\equptocomm$ that are determined by a reduced expression~$\c$ 
for $c\in\mathsf{Cox}(W,S)$.

First, we write $\uword\sameelement\vword$ if and only if $\uword,\vword$ are reduced words that represent the same element~$w\in~W$.
The equivalence class~$[\uword]_\sameelement$ depends only on $w$, so we often write $[w]_\sameelement$ instead of $[\wword]_\sameelement$.
Following Reading~\cite{reading_clusters_2007}, we define the $\c$-sorting word~$\wc$ of~$w$ as the lexicographically first subword of the infinite word $\c^\infty=\c\c\c\ldots$ (as a 
sequence of positions) which belongs to~$[w]_\sameelement$.

Second, $\uword\equptocomm\vword$ if and only if $\uword,\vword$ are  words that \Dfn{coincide up to commutations}, 
that is, one is obtained from the other by a sequence of braid relations of length~$2$ (and no deletions).
The \Dfn{$\c$-sorting word}~$\wword^\c$ of~$\wword$ is defined
as the element of~$[\wword]_\equptocomm$ that appears first lexicographically as a subword of the infinite word $\c^\infty=\c\c\c\ldots$ 
(as a sequence of positions).

Due to the similar definition,~$\wc$ and~$\wword^\c$ are both called $\c$-sorting word. We emphasize that~$\wc$ 
represents $[w]_\sameelement$ while $\wword^\c$ represents $[\wword]_\equptocomm$ and, by definition,  
$\ell(\uword)=\ell(\vword)$ if $\uword\equptocomm\vword$ but $\uword$ and $\vword$ are not necessarily reduced. 
Although the definition of~$\wc$ depends on a reduced expression~$\c$ for~$c$, we have $\pmb{w}^{\c_1}\equptocomm \pmb{w}^{\c_2}$
if $\c_1 \equptocomm \c_2$. 

 \begin{example}
	The words $\uword=s_1s_2s_1$ and $\vword=s_2s_1s_2$ are reduced words for the longest element $w_\circ\in\Sigma_3$
	of the Coxeter system~$(\Sigma_3,S)$ with generators $s_i = (i\ \ i+1)$ for $i\in\{1,2\}$. Thus $\uword\sameelement \vword$.
	As both words do not coincide up to commutations, we have $\uword\not\equptocomm\vword$. More precisely, we have
	\[
		\bwoc=s_1s_2|s_1,\qquad
		\uword^\c=s_1s_2|s_1\qquad\text{and}\qquad
		\vword^\c=s_2|s_1s_2
	\]
	if $\c=s_1s_2$ and
	\[
		\bwoc=s_2s_1|s_2,\qquad
		\uword^\c=s_1|s_2s_1,\qquad\text{and}\qquad 
		\vword^\c=s_2s_1|s_2,
	\]
	if $\c=s_2s_1$. We write~$|$ to distinguish between copies of~$\c$ in~$\c^\infty$.
\end{example}

\begin{example}[Example~\ref{expl:hasse_graph} continued]
	\hfill\break
	The $\c_\wword$-sorting word of~$\wword$ is
	$\wword^{\c_\wword}=s_3s_2s_1|s_2|s_3s_2s_4|s_3s_2s_1|s_2|s_3s_4|s_3s_2s_1|s_3s_2|s_3s_4$.
\end{example}

\begin{lemma}\label{lem:drawing_Gw}
	Let~$(W,S,\wword)$ be a Coxeter triple. The oriented graph~$\mathcal G_\wword$ is an induced oriented subgraph 
	of~$\mathcal G_{\c_\wword^m}$ for some positive integer~$m$.
\end{lemma}

\begin{proof}
	Let~$\mathbf {\widetilde w}$ be the subword of $\c_\wword\c_\wword\c_\wword\ldots$ that is lexicographically first
	(as a sequence of positions) among all subword of $\c_\wword\c_\wword\c_\wword\ldots$ that coincide with~$\wword$ 
	up to commutations and let~$m$ be the minimum integer such that~$\mathbf {\widetilde w}$ is a subword 
	of~$\c_\wword^m=(\c_\wword)^m$. Then $(\mathcal L_{\mathbf {\widetilde w}},\prec_{\mathbf {\widetilde w}})$ and 
	$(\mathcal L_{\wword},\prec_{\wword})$ are isomorphic, so their Hasse diagrams coincide and $\mathcal G_\wword$ 
	is an induced subgraph of~$\mathcal G_{\c_\wword^m}$. 
\end{proof}

We often write~$\mathcal G_{\wword}$ for the graph~$\mathcal G_{\wword}$ embedded according to Lemma~\ref{lem:drawing_Gw}.

\begin{example}[Example~\ref{expl:hasse_graph} continued]
	\hfill\break
	As $\wword$ is a subword of~$\c_\wword^9$, we obtain a planar drawing of~$\mathcal G_\wword$ induced 
	from the planar drawing of~$\mathcal G_{\c_{\wword}^9}$ as shown in Figure~\ref{fig:plane_drawing}.
\end{example}

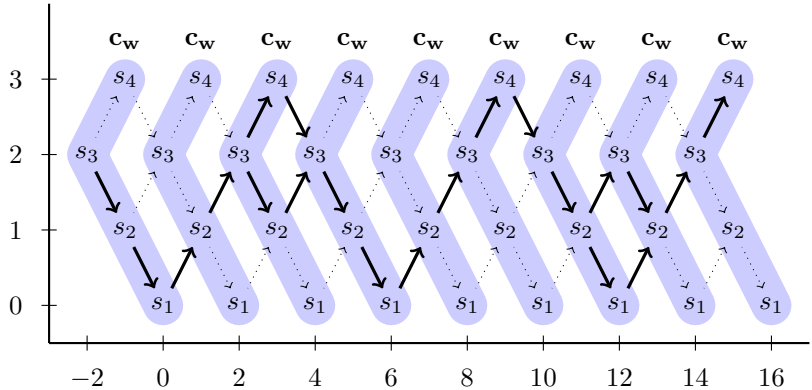
\begin{figure}[!ht]
		\begin{tikzpicture}[spec/.style={gray},
                    epais/.style={very thick,<-},
                    mince/.style={<-,dotted},
                    copies/.style={line width=15pt,color=blue,draw opacity=0.20,line cap=round,line join=round}]
	
	\draw[copies] (2,0) -- (1,2) -- (1.5,3);
	\draw[copies] (3,0) -- (2,2) -- (2.5,3);
	\draw[copies] (4,0) -- (3,2) -- (3.5,3);
	\draw[copies] (5,0) -- (4,2) -- (4.5,3);
	\draw[copies] (6,0) -- (5,2) -- (5.5,3);
	\draw[copies] (7,0) -- (6,2) -- (6.5,3);
	\draw[copies] (8,0) -- (7,2) -- (7.5,3);
	\draw[copies] (9,0) -- (8,2) -- (8.5,3);
	\draw[copies] (10,0) -- (9,2) -- (9.5,3);
	
	\node (s3-0) at (1,2) {$s_3$};
	\node (s4-0) at (1.5,3) {$s_4$} edge[mince] (s3-0);
	
	\node (s2-1) at (1.5,1) {$s_2$} edge[epais] (s3-0);
	\node (s3-1) at (2,2) {$s_3$} edge[mince] (s2-1) edge[mince] (s4-0);
	\node (s4-1) at (2.5,3) {$s_4$} edge[mince] (s3-1);
	
	\node (s1-2) at (2,0) {$s_1$} edge[epais] (s2-1);
	\node (s2-2) at (2.5,1) {$s_2$} edge[epais] (s1-2) edge[mince] (s3-1);
	\node (s3-2) at (3,2) {$s_3$} edge[epais] (s2-2) edge[mince] (s4-1);
	\node (s4-2) at (3.5,3) {$s_4$} edge[epais] (s3-2);
		
	\node (s1-3) at (3,0) {$s_1$} edge[mince] (s2-2);
	\node (s2-3) at (3.5,1) {$s_2$} edge[mince] (s1-3) edge[epais] (s3-2);
	\node (s3-3) at (4,2) {$s_3$} edge[epais] (s2-3) edge[epais] (s4-2);
	\node (s4-3) at (4.5,3) {$s_4$} edge[mince] (s3-3);

	\node (s1-4) at (4,0) {$s_1$} edge[mince] (s2-3);
	\node (s2-4) at (4.5,1) {$s_2$} edge[mince] (s1-4) edge[epais] (s3-3);
	\node (s3-4) at (5,2) {$s_3$} edge[mince] (s2-4) edge[mince] (s4-3);
	\node (s4-4) at (5.5,3) {$s_4$} edge[mince] (s3-4);

	\node (s1-5) at (5,0) {$s_1$} edge[epais] (s2-4);
	\node (s2-5) at (5.5,1) {$s_2$} edge[epais] (s1-5) edge[mince] (s3-4);
	\node (s3-5) at (6,2) {$s_3$} edge[epais] (s2-5) edge[mince] (s4-4);
	\node (s4-5) at (6.5,3) {$s_4$} edge[epais] (s3-5);

	\node (s1-6) at (6,0) {$s_1$} edge[mince] (s2-5);
	\node (s2-6) at (6.5,1) {$s_2$} edge[mince] (s1-6) edge[mince] (s3-5);
	\node (s3-6) at (7,2) {$s_3$} edge[mince] (s2-6) edge[epais] (s4-5);
	\node (s4-6) at (7.5,3) {$s_4$} edge[mince] (s3-6);

	\node (s1-7) at (7,0) {$s_1$} edge[mince] (s2-6);
	\node (s2-7) at (7.5,1) {$s_2$} edge[mince] (s1-7) edge[epais] (s3-6);
	\node (s3-7) at (8,2) {$s_3$} edge[epais] (s2-7) edge[mince] (s4-6);
	\node (s4-7) at (8.5,3) {$s_4$} edge[mince] (s3-7);

	\node (s1-8) at (8,0) {$s_1$} edge[epais] (s2-7);
	\node (s2-8) at (8.5,1) {$s_2$} edge[epais] (s1-8) edge[epais] (s3-7);
	\node (s3-8) at (9,2) {$s_3$} edge[epais] (s2-8) edge[mince] (s4-7);
	\node (s4-8) at (9.5,3) {$s_4$} edge[epais] (s3-8);

	\node (s1-9) at (9,0) {$s_1$} edge[mince] (s2-8);
	\node (s2-9) at (9.5,1) {$s_2$} edge[mince] (s1-9) edge[mince] (s3-8);

	\node (s1-10) at (10,0) {$s_1$} edge[mince] (s2-9);

	\draw (0.6,0) -- (0.4,0) node[label=left:$0$] {};
	\draw (0.6,1) -- (0.4,1) node[label=left:$1$] {};
	\draw (0.6,2) -- (0.4,2) node[label=left:$2$] {};
	\draw (0.6,3) -- (0.4,3) node[label=left:$3$] {};
	
	\draw[thick] (0.5,-0.5) -- (0.5,4);
	\draw[thick] (0.5,-0.5) -- (10.5,-0.5);
	
	\draw (1,-0.4) -- (1,-0.6) node[label=below:$-2$] {};
	\draw (2,-0.4) -- (2,-0.6) node[label=below:$0$] {};
	\draw (3,-0.4) -- (3,-0.6) node[label=below:$2$] {};
	\draw (4,-0.4) -- (4,-0.6) node[label=below:$4$] {};
	\draw (5,-0.4) -- (5,-0.6) node[label=below:$6$] {};
	\draw (6,-0.4) -- (6,-0.6) node[label=below:$8$] {};
	\draw (7,-0.4) -- (7,-0.6) node[label=below:$10$] {};
	\draw (8,-0.4) -- (8,-0.6) node[label=below:$12$] {};
	\draw (9,-0.4) -- (9,-0.6) node[label=below:$14$] {};
	\draw (10,-0.4) -- (10,-0.6) node[label=below:$16$] {};	

	\node at (1.5,3.5) {$\c_\wword$};
	\node at (2.5,3.5) {$\c_\wword$};
	\node at (3.5,3.5) {$\c_\wword$};
	\node at (4.5,3.5) {$\c_\wword$};
	\node at (5.5,3.5) {$\c_\wword$};
	\node at (6.5,3.5) {$\c_\wword$};
	\node at (7.5,3.5) {$\c_\wword$};
	\node at (8.5,3.5) {$\c_\wword$};
	\node at (9.5,3.5) {$\c_\wword$};
	
\end{tikzpicture}
		\caption{The crossing-free straight-line embedding of $\mathcal G_{\c_\wword^9}$ described in 
			Proposition~\ref{prop:drawing} together with $\mathcal G_\wword$ as subgraph 
			according to Lemma~\ref{lem:drawing_Gw}.}
		\label{fig:plane_drawing} 
\end{figure}

\begin{remark}\label{rem:remarks_on_equivalences}\hfill
	\begin{compactenum}[a)]
		\item If~$\uword\equptocomm\vword$ then $(\mathcal L_\uword,\prec_\uword)$ 
			and~$(\mathcal L_{\vword},\prec_{\vword})$ are isomorphic posets and $\mathcal G_\uword$ 
			and~$\mathcal G_{\vword}$ are isomorphic directed graphs.
		\item Let $(W,S,\c)$ be a Coxeter triple of type~$A$ and $\wword\in[\bwoc]_\equptocomm$. 
			Then~$(\mathcal L_\wword,\prec_\wword)$ is isomorphic to~$(\mathcal S,\prec_c)$ described in 
			Section~\ref{sec:counting_c_singletons_type_A}.
		\item The graph~$\mathcal G_{\bwoc}$ is isomorphic to the Auslander--Reiten quiver associated to~$c\in\mathsf{Cox}(W,S)$ 
			and~$\mathcal G_{\c^k}$ is a finite truncation of the repetition quiver described by 
			Keller~\cite[Section~2.2]{keller_cluster_2010} for all positive integers~$k$.
	\end{compactenum}
\end{remark}

A word~$\sigma_1\ldots\sigma_r$ is a \Dfn{prefix up to commutations} of a word~$\mathbf w$ if and only if there is a word~$\mathbf w'\equptocomm\mathbf w$ such that the first~$r$ letters of~$\mathbf w'$ are $\sigma_1\ldots \sigma_r$. 
The following characterization of~\Dfn{$c$-singletons} serves as definition and does not depend on~$\c$ but on $c\in\mathsf{Cox}(W,S)$.

\begin{definition}[{$\c$-singletons \cite[Theorem~2.2]{hohlweg_permutahedra_2011}}]\label{defn:c_singleton}
	\hfill\break
	Let~$(W,S,\c)$ be a Coxeter triple. An element~$w\in W$ is a $\c$-singleton if and only if some reduced expression 
	of~$w$ is a prefix of $\pmb w_{\pmb \circ}^\c$ up to commutations. The number of $\c$-singletons is denoted by~$\mathsf S_\c$.
\end{definition}

\begin{definition}[Cambrian acyclic domains]\label{defn:cambrian_acyclic_domain}
	\hfill\break
	Let~$(W,S,\c)$ be a Coxeter triple. The set $\mathsf{Acyc}_\c$ of $c$-singletons is called \Dfn{Cambrian acyclic domain}
	and its cardinality is~$\mathsf S_\c$.
\end{definition}
 
The set of $\c$-singletons, endowed with the weak order inherited from~$(W,S)$, forms a distributive lattice~$L_\c$,  
\cite[Proposition~2.5]{hohlweg_permutahedra_2011}. Any distributive lattice~$L$ is isomorphic to the lattice of order ideals of 
a poset~$(P,\leq)$ which is unique up to isomorphism, \cite[Theorem~3.4.1]{stanley_enumerative_2012}. Before we show that 
$(\mathcal L_{\pmb w_{\pmb \circ}^\c}, \prec_{\pmb w_{\pmb \circ}^\c})$ is such a poset~$(P,\leq)$ for~$L_\c$, we recall 
that an \Dfn{order ideal} (or down-set or semi-ideal) of~$(P,\leq)$ is a subset $I\subseteq P$ such that $t\in I$ 
and $s\leq t$ implies $s\in I$ and that antichains of a finite poset~$P$ are in bijection with order ideals of~$P$ \cite[Section~3.1]{stanley_enumerative_2012}. A generator $s \in S$ is called \Dfn{initial} (resp. \Dfn{final}) in~$w\in W$ if and only 
if $\ell(sw) < \ell(w)$ (resp. $\ell(ws) < \ell(w)$). 

\begin{proposition}\label{lem:ideals}
	Let $(W,S,\c)$ be Coxeter triple. The lattice of order ideals of $(\mathcal L_{\bwoc},\prec_{\bwoc})$ 
	is isomorphic to the poset of $c$-singletons ordered by the weak order.
\end{proposition}

\begin{proof}
If $w\in W$ is a $\c$-singleton then~$w$ is represented by a prefix~$\sigma_{i_1}\ldots \sigma_{i_k}$ of $\bwoc=\sigma_1\ldots \sigma_N$ 
up to commutations and the set~$F_w\subseteq\{\sigma_{i_1},\ldots, \sigma_{i_k}\}$ of final letters for~$w$ is an 
antichain of~$(\mathcal L_\bwoc,\prec_\bwoc)$. Conversely, if~$\{a_{1},\ldots, a_{k}\}$ is an antichain 
of~$(\mathcal L_\bwoc,\prec_\bwoc)$ then let~$I_j$ be the order ideal of~$(\mathcal L_\bwoc,\prec_\bwoc)$ generated 
by~$a_j$ for $1\leq j\leq k$ and consider the order ideal~$I:=\bigcup_{j=1}^{k}I_j$ of~$(\mathcal L_\bwoc,\prec_\bwoc)$ 
together with some linear extension of~$\prec_\bwoc$ on~$I$. The product of elements of~$I$ with respect to this 
linear order yields a prefix up to commutations of $\bwoc = \sigma_1\ldots \sigma_N$ with final letters~$\{a_1,\ldots,a_k\}$.  
\end{proof}

\subsection{$2$-covers}\label{subsec:2covers_cut_paths}

The longest element~$\wo$ of~$(W,S)$ defines an automorphism 
$\psi:W\rightarrow W$ defined via $w\mapsto\wo^{-1}w\wo$ that preserves length and adapts to words~$\wword=\sigma_1\ldots\sigma_r$ 
via $\psi(\wword):=\psi(\sigma_1)\ldots\psi(\sigma_r)$. Let $\operatorname{rev}(\wword):=\sigma_r\ldots\sigma_1$ 
denote the \Dfn{reverse word} of~$\wword$ and let~$\bwoc\psi(\bwoc)$ denote the concatenation of~$\bwoc$ and~$\psi(\bwoc)$. 
By Remark~7.6 of \cite{ceballos_subword_2014},

\[
	\pmb{w}_{\pmb{\circ}}^{\psi(\c)}\equptocomm\operatorname{rev}\big(\pmb{w}_{\pmb{\circ}}^{\operatorname{rev}(\c)}\big)
	\qquad\text{and}\qquad
	\c^h\equptocomm\bwoc\pmb{w}_{\pmb{\circ}}^{\psi(\c)}=\bwoc\psi(\bwoc).
\]

In combination with Remark~\ref{rem:remarks_on_equivalences}, we obtain the next lemma.

\begin{lemma}\label{lem:isomorphic_hasse}
	Let~$(W,S,\c)$ be a Coxeter triple. 
	The graphs $\mathcal G_{\c^h}$ and $\mathcal G_{\bwoc\psi(\bwoc)}$ are isomorphic as oriented graphs. 
	In particular, $\mathcal G_{\bwoc\psi(\bwoc)}$ depends on~$c\in\mathsf{Cox}(W,S)$ but not on the reduced expression~$\c$.
\end{lemma}

\begin{example}[Example~\ref{expl:hasse_graph} continued]
	\hfill\break
	Consider the longest word~$\widetilde{\woword} := s_3s_2s_1s_2s_3s_4s_2s_3s_2s_1$ of $(\Sigma_5,S)$. 
	Then~$\wword$ is the concatenation~$\widetilde{\woword}\psi(\widetilde{\woword})$ and a direct computation yields 
	\[
		\widetilde{\wword}_{\pmb \circ}^{\c_\wword}=s_3s_2s_1|s_2|s_3s_2s_4|s_3s_2s_1
		\qquad\text{and}\qquad 
		\psi(\widetilde{\wword}^{\c_\wword})=s_2s_3s_4|s_3|s_2s_3s_1|s_2s_3s_4.
	\] 
	Hence, $\mathcal G_{\c^h}$ and $\mathcal G_{\widetilde{\wword}_{\pmb \circ}^{\c_\wword}\psi(\widetilde{\wword}_{\pmb \circ}^{\c_\wword})}$ 
	are not isomorphic as oriented graphs.
	On the other hand,
	\[ 
		\pmb{w}_{\pmb \circ}^{\c_\wword}=s_3s_2s_1s_4|s_3s_2s_1s_4|s_3s_4
		\qquad\text{and}\qquad
		\psi(\pmb{w}_{\pmb \circ}^{\c_\wword}) = s_2s_3s_4s_1|s_2s_3s_4s_1|s_2s_1
	\]
	show that $\c^5 \equptocomm \pmb{w}_{\pmb \circ}^{\c_\wword}\psi(\pmb{w}_{\pmb \circ}^{\c_\wword})$.
	Thus $\mathcal G_{\c^5}$ and $\mathcal G_{\pmb{w}_{\pmb \circ}^{\c_\wword}\psi(\pmb{w}_{\pmb \circ}^{\c_\wword})}$ are 
	isomorphic as oriented graphs.
\end{example}

\begin{remark}\label{rem:bi-infinite_graph}\hfill
	\begin{compactenum}[a)]
		\item The constructions of $(\mathcal L_\wword,\prec_\wword)$ and of~$\mathcal G_{\c^k}$ 
			extend to the infinite word~$\c^\infty$. Thus, there is a canonical `left-most' embedding 
			of~$\mathcal G_{\pmb{w}_{\pmb \circ}^\c\psi(\pmb{w}_{\pmb \circ}^\c)}$ inside~$\mathcal G_{\c^\infty}$ 
			by Lemma~\ref{lem:isomorphic_hasse} (using~$\mathcal G_{\c^h}$).
			Moreover, there is a natural projection $pr_{\c^h}:\mathcal L_{\c^\infty}\rightarrow\mathcal L_{\c^h}$ 
			that maps a letter of copy~$\c^{hk+i}$ in~$\c^\infty$ to the corresponding letter of copy~$\c^i$ in~$\c^h$ 
			for $0\leq i\leq h-1$ and non-negative integers~$k$.
		\item There are obvious `bi-infinite' versions $\mathcal L_{\c^\Z}$ of $\mathcal L_{\c^\infty}$ and $\mathcal G_{\c^\Z}$ 
			of $\mathcal G_{\c^\infty}$ with infinitely many copies of $\c$ in the two possible 
			directions and a projection $pr_{\c^h}:\mathcal L_{\c^\Z}\rightarrow\mathcal L_{\c^h}$ that we use in  
			Section~\ref{sec:bounds}.
	\end{compactenum}
\end{remark}

\begin{definition}[2-cover]\label{dfn:2Cover_CutPaths}
	\hfill\break
	Let~$(W,S,\c)$ be a Coxeter triple and $\pmb{w}_{\pmb \circ}^\c$ the $\c$-sorting word of~$\wo$. 
	\begin{compactenum}[a)]
		\item The \Dfn{$2$-cover}~$\TwoCov$ is a graph with 
			vertices~$\mathcal L_{\pmb{w}_{\pmb \circ}^\c\psi(\pmb{w}_{\pmb \circ}^\c)}$ 
			and directed edges induced by~$pr_{\c^h}$ from $\mathcal G_{\c^\infty}$.
		\item Let~$V_{\pmb{w}_{\pmb \circ}^\c}$ be the vertices of~$\TwoCov$ that correspond 
			to the first~$\frac{nh}{2}$ letters of~${\pmb w}_{\pmb \circ}^\c\psi(\pmb{w}_{\pmb \circ}^\c)$ 
			and~$V_{\psi(\pmb{w}_{\pmb \circ}^\c)}$ be the vertices of~$\TwoCov$ that correspond 
			to the last~$\frac{nh}{2}$ letters of~$\pmb{w}_{\pmb \circ}^\c\psi(\pmb{w}_{\pmb \circ}^\c)$. 
			Then~$\mathcal C_{\pmb{w}_{\pmb \circ}^\c}$ is the subgraph	of~$\TwoCov$ induced 
			by~$V_{\pmb{w}_{\pmb \circ}^\c}$ and~$\mathcal C_{\psi(\pmb{w}_{\pmb \circ}^\c)}$ 
			is the subgraph of~$\TwoCov$ induced by~$V_{\psi(\pmb{w}_{\pmb \circ}^\c)}$.
	\end{compactenum}
\end{definition}

The planar drawing for $\mathcal G_{\c^h}$ described in Proposition~\ref{prop:drawing} and the simple 
observation that $\mathcal G_{\c^h}$ is isomorphic to some induced subgraph of $\mathcal G_{\tc^\infty}$ 
for all $\c, \tc \in  \mathsf{Cox}(W,S)$ imply the following lemma.

\begin{lemma}\label{lem:drawing_twocov_and_iso}
	Let~$(W,S,\c)$ and~$(W,S,\tc)$ be Coxeter triples. 
	\begin{compactenum}[i)]
		\item The $2$-cover~$\TwoCov$ has a crossing-free drawing on the open cylinder~$S^1\times \R$. 
		\item The $2$-covers~$\TwoCov$ and~$\tTwoCov$ are isomorphic as directed graphs.
	\end{compactenum}
\end{lemma}

We refer to a particular embedding of $\TwoCov\subset S^1\times\R\subset \R^2\times\R$ induced by~$pr_{\c^h}$ which can be  
visualized as `wrapping the plane drawing of~$\mathcal G_{\c^k}$ around a cylinder'. Without loss of generality, we assume 
that the $y$-direction for the plane drawing of~$\mathcal G_{\c^k}$ is parallel to the $z$-axis of $\R^3=\R^2\times\R$ which
we assume to coincide with the axis of the cylinder $S^1\times\R\subset \R^3$. For this reason, we identify $y$-coordinates
in the plane drawing of~$\mathcal G_{\c^k}$ with third coordinates in $\R^3$. This embedding $\TwoCov\subset \R^3$ 
has the following obvious properties:
\begin{compactenum}[i)]
	\item copies of~$s_i$ have strictly smaller third coordinate than copies of~$s_j$ for $1\leq i < j \leq p$ (where~$p$ is defined in the proof of Proposition~\ref{prop:drawing});
	\item for fixed $i \in[n]$, the third coordinate of all copies of~$s_i$ coincide;
	\item if $p<n$ then the third coordinate of~$s_r$ and~$s_n$ coincide. 
\end{compactenum}
We say that copies of the generator~$s_1$ are located at the bottom of~$\TwoCov$ while the copies of the generator~$s_p$ 
are located at the top of~$\TwoCov$. The definition of a tile as well as the proof of Corollary~\ref{cor:Gw_in_Gck} extend 
to this embedding of~$\TwoCov$. 
  
\subsection{Cut paths}\label{ssec:cutpaths}

\begin{definition}[Cut paths, primary and secondary cut paths]\label{def:cut_paths}\hfill\break
	Let~$(W,S,\c)$ be a Coxeter triple and~$\TwoCov$ the associated $2$-cover.
  	\begin{compactenum}[i)]
  		\item A \Dfn{cut path}~$\kappa$ of~$\TwoCov$ is a set of edges of~$\TwoCov$ such that every directed 
			cycle of~$\TwoCov$ contains precisely one edge of~$\kappa$. The set of all cut paths of $\TwoCov$
			is denoted by~$\mathsf{CP}(\TwoCov)$.
		\item The \Dfn{primary cut path}~$\kappa_\c$ of~$\TwoCov$ is the cut path that consists of all edges 
			of~$\TwoCov$ that are not (projections of) edges of~$\mathcal G_{\bwoc\psi(\bwoc)}$. The 
			\Dfn{secondary cut path}~$\kappa_\c^*$ is the cut path that consists of all edges 
			of~$\TwoCov\setminus\kappa_c$ that are neither edges of~$\mathcal C_{\bwoc}$ 
			nor edges of~$\mathcal C_{\psi(\bwoc)}$.
  	\end{compactenum}
\end{definition}

\noindent
The primary and secondary cut paths are disjoint because $\operatorname{supp}(\woword^\c)=\operatorname{supp}(\psi(\woword^\c))=S$.  
Primary and secondary cut paths relate to cuts of a graph as $\kappa_\c\cup\kappa_\c^*$ partitions 
the $2$-cover~$\TwoCov$ into two connected components. Every cut path~$\kappa$ induces a sequence 
(`path') of tiles that cuts the $2$-cover~$\TwoCov$ from bottom to top: consider the set of tiles 
such that consecutive tiles have at least one common edge in~$\kappa$. 

\begin{example}\label{expl:cut_paths}
	The notions of Definition~\ref{def:cut_paths} are illustrated in Figure~\ref{fig:cut_paths} for $\c_{Lod}$ 
	and $\c_{alt}$ in $(\Sigma_5,S)$. The $2$-covers~$\TwoCov$ coincide in both situations and are shown in a 
	planar drawing where vertices~$\sigma\in\mathcal L_{\pmb{w}_{\pmb \circ}^\c\psi(\pmb{w}_{\pmb \circ}^\c)}$ 
	of~$\TwoCov$ are labeled by corresponding generators $\mathsf{g}(\sigma)\in S$ and oriented edges of~$\TwoCov$ 
	contained in a cut path are indicated by~$\rightsquigarrow$. The primary and secondary cut paths~$\kappa_i$ 
	and $\kappa_i^*$ for $i\in\{1,2\}$ are edges~$\rightsquigarrow$ intersected by a dashed red line.
\end{example}

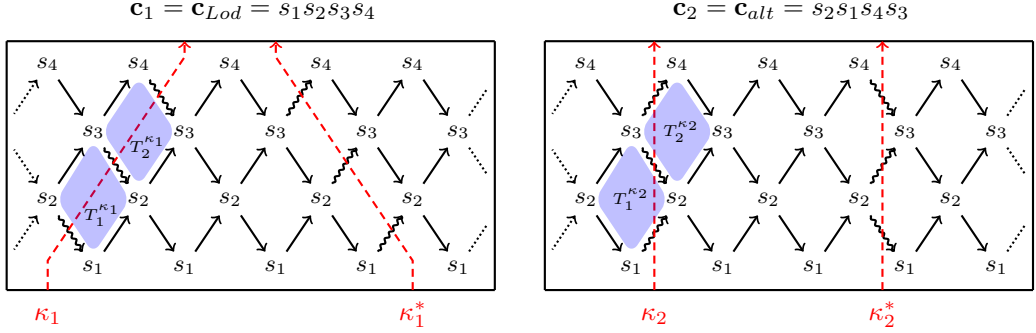
\begin{figure}[!htbp]
\begin{center}
	\begin{tabular}{cc}
	\begin{tikzpicture}[scale=0.6,decoration={snake,amplitude=.2mm,segment length=1mm},]

  \node (left1) at (4.125,0.1875) {};
  \node (left2a) at (4.125,3.1875) {};
  \node (left2b) at (4.125,2.8125) {};
  
  \draw[thick] (4.1,-0.5) -- (4.1,5);
  \draw[thick] (4.1,-0.5) -- (14.8,-0.5);
  \draw[thick] (4.1,5) -- (14.8,5);
  \draw[thick] (14.8,-0.5) -- (14.8,5);
  
%%%%%%%%%%%%%%%%%%%%%%%%%%

  \node (s2-1) at (5,1.5) {\small$s_2$} edge[<-,thick,densely dotted] (left1) edge[<-,thick,densely dotted] (left2b);
  \node (s4-1) at (5,4.5) {\small$s_4$} edge[<-,thick,densely dotted] (left2a);
  \node (s1-1) at (6,0) {\small$s_1$} edge[<-,thick,decorate] (s2-1);
  \node (s3-1) at (6,3) {\small$s_3$} edge[<-,thick] (s2-1)  edge[<-,thick] (s4-1);
    
  \node (s2-2) at (7,1.5) {\small$s_2$} edge[<-,thick,decorate] (s3-1)  edge[<-,thick] (s1-1);
  \node (s4-2) at (7,4.5) {\small$s_4$} edge[<-,thick] (s3-1);
  \node (s1-2) at (8,0) {\small$s_1$} edge[<-,thick] (s2-2);
  \node (s3-2) at (8,3) {\small$s_3$} edge[<-,thick] (s2-2)  edge[<-,thick,decorate] (s4-2);
  
  \node (s2-3) at (9,1.5) {\small$s_2$} edge[<-,thick] (s3-2)  edge[<-,thick] (s1-2);
  \node (s4-3) at (9,4.5) {\small$s_4$} edge[<-,thick] (s3-2);

%%%%%%%%%%%%%%%%%%%%%%%%%%

  \node (s1-3) at (10,0) {\small$s_1$} edge[<-,thick] (s2-3);
  \node (s3-3) at (10,3) {\small$s_3$} edge[<-,thick] (s4-3)  edge[<-,thick] (s2-3);
  \node (s2-4) at (11,1.5) {\small$s_2$} edge[<-,thick] (s1-3)  edge[<-,thick] (s3-3);
  \node (s4-4) at (11,4.5) {\small$s_4$} edge[<-,thick,decorate] (s3-3);

  \node (s1-4) at (12,0) {\small$s_1$} edge[<-,thick] (s2-4);
  \node (s3-4) at (12,3) {\small$s_3$} edge[<-,thick,decorate] (s2-4)  edge[<-,thick] (s4-4);
  \node (s2-5) at (13,1.5) {\small$s_2$} edge[<-,thick,decorate] (s1-4)  edge[<-,thick] (s3-4);
  \node (s4-5) at (13,4.5) {\small$s_4$} edge[<-,thick] (s3-4);
  
  \node (s1-5) at (14,0) {\small$s_1$} edge[<-,thick] (s2-5);
  \node (s3-5) at (14,3) {\small$s_3$} edge[<-,thick] (s2-5)  edge[<-,thick] (s4-5);

  \node (s2-6b) at (14.75,1.125) {} edge[thick,densely dotted] (s1-5);
  \node (s2-6a) at (14.75,1.875) {} edge[thick,densely dotted] (s3-5);
  \node (s4-6) at (14.75,4.125) {} edge[thick,densely dotted] (s3-5);

  \node[red,label=south:$\color{red}{\kappa_1}$] at (5,-0.5) {};
  \draw[red,thick,->,densely dashed] (5,-0.5) -- (5,0.15) -- (8,4.65) -> (8,5);
  \fill[blue,fill opacity=0.25,rounded corners] (5.2,1.5) -- (6,0.3) -- (6.8,1.5) -- (6,2.8) -- cycle;
  \fill[blue,fill opacity=0.25,rounded corners] (6.2,3) -- (7,1.8) -- (7.8,3) -- (7,4.2) -- cycle;
  \node (tile1) at (6.2,1.15) {\tiny{$T^{\kappa_1}_1$}};
  \node (tile1) at (7.2,2.65) {\tiny{$T^{\kappa_1}_2$}};

  \node[red,label=south:$\color{red}{\kappa^*_1}$] at (13,-0.35) {};
  \draw[red,thick,->,densely dashed] (13,-0.5) -- (13,0.15) -- (10,4.65) -> (10,5);
  \node[black,label=south:{$\c_1=\c_{Lod}=s_1s_2s_3s_4$}] at (9.5,6.25) {}; 
\end{tikzpicture} & \begin{tikzpicture}[scale=0.6,decoration={snake,amplitude=.2mm,segment length=1mm},]

  \node (left1) at (4.125,0.1875) {};
  \node (left2a) at (4.125,3.1875) {};
  \node (left2b) at (4.125,2.8125) {};
  
  \draw[thick] (4.1,-0.5) -- (4.1,5);
  \draw[thick] (4.1,-0.5) -- (14.8,-0.5);
  \draw[thick] (4.1,5) -- (14.8,5);
  \draw[thick] (14.8,-0.5) -- (14.8,5);
  
%%%%%%%%%%%%%%%%%%%%%%%%%%

  \node (s2-1) at (5,1.5) {\small$s_2$} edge[<-,thick,densely dotted] (left1) edge[<-,thick,densely dotted] (left2b);
  \node (s4-1) at (5,4.5) {\small$s_4$} edge[<-,thick,densely dotted] (left2a);
  \node (s1-1) at (6,0) {\small$s_1$} edge[<-,thick] (s2-1);
  \node (s3-1) at (6,3) {\small$s_3$} edge[<-,thick] (s2-1)  edge[<-,thick] (s4-1);
  
  \node (s2-2) at (7,1.5) {\small$s_2$} edge[<-,thick,decorate] (s3-1)  edge[<-,thick,decorate] (s1-1);
  \node (s4-2) at (7,4.5) {\small$s_4$} edge[<-,thick,decorate] (s3-1);
  \node (s1-2) at (8,0) {\small$s_1$} edge[<-,thick] (s2-2);
  \node (s3-2) at (8,3) {\small$s_3$} edge[<-,thick] (s2-2)  edge[<-,thick] (s4-2);
  
  \node (s2-3) at (9,1.5) {\small$s_2$} edge[<-,thick] (s3-2)  edge[<-,thick] (s1-2);
  \node (s4-3) at (9,4.5) {\small$s_4$} edge[<-,thick] (s3-2);

%%%%%%%%%%%%%%%%%%%%%%%%%%

  \node (s1-3) at (10,0) {\small$s_1$} edge[<-,thick] (s2-3);
  \node (s3-3) at (10,3) {\small$s_3$} edge[<-,thick] (s4-3)  edge[<-,thick] (s2-3);
  \node (s2-4) at (11,1.5) {\small$s_2$} edge[<-,thick] (s1-3)  edge[<-,thick] (s3-3);
  \node (s4-4) at (11,4.5) {\small$s_4$} edge[<-,thick] (s3-3);

  \node (s1-4) at (12,0) {\small$s_1$} edge[<-,thick,decorate] (s2-4);
  \node (s3-4) at (12,3) {\small$s_3$} edge[<-,thick,decorate] (s2-4)  edge[<-,thick,decorate] (s4-4);
  \node (s2-5) at (13,1.5) {\small$s_2$} edge[<-,thick] (s1-4)  edge[<-,thick] (s3-4);
  \node (s4-5) at (13,4.5) {\small$s_4$} edge[<-,thick] (s3-4);
  
  \node (s1-5) at (14,0) {\small$s_1$} edge[<-,thick] (s2-5);
  \node (s3-5) at (14,3) {\small$s_3$} edge[<-,thick] (s2-5)  edge[<-,thick] (s4-5);

  \node (s2-6b) at (14.75,1.125) {} edge[thick,densely dotted] (s1-5);
  \node (s2-6a) at (14.75,1.875) {} edge[thick,densely dotted] (s3-5);
  \node (s4-6) at (14.75,4.125) {} edge[thick,densely dotted] (s3-5);

  \node[red,label=south:$\color{red}{\kappa_2}$] at (6.5,-0.5) {};
  \draw[red,thick,->,densely dashed] (6.5,-0.5) -> (6.5,5);
  \fill[blue,fill opacity=0.25,rounded corners] (5.2,1.5) -- (6,0.3) -- (6.8,1.5) -- (6,2.8) -- cycle;
  \fill[blue,fill opacity=0.25,rounded corners] (6.2,3) -- (7,1.8) -- (7.8,3) -- (7,4.2) -- cycle;
  \node (tile1) at (6.0,1.5) {\tiny{$T^{\kappa_2}_1$}};
  \node (tile1) at (7.1,3.0) {\tiny{$T^{\kappa_2}_2$}};

  \node[red,label=south:$\color{red}{\kappa^*_2}$] at (11.5,-0.35) {};
  \draw[red,thick,->,densely dashed] (11.5,-0.5) -> (11.5,5);
  \node[black,label=south:{$\c_2=\c_{alt}=s_2s_1s_4s_3$}] at (9.5,6.25) {};
\end{tikzpicture}
	\end{tabular}
	\caption{\label{fig:cut_paths} The primary and secondary cut paths~$\kappa_i,\kappa_i^*$ depicted in a planar 
		drawing of the 2-cover~$\TwoCov$ for $\c_1=s_1s_2s_3s_4$ and $\c_2=s_2s_4s_1s_3$.}
\end{center}
\end{figure}

Since every cut path of~$\TwoCov$ that avoids edges of~$\mathcal C_{\psi(\bwoc)}$ (or, equivalently, every cut 
path of~$\TwoCov$ that avoids edges of the Hasse diagram of~$(\mathcal L_{\psi(\bwoc)},\prec_{\psi(\bwoc)})$) 
defines an antichain of~$(\mathcal L_\bwoc,\prec_\bwoc)$, we obtain the following characterization of order 
ideals of~$(\mathcal L_\bwoc,\prec_\bwoc)$.

\begin{lemma}\label{lem:c-singletons_as_cut_paths}
	Let $(W,S,\c)$ be a Coxeter triple. The set of $c$-singletons is in bijection with the set of cut paths 
	of~$\TwoCov$ that avoid edges of~$\mathcal C_{\psi({\pmb w_{\pmb \circ}^\c})}$.
\end{lemma}

In particular, if the number of all cut paths and the number of cut paths that contain edges 
of~$\mathcal C_{\psi({\pmb w_{\pmb \circ}^\c})}$ are known, Lemma~\ref{lem:c-singletons_as_cut_paths}
implies a formula for the cardinality~$\mathsf S_c$. A formula for $|\mathsf {CP}(\TwoCov)|$ is obtained from
the next theorem and a formula for~$\mathsf S_c$ will be derived in Section~\ref{ssec:crossingcutpaths}.

\begin{theorem}\label{thm:cut_path_and_tiles}
	Let~$(W,S,\c)$ be a Coxeter triple, with Coxeter graph $\Gamma$, and~$\kappa \in \mathsf{CP}(\TwoCov)$.
	\begin{compactenum}[i)]
		\item For each edge~$e=\{s_i,s_j\}$ of~$\Gamma$, there exists a unique directed
			edge~$e_\kappa=(\sigma,\tau)\in \kappa$ with $\{\mathsf g(\sigma), \mathsf g(\tau)\}=\{s_i,s_j\}$.
		\item Let~$s_j$ be a vertex of degree~$2$ of~$\Gamma$ with incident edges $e=\{s_i,s_j\}$, 
			$\widetilde e=\{s_j, s_k\}$. There is a unique tile~$T$ that contains the corresponding directed edges~$e_\kappa=(\sigma,\tau)$, 
			and $\widetilde e_\kappa=(\widetilde\sigma,\widetilde\tau)$ from~i).
		\item 
			Let~$s_{r}$ be a vertex of degree~$3$ of~$\Gamma$ with incident edges\\[1mm]
			\centerline{$
				e=\{s_{r-1},s_{r}\},\qquad 
				\widetilde e=\{s_{r},s_{r+1}\}\text{ and}\qquad 
				\overline e=\{s_{r},s_n\}
			$}\\[1mm]
			and corresponding directed edges of~$\kappa$ from~i)\\[1mm]
			\centerline{$
				e_\kappa=(\sigma,\tau),\qquad
				\widetilde e_\kappa=(\widetilde \sigma,\widetilde\tau)\text{ and}\qquad 
				\overline e_\kappa=(\overline\sigma,\overline\tau).
			$}\\[1mm] 
			There are two unique tiles~$T_1,T_2$ such that $e_\kappa,\overline e_\kappa\in \partial T_1$,
			$\widetilde e_\kappa,\overline e_\kappa\in \partial T_2$ and~$\partial T_1\cap\partial T_2$
			consists of two edges of~$\TwoCov$.
	\end{compactenum}
\end{theorem}

\begin{proof}
	\hfill
	\begin{compactenum}[i)]
		\item \label{first_case}
			For every edge $\{s_i,s_j\}$ of the Coxeter graph~$\Gamma$ there exists a directed 
			cycle in~$\TwoCov$ that visits only vertices corresponding to~$s_i$ and~$s_j$. This implies for 
			every edge $\{s_i,s_j\}$ of~$\Gamma$ that a cut path~$\kappa$ must contain precisely one oriented 
			edge~$(\sigma,\tau)$ of~$\TwoCov$ such that $\{\mathsf{g}(\sigma),\mathsf{g}(\tau)\}=\{s_i,s_j\}$.
		\item \label{second_case} 
			By \ref{first_case}), there are unique oriented edges~$e_\kappa=(\sigma,\tau)$,
			$\widetilde e_\kappa=(\widetilde \sigma,\widetilde\tau) \in \kappa$ 
			with $\{\mathsf g(\sigma),\mathsf g(\tau)\}=\{s_i,s_j\}$ 
			and $\{\mathsf g(\widetilde \sigma),\mathsf g(\widetilde \tau)\}=\{s_j,s_k\}$.
			Suppose there is no tile~$T$ with~$e_\kappa,\widetilde e_\kappa \in \partial T$. Without loss 
			of generality, let~$T$ be the unique tile such that~$e_\kappa\in \partial T$ and the two other
			vertices of~$T$ correspond to the generators~$s_j$ and~$s_k$. Then consider the directed cycle 
			that only uses edges~$(\sigma,\tau)$ of~$\TwoCov$ with
			$\{\mathsf g(\sigma),\mathsf g(\tau)\}=\{s_i,s_j\}$ where the two edges of~$\partial T$ are 
			replaced by the other two edges of~$\partial T$. Clearly, no edge of this cycle is an edge 
			of~$\kappa$. This contradicts the assumption that~$\kappa$ is a cut path.
		\item The argument to prove \ref{second_case}) can be used to show that there are unique tiles~$T_1$ 
			and~$T_2$ such that $e_\kappa,\overline e_\kappa\in \partial T_1$ and 
			$\widetilde e_\kappa,\overline e_\kappa\in \partial T_2$. But $\partial T_1$ and $\partial T_2$ 
			clearly share a directed edge $(\alpha, \beta)$ of~$\TwoCov$ with 
			$\{\mathsf g(\alpha),\mathsf g(\beta)\}=\{s_{r},s_n\}$ that is distinct from $\overline e_\kappa$.
		\qedhere
	\end{compactenum}
\end{proof}

\begin{corollary}\label{cor:tiles_for_cut_path}
	Each cut path~$\kappa\in \mathsf{CP}(\TwoCov)$ determines a unique set of~$n-2$ tiles:\\[2mm]
	\centerline{$
		\mathsf{tile}(\kappa)=
		 \set{T_1,\ldots,T_k}
		 	 {\begin{smallmatrix}
				 \partial T_i\text{ contains}\\[0.7mm]
				 \text{$2$ edges of $\kappa$}
			 \end{smallmatrix}}.$}
\end{corollary}

We tacitly order the tiles $T_i\in\mathsf{tile}(\kappa)$ from bottom to top in 
$\TwoCov\subset S^1\times \R$: if~$z_i$ denotes the smallest third coordinate of all points 
in~$T_i$ then $1\leq i < j \leq n-2$ implies $z_i \leq z_j$. 

Lemma~\ref{lem:drawing_twocov_and_iso} states that~$\TwoCov$ is isomorphic to~$\tTwoCov$ for all $\c,\tc$, so 
it is impossible to recover $c\in\mathsf{Cox}(W,S)$ from~$\TwoCov$. As any cut path~$\kappa$ provides one 
oriented edge for every edge of~$\Gamma$, we have an induced Coxeter element $c_\kappa$, its reduced expressions
are uniquely determined up to commutations.

\begin{corollary}\label{cor:c_and kappa_c}
	Let $(W,S,\c)$ be a Coxeter triple, $\kappa_c$ be the associated primary cut path 
	and~$c_{\kappa_\c}$ be the Coxeter element obtained from $\kappa_\c$.
	For any reduced expression $\wword$ of $c_\kappa$, we have $\wword  \equptocomm \operatorname{rev}(\c)$.
\end{corollary}

\begin{proof}
	The edges of~$\TwoCov$ that are not projections of edges of $\mathcal G_{\bwoc\psi(\bwoc)}$ define Coxeter 
	element that correspond to the equivalence class~$[\operatorname{rev}(\c)]_\equptocomm$.
\end{proof}

\begin{example}[Example~\ref{expl:cut_paths} continued]
	\hfill\break
	Figure~\ref{fig:cut_paths} also illustrates Corollaries~\ref{cor:tiles_for_cut_path} and~\ref{cor:c_and kappa_c}.
	First, the set of tiles~$\mathsf{tile}(\kappa_1)$ and~$\mathsf{tile}(\kappa_2)$ associated to~$\kappa_1$ 
	and~$\kappa_2$ according to Corollary~\ref{cor:tiles_for_cut_path} are illustrated. 
	This example shows that the set of tiles can coincide even if $\kappa_1\neq\kappa_2$. 
	Moreover, the Coxeter element~$c_{\kappa_1}$ is represented by~$s_4s_3s_2s_1$ and $s_4s_3s_2s_1 \in [\operatorname{rev}(\c_{Lod})]_\equptocomm$.
	Similarly, $c_{\kappa_2}$ is represented by~$s_1s_3s_2s_4$ and $s_1s_3s_2s_4 \in [\operatorname{rev}(\c_{alt})]_\equptocomm$.
\end{example}

\begin{corollary}\label{cor:number_of_cut_paths}
	The map $\Phi:\mathsf{CP}(\TwoCov)\rightarrow\mathsf{Cox}(W,S)$ sending a cut path $\kappa$ to its 
	corresponding Coxeter element $c_\kappa$ is surjective and satisfies $|\Phi^{-1}(c)|=h$ for each $c\in\mathsf{Cox}(W,S)$. 
	In particular, $|\mathsf {CP}(\TwoCov)|=2^{n-1}h$.
\end{corollary}

\begin{proof}
	We only prove $|\Phi^{-1}(c)|=h$. There are $h$ choices in $\TwoCov$ to pick a 
	vertex~$\sigma_1$ with $\mathsf{g}(\sigma_1)=s_1$. Now~$\sigma_1$ determines a 
	unique tile $T_1$ with $\sigma_1\in\partial T_1$ and there is a unique directed 
	edge~$e_1^\kappa$ of~$\partial T_1$ that reflects the order of~$s_1$ and~$s_2$ 
	in~$c$. Now consider the unique tile~$T_2$ whose vertices map to $s_2$, $s_3$ 
	and $s_4$ under~$\mathsf g$ such that the orientation of $T_1\cap T_2$ reflects
	the order of~$s_2$ and~$s_3$ in~$c$ and proceed similarly with the following generators 
	until all generators have been considered. This process determines a unique 
	cut path~$\kappa$ with~$\Phi(\kappa)=c$ after choosing one of the~$h$ possible
	initial vertices~$\sigma$ at the bottom of $\TwoCov$.  
\end{proof}

\subsection{Crossings of cut paths}\label{ssec:crossingcutpaths}

\begin{definition}[crossing of cut paths, initial and final side, crossing tile]\label{def:crossing}
\hfill\break
Let~$(W,S,\c)$ be a Coxeter triple and~$\kappa_\c\in\mathsf{CP}(\TwoCov)$ be the associated primary cut path.
\begin{compactenum}[i)]
	\item A cut path $\kappa$ \Dfn{crosses} $\kappa_\c$ if 
		$\mathsf{tile}(\kappa)\cap\mathsf{tile}(\kappa_\c)\neq \varnothing$ and there are edges $e_1,e_2\in\kappa$ 
		with $e_1\in\mathcal C_{\wo^\c}$ and $e_2\in\mathcal C_{\psi(\wo^\c)}$.
	\item Let~$\kappa$ be a cut path that crosses~$\kappa_\c$. The \Dfn{initial side} of~$\kappa$ is the 
		connected component of $\TwoCov\setminus (\kappa_\c\cup \kappa^*_\c)$ that contains the edge of 
		$\kappa\setminus(\kappa_\c\cup\kappa_\c^*)$ whose midpoint has minimal third coordinate. The \Dfn{final side}
		of~$\kappa$ is the connected component of $\TwoCov\setminus (\kappa_\c\cup \kappa^*_\c)$ that is 
		not the initial side of~$\kappa$.
	\item Let~$\kappa$ be a cut path that crosses~$\kappa_\c$. The \Dfn{crossing tile} $T^{\kappa,\c}$ 
		of~$\kappa$ in~$\TwoCov$ is the first tile of $\mathsf{tile}(\kappa)$ (with respect to the 
		bottom-to-top order) that contains an edge of $\kappa$ in the final side of~$\kappa$.
\end{compactenum}
\end{definition}

\begin{definition}[Initial and final segments]\label{def:init_final_segment}
\hfill\break
Let~$(W,S,\c)$ be a Coxeter triple,~$\kappa_\c\in \mathsf{CP}(\TwoCov)$ be the associated primary cut path, 
$T^\c \in \mathsf{tile}(\kappa_\c)$ and~$\kappa\in\mathsf{CP}(\TwoCov)$ with $\mathsf{tile}(\kappa)=\{T_1,\ldots, T_{n-2}\}$.
\begin{compactenum}[i)]
	\item Let $i\in[n-1]$. The \Dfn{initial segment of~$\kappa$ up to $T_i$} is defined as 
		\[
			\set{e\in\kappa}{e\in\partial T_j \text{ for }j\in[i-1]^{\textcolor{white}{A\!\!\!}}}
			\cup
			\set{\{u,v\}\in\kappa}{\mathsf g(u)=s_1\text{ or }\mathsf g(v)=s_1^{\textcolor{white}{A\!\!\!}}}
		\]
		and the \Dfn{final segment of~$\kappa$ starting at~$T_i$} is defined as
		\[
			\set{e\in\kappa}{e\in\partial T_j \text{ with }j>i^{\textcolor{white}{A\!\!\!}}}
			\cup
			\set{\{u,v\}\in\kappa}{\mathsf g(u)=s_p\text{ or }\mathsf g(v)=s_p^{\textcolor{white}{A\!\!\!}}}
		\]
		where $s_1,\ldots,s_p$ are successive vertices of~$\Gamma$ along a path of maximum length.
	\item Let~$e_1 = (u_1,v_1)$ and $e_2 = (u_2,v_2)$ be the distinct edges of~$\partial T^\c \setminus \kappa_\c$ 
		such that the midpoint of~$e_1$ has smaller third coordinate than the midpoint of~$e_2$. The connected 
		component of $\TwoCov\setminus (\kappa_\c\cup \kappa^*_\c)$ that contains~$e_2$ is denoted by~$\mathsf{out}(T^\c)$ and the other component is denoted by~$\mathsf{in}(T^\c)$. 
	\item Let~$I(T^\c)$ be the number of distinct initial segments of cut paths~$\kappa$ up 
		to~$T^\c$ with edges contained in $\TwoCov\setminus \mathsf{out}(T^\c)$.
	\item Let $F(T^\c)$ be the number of distinct final segments of cut paths~$\kappa$ starting at~$T^\c$ that contain~$e_2$.
\end{compactenum}
\end{definition}

\noindent
In Definitions~\ref{def:crossing} and~\ref{def:init_final_segment}, the primary cut path $\kappa_\c$ can be replaced by 
any cut path $\kappa\in\mathsf{CP}(\TwoCov)$. Moreover, concatenation of an initial segment counting towards $I(T^\c)$ 
that differs from the initial segment of $\kappa_\c$ and a final segment counting towards~$F(T^\c)$ yields a cut path 
that crosses $\kappa_\c$.

\begin{example}[Example~\ref{expl:cut_paths} continued]\label{expl:I(T_i)}
	\hfill\break
	We illustrate Definitions~\ref{def:crossing} and~\ref{def:init_final_segment} in Figure~\ref{fig:universal}.  
	The cut path~$\widetilde \kappa_1$ crosses~$\kappa_1$ with crossing tile~$T^{\kappa_1}_2$ 
	while~$\widetilde \kappa_2$ does not cross~$\kappa_2$. A straightforward counting of crossing 
	cut paths verifies
	\[
		I(T^{\kappa_1}_1)=2
		\quad\text{and}\quad
		I(T^{\kappa_2}_2)=4
	\qquad\quad
	\text{as well as}
	\qquad\quad
		I(T^{\kappa_2}_1)=1
		\quad\text{and}\quad
		I(T^{\kappa_2}_2)=2.
	\]
	The reasoning of Section~\ref{sec:counting_in_type_a} yields formulae for all positive integers~$i$:
	\[
		I(T^{c_{Lod}}_{i}) = I(T^{\operatorname{rev}(c_{Lod})}_{i}) = 2^i
		\qquad\text{and}\qquad
		I(T^{c_{alt}}_{i}) 
		=  \begin{cases}
				\binom{2j}{j}, 				& \text{if $i=2j$,}\\[2mm]
				\frac{1}{2}\binom{2j}{j},	& \text{if $i=2j-1$.}
			\end{cases}
	\]
\end{example}

\begin{figure}[!htbp]
\begin{center}
	\centerline{\begin{tikzpicture}[scale=0.6,decoration={snake,amplitude=.2mm,segment length=1mm},]

  \node (left1) at (4.125,0.1875) {};
  \node (left2a) at (4.125,3.1875) {};
  \node (left2b) at (4.125,2.8125) {};
  
  \draw[thick] (4.1,-0.5) -- (4.1,5);
  \draw[thick] (4.1,-0.5) -- (14.8,-0.5);
  \draw[thick] (4.1,5) -- (14.8,5);
  \draw[thick] (14.8,-0.5) -- (14.8,5);
  
%%%%%%%%%%%%%%%%%%%%%%%%%%

  \node (s2-1) at (5,1.5) {\small$s_2$} edge[<-,thick,densely dotted] (left1) edge[<-,thick,densely dotted] (left2b);
  \node (s4-1) at (5,4.5) {\small$s_4$} edge[<-,thick,densely dotted] (left2a);
  \node (s1-1) at (6,0) {\small$s_1$} edge[<-,thick,decorate] (s2-1);
  \node (s3-1) at (6,3) {\small$s_3$} edge[<-,thick] (s2-1)  edge[<-,thick] (s4-1);
    
  \node (s2-2) at (7,1.5) {\small$s_2$} edge[<-,thick,decorate] (s3-1)  edge[<-,thick] (s1-1);
  \node (s4-2) at (7,4.5) {\small$s_4$} edge[<-,thick] (s3-1);
  \node (s1-2) at (8,0) {\small$s_1$} edge[<-,thick] (s2-2);
  \node (s3-2) at (8,3) {\small$s_3$} edge[<-,thick] (s2-2)  edge[<-,thick,decorate] (s4-2);
  
  \node (s2-3) at (9,1.5) {\small$s_2$} edge[<-,thick] (s3-2)  edge[<-,thick] (s1-2);
  \node (s4-3) at (9,4.5) {\small$s_4$} edge[<-,thick] (s3-2);

%%%%%%%%%%%%%%%%%%%%%%%%%%

  \node (s1-3) at (10,0) {\small$s_1$} edge[<-,thick] (s2-3);
  \node (s3-3) at (10,3) {\small$s_3$} edge[<-,thick] (s4-3)  edge[<-,thick] (s2-3);
  \node (s2-4) at (11,1.5) {\small$s_2$} edge[<-,thick] (s1-3)  edge[<-,thick] (s3-3);
  \node (s4-4) at (11,4.5) {\small$s_4$} edge[<-,thick,decorate] (s3-3);

  \node (s1-4) at (12,0) {\small$s_1$} edge[<-,thick] (s2-4);
  \node (s3-4) at (12,3) {\small$s_3$} edge[<-,thick,decorate] (s2-4)  edge[<-,thick] (s4-4);
  \node (s2-5) at (13,1.5) {\small$s_2$} edge[<-,thick,decorate] (s1-4)  edge[<-,thick] (s3-4);
  \node (s4-5) at (13,4.5) {\small$s_4$} edge[<-,thick] (s3-4);
  
  \node (s1-5) at (14,0) {\small$s_1$} edge[<-,thick] (s2-5);
  \node (s3-5) at (14,3) {\small$s_3$} edge[<-,thick] (s2-5)  edge[<-,thick] (s4-5);

  \node (s2-6b) at (14.75,1.125) {} edge[thick,densely dotted] (s1-5);
  \node (s2-6a) at (14.75,1.875) {} edge[thick,densely dotted] (s3-5);
  \node (s4-6) at (14.75,4.125) {} edge[thick,densely dotted] (s3-5);

  \node[red,label=south:$\color{red}{\kappa_1}$] at (5,-0.5) {};
  \draw[red,thick,->,densely dashed] (5,-0.5) -- (5,0.15) -- (8,4.65) -> (8,5);
  \fill[blue,fill opacity=0.25,rounded corners] (5.2,1.5) -- (6,0.3) -- (6.8,1.5) -- (6,2.8) -- cycle;
  \fill[blue,fill opacity=0.25,rounded corners] (6.2,3) -- (7,1.8) -- (7.8,3) -- (7,4.2) -- cycle;
  \node (tile1) at (6.2,1.15) {\tiny{$T^{\kappa_1}_1$}};
  \node (tile1) at (7.1,2.5) {\tiny{$T^{\kappa_1}_2$}};
  \node[red,label=south:$\color{red}{\widetilde \kappa_1}$] at (7,-0.33) {};
  \draw[red,thick,->,dashdotdotted] (7,-0.5) -- (7,0.15) -- (8,1.85) -- (6,4.85 )-> (6,5);

  \node[red,label=south:$\color{red}{\kappa^*_1}$] at (13,-0.35) {};
  \draw[red,thick,->,densely dashed] (13,-0.5) -- (13,0.15) -- (10,4.65) -> (10,5);
  
\end{tikzpicture}$\qquad$\begin{tikzpicture}[scale=0.6,decoration={snake,amplitude=.2mm,segment length=1mm},]

  \node (left1) at (4.125,0.1875) {};
  \node (left2a) at (4.125,3.1875) {};
  \node (left2b) at (4.125,2.8125) {};
  
  \draw[thick] (4.1,-0.5) -- (4.1,5);
  \draw[thick] (4.1,-0.5) -- (14.8,-0.5);
  \draw[thick] (4.1,5) -- (14.8,5);
  \draw[thick] (14.8,-0.5) -- (14.8,5);
  
%%%%%%%%%%%%%%%%%%%%%%%%%%

  \node (s2-1) at (5,1.5) {\small$s_2$} edge[<-,thick,densely dotted] (left1) edge[<-,thick,densely dotted] (left2b);
  \node (s4-1) at (5,4.5) {\small$s_4$} edge[<-,thick,densely dotted] (left2a);
  \node (s1-1) at (6,0) {\small$s_1$} edge[<-,thick] (s2-1);
  \node (s3-1) at (6,3) {\small$s_3$} edge[<-,thick] (s2-1)  edge[<-,thick] (s4-1);
  
  \node (s2-2) at (7,1.5) {\small$s_2$} edge[<-,thick,decorate] (s3-1)  edge[<-,thick,decorate] (s1-1);
  \node (s4-2) at (7,4.5) {\small$s_4$} edge[<-,thick,decorate] (s3-1);
  \node (s1-2) at (8,0) {\small$s_1$} edge[<-,thick] (s2-2);
  \node (s3-2) at (8,3) {\small$s_3$} edge[<-,thick] (s2-2)  edge[<-,thick] (s4-2);
  
  \node (s2-3) at (9,1.5) {\small$s_2$} edge[<-,thick] (s3-2)  edge[<-,thick] (s1-2);
  \node (s4-3) at (9,4.5) {\small$s_4$} edge[<-,thick] (s3-2);

%%%%%%%%%%%%%%%%%%%%%%%%%%

  \node (s1-3) at (10,0) {\small$s_1$} edge[<-,thick] (s2-3);
  \node (s3-3) at (10,3) {\small$s_3$} edge[<-,thick] (s4-3)  edge[<-,thick] (s2-3);
  \node (s2-4) at (11,1.5) {\small$s_2$} edge[<-,thick] (s1-3)  edge[<-,thick] (s3-3);
  \node (s4-4) at (11,4.5) {\small$s_4$} edge[<-,thick] (s3-3);

  \node (s1-4) at (12,0) {\small$s_1$} edge[<-,thick,decorate] (s2-4);
  \node (s3-4) at (12,3) {\small$s_3$} edge[<-,thick,decorate] (s2-4)  edge[<-,thick,decorate] (s4-4);
  \node (s2-5) at (13,1.5) {\small$s_2$} edge[<-,thick] (s1-4)  edge[<-,thick] (s3-4);
  \node (s4-5) at (13,4.5) {\small$s_4$} edge[<-,thick] (s3-4);
  
  \node (s1-5) at (14,0) {\small$s_1$} edge[<-,thick] (s2-5);
  \node (s3-5) at (14,3) {\small$s_3$} edge[<-,thick] (s2-5)  edge[<-,thick] (s4-5);

  \node (s2-6b) at (14.75,1.125) {} edge[thick,densely dotted] (s1-5);
  \node (s2-6a) at (14.75,1.875) {} edge[thick,densely dotted] (s3-5);
  \node (s4-6) at (14.75,4.125) {} edge[thick,densely dotted] (s3-5);

  \node[red,label=south:$\color{red}{\kappa_2}$] at (6.5,-0.5) {};
  \draw[red,thick,->,densely dashed] (6.5,-0.5) -> (6.5,5);
  \fill[blue,fill opacity=0.25,rounded corners] (5.2,1.5) -- (6,0.3) -- (6.8,1.5) -- (6,2.8) -- cycle;
  \fill[blue,fill opacity=0.25,rounded corners] (6.2,3) -- (7,1.8) -- (7.8,3) -- (7,4.2) -- cycle;
  \node (tile1) at (6.0,1.5) {\tiny{$T^{\kappa_2}_1$}};
  \node (tile1) at (7.1,3.0) {\tiny{$T^{\kappa_2}_2$}};
  \node[red,label=south:$\color{red}{\widetilde \kappa_2}$] at (5.5,-0.33) {};
  \draw[red,thick,->,dashdotdotted] (5.5,-0.5) -> (5.5,5);

  \node[red,label=south:$\color{red}{\kappa^*_2}$] at (11.5,-0.35) {};
  \draw[red,thick,->,densely dashed] (11.5,-0.5) -> (11.5,5);
  
\end{tikzpicture}}
	\caption{\label{fig:universal} Two planar drawings of the 2-cover for~$\Sigma_5$. The cut 
	path~$\widetilde \kappa_1$ crosses the primary cut path~$\kappa_1$ (left) and the cut 
	path~$\widetilde \kappa_2$ does not cross the primary cut path~$\kappa_2$ (right).}
\end{center}
\end{figure}
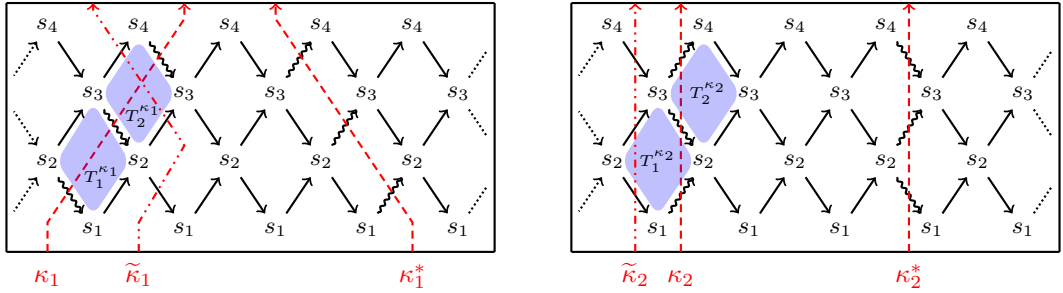

\subsection{Enumerating $c$-singletons}\label{subsec:enumerate_c_singletons}

\begin{theorem}\label{thm:singleton_cuts}
	Let $(W,S,\c)$ be Coxeter triple with associated primary cut path~$\kappa_\c$ and  set of tiles
	$\mathsf{tile}(\kappa_\c)=\{T^c_1,\ldots, T^c_{n-2}\}$. The cardinality of the Cambrian acyclic domain $\mathsf{Acyc}_\c$ is
	\[
		\mathsf S_c = 2^{n-2}(h+1)-\sum_{i\in[n-2]} 2^{n-2-i}I(T_i^c).
	\]
\end{theorem}

\begin{proof}
	We count the cut paths of $\mathsf {CP}(\TwoCov)$ twice. By Corollary~\ref{cor:number_of_cut_paths}, 
	the cardinality of~$\mathsf {CP}(\TwoCov)$ equals $2^{n-1}h$. 
	On the other hand, $\kappa \in \mathsf {CP}(\TwoCov)$ satisfies precisely one of the following statements:
	\begin{compactenum}[i)]
		\item $\kappa$ crosses~$\kappa_c$ or~$\kappa_c^*$, but not both;
		\item $\kappa\subseteq\mathcal C_{\wo^\c}\cup\kappa_c\cup\kappa_c^*$ 
			or $\kappa\subseteq\mathcal C_{\psi(\wo^\c)}\cup\kappa_c\cup\kappa_c^*$ but $\kappa\not\in\{\kappa_c,\kappa_c^*\}$;
		\item $\kappa\in\{\kappa_c,\kappa_c^*\}$.
	\end{compactenum}
	We first claim that the number~$\mathsf Q_c$ of cut paths that cross~$\kappa_c$ equals the number of cut paths 
	that cross~$\kappa_c^*$. 
	Indeed, the automorphism~$\psi$ maps~$c$ to~$\psi(c)$ and induces an involution~$\varphi$ between $\mathsf{tile}(\kappa_c)$ 
	and $\mathsf{tile}(\kappa_c^*)$ that extends to an involution between cut paths that cross~$\kappa_c$ and cut paths that 
	cross~$\kappa_c^*$ (where $\varphi(T^{\kappa, c})$ is the crossing tile of $\varphi(\kappa)$). 
	Thus, the number of cut paths satisfying~i) equals~$2\mathsf Q_c$. 
	Moreover, this involution extends to an involution on~$\mathsf{CP}(\TwoCov)$ that maps cut paths contained 
	in~$\mathcal C_{\woc}\cup\kappa_c\cup\kappa_c^*$ to cut paths contained in $\mathcal C_{\psi(\woc)}\cup\kappa_c\cup\kappa_c^*$. 
	As~$\kappa_c$ and~$\kappa_c^*$ are the only cut paths that correspond simultaneously to a $c$-singleton and a $\psi(c)$-singleton, 
	the number of cut paths satisfying~ii) equals $2\mathsf S_c-4$ by Lemma~\ref{lem:c-singletons_as_cut_paths}. 
	Thus we established

	\begin{equation}\label{eqn:double_counting}
		2\mathsf Q_c + (2\mathsf S_c - 4) + 2 = 2^{n-1}h
		\qquad\text{or equivalently}\qquad
		\mathsf S_c = 2^{n-1}h - \mathsf Q_c +1.
	\end{equation}
	
	To analyze $\mathsf Q_c$, we partition the cut paths that cross~$\kappa_c$ according to their crossing
	tile~$T_i^c$ and observe that $I(T_i^c)$ exceeds the number of cut paths with crossing tile~$T_i^c$ 
	by one (the initial segment of~$\kappa_c$ is counted by~$I(T_i^c)$ but it is not the initial segment of 
	a cut path with crossing tile~$T_i^c$). The number of final segments~$F(T_i)$ of cut paths starting 
	at~$T_i^c$ satisfies $F(T_i)=2^{n-2-i}$ because each final segment consists of $n-2-i$ tiles (not 
	counting~$T_i^c$) and there are~$2$ valid choices to exit each tile.
	This gives

	\[
		\mathsf Q_c 
			= \sum_{i\in[n-2]}\left(I(T_i^c)-1\right) F(T_i^c)
			= \sum_{i\in[n-2]}2^{n-2-i}I(T_i^c) - \left(2^{n-2}-1\right).
	\]

	Substitution in Equation~(\ref{eqn:double_counting}) yields the claim.
\end{proof}

At this point, we could characterize the Coxeter elements of any irreducible finite Coxeter system~$(W,S)$ that minimize 
the cardinality~$\mathsf S_\c$ of the Cambrian acyclic domain~$\mathsf{Acyc}_\c$. Instead, we first 
provide explicit formulae for~$\mathsf S_\c$ of two families of Cambrian acyclic domains in Section~\ref{sec:examples} 
and characterize the Coxeter elements of~$(W,S)$ that minimize and maximize~$\mathsf S_\c$ in Section~\ref{sec:bounds}
where we use the following result obtained in the previous proof.

\begin{corollary}\label{cor:number_crossing_cut_paths}
Let $(W,S)$ be an irreducible finite Coxeter system of rank~$n$ and $c\in \mathsf {Cox}(W,S)$.
Then the number of cut paths~$\kappa$ that cross~$\kappa_c$ is

\[
	\mathsf Q_c
	=
	\sum_{i\in[n-2]}2^{n-2-i}I(T_i^c) - \left(2^{n-2}-1\right).
\]
\end{corollary}

\begin{remark}
For reducible finite Coxeter groups, the cardinality of a Cambrian acyclic domain is the product of the cardinalities
of the acyclic domains for each irreducible component with respect to the corresponding parabolic Coxeter elements.
\end{remark}

\section{Examples}\label{sec:examples}

In this section, we determine explicit formulae for the cardinality~$\mathsf S_c$ of a Cambrian acyclic 
domain~$\mathsf{Acyc}_\c$ when~$c$ is a Coxeter element that minimizes or 
maximizes the total number of sources and sinks of $\Gamma_c$.
We denote the relevant subsets of $\mathsf{Cox}(W,S)$ by~$\mathsf{Cox_{min}}$ and $\mathsf{Cox_{max}}$, 
standard examples in type~$A$ are $c_{Lod}\in\mathsf {Cox_{min}}$ and $c_{alt}\in\mathsf {Cox_{max}}$.
We call a Coxeter system \Dfn{path-like} if $\Gamma$ is a path.

\subsection{Maximum total number of sources and sinks}

If~$c$ provides a bipartition of~$\Gamma$ then every node of~$\Gamma_c$ is a source or a sink and there are~$n$ sources 
and sinks in total. Theorem~2.3 of~\cite{bergeron_isometry_2009} implies that~$\mathsf S_c$ does not depend 
on~$c$ as the associated associahedra~$\Asso_c$ are isometric. In particular, $\mathsf S_c$ depends only on the type and 
rank of~$(W,S)$. 

\begin{proposition}\label{prop:S_c_for_Cox_max}
	Let $(W,S)$ be an irreducible Coxeter system of rank~$n>1$ and $c\in \mathsf {Cox_{max}}$.
	\begin{compactenum}[i)]
		\item If $(W,S)$ is path-like then
			\[
				\mathsf S_c
				= \begin{cases}
					2^{n-2}(h+3)-n\cdot\binom{n-1}{\tfrac{n}{2}}, 				& \text{$n$ even,}\\[2mm]
					2^{n-2}(h+3)-\frac{2n-1}{2}\cdot\binom{n-1}{\tfrac{n-1}{2}}, & \text{$n$ odd.}
				\end{cases}
			\]
		\item If $(W,S)$ is of type $D_n$ then
		\[
			\mathsf S_c 
				= \begin{cases}
					2^{n-2}(h+3) - n\cdot\binom{n-1}{\tfrac{n}{2}} +\frac{1}{2}\cdot\binom{n-2}{\tfrac{n-2}{2}},	& \text{$n$ even,}\\[2mm]
					2^{n-2}(h+3) - (n-1)\binom{n-1}{\tfrac{n-1}{2}}	- \binom{n-3}{\tfrac{n-3}{2}},		& \text{$n$ odd.}
				  \end{cases}
		\]
		\item If $(W,S)$ is of type~$E_6$, $E_7$ or $E_8$ then
		\[
			\mathsf S_c
				= \begin{cases}
					2^{n-2}(h+3) - 2(n-2)\binom{n-2}{\tfrac{n-2}{2}} - 2\cdot\binom{n-4}{\tfrac{n-4}{2}} - (n-3)(n-4),	
							& \text{$n$ even,}\\[2mm]
					2^{n-2}(h+3) - (n-1)\binom{n-1}{\tfrac{n-1}{2}} + \binom{n-3}{\tfrac{n-3}{2}} - (n-3)(n-4),
							& \text{$n$ odd.}
				  \end{cases}
		\]
	\end{compactenum}
\end{proposition}

\begin{proof}
	We aim for explicit formulae for $I(T^c_{i})$, $1\leq i\leq n-2$, apply	Theorem~\ref{thm:singleton_cuts}
	and simplify the result using the closed form of a hypergeometric sum used in Section~\ref{sec:counting_in_type_a}.
	
	Suppose that $(W,S)$ is path-like and recall from Example~\ref{expl:I(T_i)} that
	\[
		I(T^c_{i}) =  \binom{2j}{j} \quad\text{if $i=2j$,}
		\qquad\qquad\text{and}\qquad\qquad
		I(T^c_{i}) =  \frac{1}{2}\binom{2j}{j} \quad\text{if $i=2j-1$.}
	\]
	We prove the claim for $n=2k-1$, the proof is along the same lines if $n=2k$. Theorem~\ref{thm:singleton_cuts} 
	and $2^{2n+1}\sum_{i=0}^{n}2^{-2i}\binom{2i}{i} = (n+1)\binom{2n+2}{n+1}$ imply
	\begin{align*}
		\mathsf S_c 
			&= 2^{n-2}(h+1) - \sum_{i\in[n-2]}2^{n-2-i}I(T^c_i)\\
			&= 2^{n-2}(h+1) 
				- 2^{n-1}\left(
						 \sum_{j\in[k-2]}
							\left( 2^{-2j}I(T^c_{2j-1}) + 2^{-(2j+1)}I(T^c_{2j}) \right)
						 + 2^{-(n-1)}I(T^c_{n-2})
						 \right)\\
			&= 2^{n-2}(h+1) 
				- 2^{n-1}\left(
						 \sum_{j\in[k-2]}
							2^{-2j}\binom{2j}{j}
							+ 2^{-(2k-1)}\binom{2(k-1)}{k-1}
							+ 1
							- 1
						 \right)\\
			&= 2^{n-2}(h+3) 
				- 2^{n-1}\left(
						 	2^{-(2k-3)}(k-1)\binom{2(k-1)}{k-1}
						 	+ 2^{-(2k-1)}\binom{2(k-1)}{k-1}
						 \right)\\
			&= 2^{n-2}(h+3) 
				- \frac{2n-1}{2}\binom{n-1}{\frac{n-1}{2}}.
	\end{align*}

	Suppose that $(W,S)$ is of type $D_n$. Then $I(T^c_{n-2}) = I(T^c_{n-3})$
	as well as
	\[
		I(T^c_{i}) =  \binom{2j}{j} \quad\text{if $i=2j$,}
		\qquad\quad\text{and}\qquad\quad
		I(T^c_{i}) =  \frac{1}{2}\binom{2j}{j} \quad\text{if $i=2j-1$,}
	\]
	for $1\leq i\leq n-3$. Substitution of $I(T^c_i)$ in the formula for $\mathsf S_c$ of 
	Theorem~\ref{thm:singleton_cuts} and a computation similar to the previous case yields 
	the claim.
	
	Finally, we assume that $(W,S)$ is of type $E_n$ for $n\in\{6,7,8\}$. If $i\in[n-4]$ then 
	\[
		I(T^c_{i}) =  \binom{2j}{j} \quad\text{if $i=2j$,}
		\qquad\quad\text{and}\qquad\quad
		I(T^c_{i}) =  \frac{1}{2}\binom{2j}{j} \quad\text{if $i=2j-1$,}
	\]
	as well as $I(T^c_{n-3})=I(T^c_{n-4})$ and $I(T^c_{n-2})=(n-3)(n-4)$. 
	Theorem~\ref{thm:singleton_cuts} implies the claim.
\end{proof}

\begin{remark}\hfill
\begin{compactenum}[i)]
	\item Notice that Proposition~\ref{prop:S_c_for_Cox_max} yields Equation~\eqref{formula:fishburn} of Galambos and Reiner 
		if $(W,S)$ is of type~$A$ as~$W \cong \Sigma_{n+1}$, $c_{alt}\in\mathsf {Cox_{max}}$ and $h=n+1=m$.
	\item For the Coxeter groups of type $I_2(m)$ we obtain $\mathsf S_c=m+1$ for $c\in\mathsf {Cox_{max}}$.
	\item Substitution of the relevant Coxeter numbers yields~$\mathsf S_c$ for exceptional finite Coxeter groups and $c\in\mathsf{Cox_{max}}$:\\[-2mm]
		\begin{center}
			\begin{tabular}{c|c|c|c|c|c|c}
			  	$(W,S)$			& $H_3$ & $H_4$	& $F_4$ & $E_6$ & $E_7$ & $E_8$ \\
				\hline
				$\mathsf S_c$	& $21$	& $120$	& $48$	& $182$ & $546$ & $1840$
			\end{tabular}.
		\end{center}
\end{compactenum}
\end{remark}

\subsection{Minimum total number of sources and sinks}

Since $\Gamma$ is a tree with at most one branching point, the minimum number of sources and sinks of~$\Gamma_c$ is two or three: 
if~$\Gamma$ is a path then $|\mathsf {Cox_{min}}|=2$ while $|\mathsf {Cox_{min}}|=6$ if~$\Gamma$ has a branching point. In the 
latter case, we partition $\mathsf {Cox_{min}}$ into $\mathsf {Cox_a}$, $\mathsf {Cox_b}$ and $\mathsf {Cox_c}$ 
where each set consists of the Coxeter element shown in Figure~\ref{fig:Cmin_orientations_for_Dn_En} together with $\mathsf{rev}(c)$. 

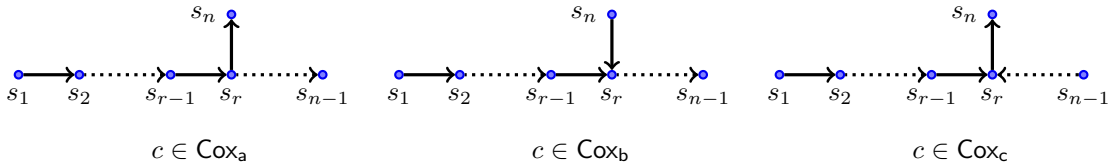
\begin{figure}[!hb]
\begin{center}
	\begin{tikzpicture}[scale=1,
	 pointille/.style={dashed},
	 axe/.style={color=black, very thick},
	 sommet/.style={inner sep=1pt,circle,draw=blue!95!black,fill=blue!50,thick,anchor=base}]

	\coordinate (0) at (0,0);
	\matrix[matrix of math nodes, column sep=0.25cm,row sep=0.2cm]
	{
	\node[sommet,label=below:s_1] (s1) at (0) {};
	\node[sommet,label=below:s_2] (s2) at (0.8,0) {} edge[<-,very thick] (s1);
	\node[sommet,label=below:{s_{r-1}}] (s3) at (2.0,0) {} edge[<-,dotted,very thick] (s2);
	\node[sommet,label=below:{s_{r}}] (s4) at (2.8,0) {} edge[<-,very thick] (s3);
	\node[sommet,label=below:s_{n-1}] (s5) at (4.0,0) {} edge[<-,dotted, very thick] (s4); 
	\node[sommet,label=left:s_n] (s7) at (2.8,0.8) {} edge[<-,very thick] (s4);&
	\node[sommet,label=below:s_1] (s1) at (0) {};
	\node[sommet,label=below:s_2] (s2) at (0.8,0) {} edge[<-,very thick] (s1);
	\node[sommet,label=below:s_{r-1}] (s3) at (2.0,0) {} edge[<-,dotted,very thick] (s2);
	\node[sommet,label=below:{s_{r}}] (s4) at (2.8,0) {} edge[<-,very thick] (s3);
	\node[sommet,label=below:s_{n-1}] (s5) at (4.0,0) {} edge[<-,dotted, very thick] (s4); 
	\node[sommet,label=left:s_n] (s7) at (2.8,0.8) {} edge[->,very thick] (s4);&
	\node[sommet,label=below:s_1] (s1) at (0) {};
	\node[sommet,label=below:s_2] (s2) at (0.8,0) {} edge[<-,very thick] (s1);
	\node[sommet,label=below:s_{r-1}] (s3) at (2.0,0) {} edge[<-,dotted,very thick] (s2);
	\node[sommet,label=below:{s_{r}}] (s4) at (2.8,0) {} edge[<-,very thick] (s3);
	\node[sommet,label=below:s_{n-1}] (s5) at (4.0,0) {} edge[->,dotted, very thick] (s4); 
	\node[sommet,label=left:s_n] (s7) at (2.8,0.8) {} edge[<-,very thick] (s4); \\
	\node[label=east:{c \in \mathsf{Cox_a}}] at (1.5,0) {}; & 
	\node[label=east:{c \in \mathsf{Cox_b}}] at (1.5,0) {};& 
	\node[label=east:{c \in \mathsf{Cox_c}}] at (1.5,0) {}; \\
	};
\end{tikzpicture}
\end{center}
\caption{Oriented Coxeter graphs~$\Gamma_c$ with branching point and minimum total number of sources and sinks ($r=n-2$ in type~$D$ and $r=n-3$ in type $E$).}
\label{fig:Cmin_orientations_for_Dn_En}
\end{figure}

The characterization of isometric~$\Asso_c$ in~\cite{bergeron_isometry_2009} implies that if $(W,S)$ is of type~$D$ or~$E$ then 
there are three distinct isometry classes of associahedra~$\Asso_c$ that correspond to the sets $\mathsf{Cox_a}$, $\mathsf{Cox_b}$ 
and $\mathsf{Cox_c}$ unless
\begin{compactitem}
	\item $(W,S)$ is of type~$D_4$, where $\mathsf{Cox_{min}}$ provides one isometry class;
	\item $(W,S)$ is of type~$D_n$ with $n\geq 5$, where $\mathsf{Cox_a}$ and $\mathsf{Cox_b} \cup \mathsf{Cox_c}$ provide two 
		isometry classes;
	\item $(W,S)$ is of type~$E_6$, where the $\mathsf{Cox_a} \cup \mathsf{Cox_b}$ and $\mathsf{Cox_c}$ provide two isometry classes.
\end{compactitem}
The next proposition provides explicit formulae for~$\mathsf S_\c$ and the various situations of $\c\in\mathsf{Cox_{min}}$.

\begin{proposition}\label{prop:S_c_for_Cox_min}
	Let $(W,S)$ be an irreducible Coxeter system of rank~$n>1$ and $c\in \mathsf {Cox_{min}}$.
	\begin{compactenum}[i)]
		\item Suppose that $(W,S)$ is path-like then
			\[
				\mathsf S_c	= 2^{n-2}(h-n+3).
			\]
		\item Suppose $(W,S)$ is of type $D_n$ and $n\geq 4$ then
		\[
			\mathsf S_c 
				= \begin{cases}
					2^{n-2}\big(h-n+\frac{7}{2}\big),		& \text{$c\in\mathsf {Cox_a}$,}\\[2mm]
					2^{n-2}\big(h-n+4\big)-2,				& \text{$c\in\mathsf {Cox_b}\cup \mathsf {Cox_c}$.}
				  \end{cases}
		\]
		\item Suppose $(W,S)$ is of type~$E_6$, $E_7$ or $E_8$. Then
		\[
			\mathsf S_c
				= \begin{cases}
					2^{n-2}(h-n+4) - 2^{n-4},
							& \text{$c\in\mathsf {Cox_a}$,}\\[2mm]
					2^{n-2}(h-n+4) - 4,
							& \text{$c\in\mathsf {Cox_b}$,}\\[2mm]
					2^{n-2}(h-n+4) + 2^{n-2} - 2n,
							& \text{$c\in\mathsf {Cox_c}$.}
				  \end{cases}
		\]
	\end{compactenum}
\end{proposition}

\begin{proof}
	Again, all claims follow from Theorem~\ref{thm:singleton_cuts} together with the following
	formulae for~$I(T^c_i)$.
	
	First, if $(W,S)$ is path-like then $I(T^\c_i)=2^i$ for all $i\in [n-2]$ and all $c\in\mathsf {C_{min}}$. 
	
	Second, if $(W,S)$ is of type $D_n$ and $n\geq 4$. Then $I(T^\c_i)=2^i$ for all $i\leq [n-4]$ and 
	$\c\in \mathsf {Cox_{min}}$ as well as
	\[
	\resizebox{\textwidth}{!}{$
		I(T^{\c_{\mathsf a}}_{i}) 
		= \begin{cases}
			2^{n-3},	&	i=n-3,\\
			2^{n-3},	& 	i=n-2,
		  \end{cases}\quad		
		I(T^{\c_{\mathsf b}}_{i}) 
		= \begin{cases}
			1,			&	i=n-3,\\
			2^{n-2},	& 	i=n-2,
		  \end{cases}\quad\text{and}\quad
		I(T^{\c_{\mathsf c}}_{i}) 
		= \begin{cases}
			2^{n-3},	&	i=n-3,\\
			2,			& 	i=n-2
		  \end{cases}
	$}
	\]
	where $c_{\mathsf a}\in\mathsf {Cox_a}$, $c_{\mathsf b}\in\mathsf {Cox_b}$ and $c_{\mathsf c}\in\mathsf {Cox_c}$.
	
	Third, if $(W,S)$ is of type~$E_n\in\{E_6, E_7, E_8\}$. Then
	$I(T^{\c}_i)=2^i$ for $i\in[n-5]$ and all $\c\in \mathsf {Cox_{min}}$ as well as
	\[
	\resizebox{\textwidth}{!}{$
		I(T^{c_{\mathsf a}}_{i}) 
		= \begin{cases}
			2^{n-4},		&	i=n-4,\\
			2^{n-4},		& 	i=n-3,\\
			3\cdot 2^{n-4},	&	i=n-2
		  \end{cases}
		\quad
		I(T^{c_{\mathsf b}}_{i}) 
		= \begin{cases}
			1,				&	i=n-4,\\
			2^{n-3},		& 	i=n-3,\\
			2^{n-2},		&	i=n-2
		  \end{cases}
		\quad\text{and}\quad
		I(T^{c_{\mathsf c}}_{i}) 
		= \begin{cases}
			2^{n-4},		&	i=n-4,\\
			2,				& 	i=n-3,\\
			2n-4,			&	i=n-2,
		  \end{cases}
	$}
	\]
	where $c_{\mathsf a}\in\mathsf {Cox_a}$, $c_{\mathsf b}\in\mathsf {Cox_b}$ and $c_{\mathsf c}\in\mathsf {Cox_c}$.
\end{proof}

\begin{remark}\label{rem:minima}
\hfill
	\begin{compactenum}[i)]
		\item If $(W,S)$ is of type $I_2(m)$ then $\mathsf{Cox_{min}}=\mathsf{Cox_{max}}$. 
			Notice that Proposition~\ref{prop:S_c_for_Cox_max} and \ref{prop:S_c_for_Cox_min} yield $\mathsf S_c=m+1$ 
			in both cases for all $m$.
		\item Let $(W,S)$ be of type $D_n$. If $n=4$ then both formulae of Proposition~\ref{prop:S_c_for_Cox_min} yield
			$\mathsf S_\c=22$. Moreover, if $n\geq 5$ and $\c\in\mathsf{Cox_{min}}$ then $\mathsf S_\c$ is minimal if and only if
			$\c\in\mathsf{Cox_a}$. 
		\item Let $(W,S)$ be of type $E_n$. \\
			If~$n=6$ and $c\in \mathsf {Cox_a}\cup \mathsf{Cox_b}$ then $\mathsf S_c=156$ while $\mathsf S_c=164$ 
			for $c\in \mathsf {Cox_c}$. \\
			If $n=7$ then $\mathsf S_c=472$ for $c\in \mathsf {Cox_a}$, $\mathsf S_c=476$ for $c\in \mathsf{Cox_b}$ 
			and $\mathsf S_c=498$ for $c\in \mathsf{Cox_c}$. \\
			If $n=8$ then $\mathsf S_c=1648$ for $c\in \mathsf {Cox_a}$, $\mathsf S_c=1660$ for $c\in \mathsf{Cox_b}$ 
			and $\mathsf S_c=1904$ for $c\in \mathsf{Cox_c}$. 
		\item The minimum numbers of $\mathsf S_c$ for all exceptional finite Coxeter groups and $c\in\mathsf{Cox_{min}}$ are:\\[-2mm]
			\begin{center}
				\scalebox{1}{
				\begin{tabular}{c|c|c|c|c|c|c}
				  	$(W,S)$			& $H_3$ & $H_4$	& $F_4$ & $E_6$ & $E_7$ & $E_8$ \\
					\hline
					$\mathsf S_c$	& $20$	& $116$	& $44$	& $156$ & $472$ & $1648$
				\end{tabular}
				}
			\end{center}
	\end{compactenum}
\end{remark}

\section{Lower and Upper bounds}\label{sec:bounds}

In this section, we use Theorem~~\ref{thm:singleton_cuts} and Corollary~\ref{cor:number_crossing_cut_paths} to prove upper 
and lower bounds for the cardinality~$\mathsf S_c$ of a Cambrian acyclic domain~$\mathsf{Acyc}_\c$ by identification of 
minimizers and maximizers for~$\sum_{i=1}^{n-2}2^{n-2-i}I(T^c_i)$. As done in the proof of Proposition~\ref{prop:drawing}, we fix a Coxeter triple $(W,S,\c)$ 
and label the generators of~$S$ along a longest path of~$\Gamma$ such that $s_1,\ldots, s_p$ with $p\in\{n-1,n\}$ are successive 
vertices and in types~$D$ and $E$ we have $p=n-1$ and the vertex $s_n$ is connected to~$s_r$ where $r=n-2$ (type~$D_n$) or $r=n-3$ (type $E$).

\subsection{Cut functions}

\begin{definition}[Cut function]
	\hfill\break
	A function $f:S \rightarrow \Z$ is a \Dfn{cut function} of $(W,S)$ if $f(s_1)$ is odd and $|f(s)-f(t)|=1$ for all non-commuting pairs $s,t\in S$. 
	We write $(f(s_1),\dots,f(s_n))$ for the cut function $f$ and denote the set of all cut functions of~$(W,S)$ by $\mathsf{CF}(W,S)$.
	A generator $s\in S$ is an \Dfn{extremum} of the cut function~$f$ if either $f(s)<f(t)$ 	or $f(s)>f(t)$ for all $t\in S$ that do not commute with~$s$.
\end{definition}

Since $|f(s)-f(t)|=1$ for all non-commuting pairs $s,t\in S$, any cut function~$f$ determines a unique Coxeter 
element~$c_f\in\mathsf{Cox}(W,S)$ such that $f(s)<f(t)$ if and only if $(s, t)$ is a directed edge of~$\Gamma_{c_f}$ and every 
Coxeter element determines a cut function up to an even constant. Moreover, sources and sinks of~$\Gamma_{c_f}$ 
correspond to extrema of~$f$ and, among the $h$ cut paths of $\Phi^{-1}\big(\operatorname{rev}(c_f)\big)$ from 
Corollary~\ref{cor:number_of_cut_paths}, the cut function~$f$ determines a unique cut path~$\kappa_f$ as follows. Let~$\sigma$ 
be the vertex of~$\TwoCov$ with $\mathsf{g}(\sigma)=s_2$ such that the $x$-coordinate~$x_v$ of $v\in pr_{\c^h}^{-1}(\sigma)$ 
satisfies $x_v\equiv f(s_1)\mod{2h}$ and let $(\rho,\sigma)$ and $(\sigma,\tau)$ be the two edges of $\TwoCov$ with 
$\mathsf{g}(\rho)=\mathsf{g}(\tau)=s_1$. Now $\kappa_f$ contains $(\rho,\sigma)$ if $f(\mathsf{g}(\rho))>f(\mathsf{g}(\sigma))$ 
and $(\sigma,\tau)$ if $f(\mathsf{g}(\sigma))<f(\mathsf{g}(\tau))$. We say that the cut path~$\kappa_f$ represents the cut 
function~$f$.

\begin{example}\label{expl:cut_functions}
	Consider $(W,S)$ of type $A_4$ and $\c =s_2s_1s_4s_3$. The cut functions $f$ and $g$ with $f(s_1,s_2,s_3,s_4)=(-1,0,1,2)$
	and $g(s_1,s_2,s_3,s_4)=(1,0,1,0)$ determine the Coxeter elements $c_f=s_1s_2s_3s_4$ and $c_g=s_2s_1s_4s_3$ indicated as 
	shaded subgraphs in the planar drawing of~$\TwoCov$ in Figure~\ref{fig:crossing}. As $f(s_1)=-1\equiv 9 \mod{10}$ and
	$f(s_1)<f(s_2)$ as well as $g(s_1)=1\equiv 1 \mod{10}$ and $g(s_1)>g(s_2)$, we obtain the cut paths $\kappa_f$ 
	and $\kappa_g$ that represent the cut functions~$f$ and~$g$ as indicated. The extrema of~$f$ are $s_1$ and $s_4$, while
	all generators in~$S$ are extrema of~$g$.
\end{example}

\begin{figure}[!htb]
\begin{center}
	\centerline{\begin{tikzpicture}[scale=0.6,
	decoration={snake,amplitude=.2mm,segment length=1mm},
	copies/.style={line width=15pt,color=blue,draw opacity=0.20,line cap=round,line join=round}]

  \node (left1) at (4.125,0.1875) {};
  \node (left2a) at (4.125,3.1875) {};
  \node (left2b) at (4.125,2.8125) {};
  
  \draw[copies] (6,0) -- (9,4.5);
  
  \draw[thick] (4.1,-0.5) -- (4.1,5);
  \draw[thick] (4.1,-0.5) -- (14.8,-0.5);
  \draw[thick] (4.1,5) -- (14.8,5);
  \draw[thick] (14.8,-0.5) -- (14.8,5);
  
%%%%%%%%%%%%%%%%%%%%%%%%%%

  \node (s2-1) at (5,1.5) {\small$s_2$} edge[<-,thick,densely dotted] (left1) edge[<-,thick,densely dotted] (left2b);
  \node (s4-1) at (5,4.5) {\small$s_4$} edge[<-,thick,densely dotted] (left2a);
  \node (s1-1) at (6,0) {\small$s_1$} edge[<-,thick,decorate] (s2-1);
  \node (s3-1) at (6,3) {\small$s_3$} edge[<-,thick] (s2-1)  edge[<-,thick] (s4-1);
    
  \node (s2-2) at (7,1.5) {\small$s_2$} edge[<-,thick,decorate] (s3-1)  edge[<-,thick] node[pos=0.5,circle,fill=blue!40,draw=white,inner sep=1pt,yshift=-0.17cm,xshift=0.17cm] {\small $a$} (s1-1);
  \node (s4-2) at (7,4.5) {\small$s_4$} edge[<-,thick] node[pos=0.5,circle,fill=blue!40,draw=white,inner sep=1pt,yshift=0.15cm,xshift=-0.15cm] {\small $b$} (s3-1);
  \node (s1-2) at (8,0) {\small$s_1$} edge[<-,thick] (s2-2);
  \node (s3-2) at (8,3) {\small$s_3$} edge[<-,thick] (s2-2)  edge[<-,thick,decorate] (s4-2);
  
  \node (s2-3) at (9,1.5) {\small$s_2$} edge[<-,thick] (s3-2)  edge[<-,thick] (s1-2);
  \node (s4-3) at (9,4.5) {\small$s_4$} edge[<-,thick] (s3-2);

%%%%%%%%%%%%%%%%%%%%%%%%%%

  \node (s1-3) at (10,0) {\small$s_1$} edge[<-,thick] (s2-3);
  \node (s3-3) at (10,3) {\small$s_3$} edge[<-,thick] (s4-3)  edge[<-,thick] (s2-3);
  \node (s2-4) at (11,1.5) {\small$s_2$} edge[<-,thick] (s1-3)  edge[<-,thick] (s3-3);
  \node (s4-4) at (11,4.5) {\small$s_4$} edge[<-,thick,decorate] (s3-3);
  
  \node (s1-4) at (12,0) {\small$s_1$} edge[<-,thick] (s2-4);
  \node (s3-4) at (12,3) {\small$s_3$} edge[<-,thick,decorate] (s2-4)  edge[<-,thick] (s4-4);
  \node (s2-5) at (13,1.5) {\small$s_2$} edge[<-,thick,decorate] (s1-4)  edge[<-,thick] (s3-4);
  \node (s4-5) at (13,4.5) {\small$s_4$} edge[<-,thick] (s3-4);
  
  \node (s1-5) at (14,0) {\small$s_1$} edge[<-,thick] (s2-5);
  \node (s3-5) at (14,3) {\small$s_3$} edge[<-,thick] (s2-5)  edge[<-,thick] (s4-5);

  \node (s2-6b) at (14.75,1.125) {} edge[thick,densely dotted] (s1-5);
  \node (s2-6a) at (14.75,1.875) {} edge[thick,densely dotted] (s3-5);
  \node (s4-6) at (14.75,4.125) {} edge[thick,densely dotted] (s3-5);

  \node[red] at (8,5.5) {$\color{red}{\kappa_{f}}$};
  \draw[red,thick,->,densely dashed] (5,-0.5) -- (5,0.15) -- (8,4.65) -> (8,5);

  \node[red] at (10,5.5) {$\color{red}{\kappa^*_{f}}$};
  \draw[red,thick,->,densely dashed] (13,-0.5) -- (13,0.15) -- (10,4.65) -> (10,5);
  
  	\node at (6,-1) {$0$};
	\node at (8,-1) {$2$};
	\node at (10,-1) {$4$};
	\node at (12,-1) {$6$};
	\node at (14,-1) {$8$};
	
	\node at (10,-2) {$f(s_1,s_2,s_3,s_4)=(-1,0,1,2)$};
  
\end{tikzpicture}$\qquad$\begin{tikzpicture}[scale=0.6,decoration={snake,amplitude=.2mm,segment length=1mm},
copies/.style={line width=15pt,color=blue,draw opacity=0.20,line cap=round,line join=round}]

  \draw[copies] (8,0) -- (7,1.5) -- (8,3) -- (7,4.5);

  \node (left1) at (4.125,0.1875) {};
  \node (left2a) at (4.125,3.1875) {};
  \node (left2b) at (4.125,2.8125) {};
  
  \draw[thick] (4.1,-0.5) -- (4.1,5);
  \draw[thick] (4.1,-0.5) -- (14.8,-0.5);
  \draw[thick] (4.1,5) -- (14.8,5);
  \draw[thick] (14.8,-0.5) -- (14.8,5);
  
%%%%%%%%%%%%%%%%%%%%%%%%%%

  \node (s2-1) at (5,1.5) {\small$s_2$} edge[<-,thick,densely dotted] (left1) edge[<-,thick,densely dotted] (left2b);
  \node (s4-1) at (5,4.5) {\small$s_4$} edge[<-,thick,densely dotted] (left2a);
  \node (s1-1) at (6,0) {\small$s_1$} edge[<-,thick] node[pos=0.5,circle,fill=blue!40,draw=white,inner sep=1pt,yshift=-0.15cm,xshift=-0.15cm] {\small $c$} (s2-1);
  \node (s3-1) at (6,3) {\small$s_3$} edge[<-,thick] (s2-1)  edge[<-,thick] (s4-1);
  
  \node (s2-2) at (7,1.5) {\small$s_2$} edge[<-,thick,decorate] (s3-1)  edge[<-,thick,decorate] (s1-1);
  \node (s4-2) at (7,4.5) {\small$s_4$} edge[<-,thick,decorate] (s3-1);
  \node (s1-2) at (8,0) {\small$s_1$} edge[<-,thick] (s2-2);
  \node (s3-2) at (8,3) {\small$s_3$} edge[<-,thick] (s2-2)  edge[<-,thick] node[pos=0.5,circle,fill=blue!40,draw=white,inner sep=1pt,yshift=0.17cm,xshift=0.17cm] {\small $d$} (s4-2);
  
  \node (s2-3) at (9,1.5) {\small$s_2$} edge[<-,thick] (s3-2)  edge[<-,thick] (s1-2);
  \node (s4-3) at (9,4.5) {\small$s_4$} edge[<-,thick] (s3-2);

%%%%%%%%%%%%%%%%%%%%%%%%%%

  \node (s1-3) at (10,0) {\small$s_1$} edge[<-,thick] (s2-3);
  \node (s3-3) at (10,3) {\small$s_3$} edge[<-,thick] (s4-3)  edge[<-,thick] (s2-3);
  \node (s2-4) at (11,1.5) {\small$s_2$} edge[<-,thick] (s1-3)  edge[<-,thick] (s3-3);
  \node (s4-4) at (11,4.5) {\small$s_4$} edge[<-,thick] (s3-3);

  \node (s1-4) at (12,0) {\small$s_1$} edge[<-,thick,decorate] (s2-4);
  \node (s3-4) at (12,3) {\small$s_3$} edge[<-,thick,decorate] (s2-4)  edge[<-,thick,decorate] (s4-4);
  \node (s2-5) at (13,1.5) {\small$s_2$} edge[<-,thick] (s1-4)  edge[<-,thick] (s3-4);
  \node (s4-5) at (13,4.5) {\small$s_4$} edge[<-,thick] (s3-4);
  
  \node (s1-5) at (14,0) {\small$s_1$} edge[<-,thick] (s2-5);
  \node (s3-5) at (14,3) {\small$s_3$} edge[<-,thick] (s2-5)  edge[<-,thick] (s4-5);

  \node (s2-6b) at (14.75,1.125) {} edge[thick,densely dotted] (s1-5);
  \node (s2-6a) at (14.75,1.875) {} edge[thick,densely dotted] (s3-5);
  \node (s4-6) at (14.75,4.125) {} edge[thick,densely dotted] (s3-5);

  \node at (6.5,5.5) {$\color{red}{\kappa_{g}}$};
  \draw[red,thick,->,densely dashed] (6.5,-0.5) -> (6.5,5);

  \node at (11.5,5.5) {$\color{red}{\kappa^*_{g}}$};
  \draw[red,thick,->,densely dashed] (11.5,-0.5) -> (11.5,5);
  
    	\node at (6,-1) {$0$};
	\node at (8,-1) {$2$};
	\node at (10,-1) {$4$};
	\node at (12,-1) {$6$};
	\node at (14,-1) {$8$};
	
	\node at (10,-2) {$g(s_1,s_2,s_3,s_4)=(1,0,1,0)$};
  
\end{tikzpicture}}
\end{center}
	\caption{\label{fig:crossing} The cut functions $f$ and $g$ and their associated cut paths $\kappa_{f}$ and $\kappa_{g}$ and Coxeter elements $c_f$ and $c_g$ in type $A_4$.}
\end{figure}
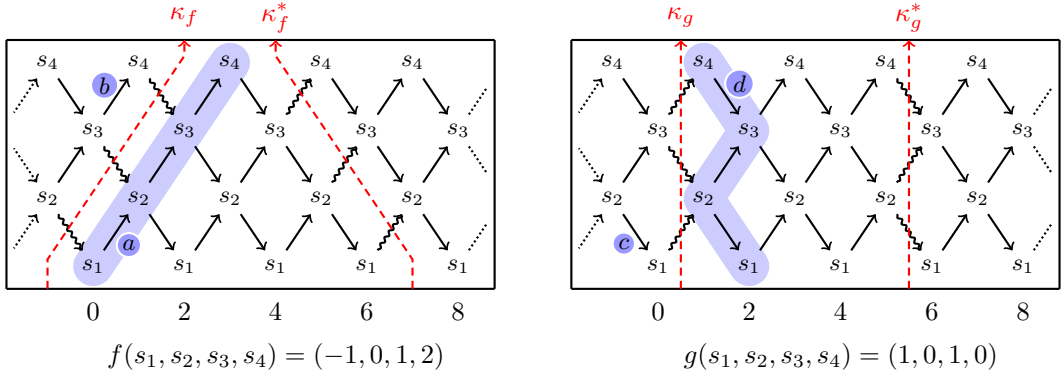

\begin{definition}[equivalence and crossing of cut functions]\hfill\break
	Let $f, g:S\rightarrow \Z$ be cut functions of the finite and irreducible Coxeter system~$(W,S)$.
	\begin{compactenum}[i)]
		\item $f$ and $g$ are \Dfn{equivalent}, $f\simeq g$, if and only if $f(s)\equiv g(s)\mod{2h}$ for all $s\in S$.
		\item $f$ and $g$ \Dfn{cross} if and only if there exist $s,t\in S$ and $\widetilde f\simeq f$ such that
			$\widetilde f(s) < g(s)$ and $\widetilde f(t) > g(t)$.
	\end{compactenum}
\end{definition}

Before showing that the notion of crossing cut functions~$f$ and~$g$ coincides with the crossing of cut paths~$\kappa_f$ 
and $\kappa_g$ that represent $f$ and $g$ in the next lemma, we illustrate the definition with an example.

\begin{example}[Example~\ref{expl:cut_functions} continued]
	\hfill\break
	Consider $f'(s_1,s_2,s_3,s_4)=(19,20,21,22)$ and $f''(s_1,s_2,s_3,s_4)=(0,1,2,3)$. Then $f\simeq f'$ and $f\not \simeq f''$.
	Moreover, $f'$ crosses $g$ as $f'\simeq f$ and $f(s_1)<g(s_1)$ and $f(s_4)>g(s_4)$. 
	Further, notice that the edges $a$ and $b$ required by Definition~\ref{def:crossing} show 
	that $\kappa_{g}$ crosses $\kappa_{f}$ while the edges~$c$ and $d$ show that $\kappa_f$ crosses $\kappa_g$.
\end{example}

\begin{lemma}\label{lem:equiv_of_crossing}
	Let~$(W,S,\c)$ be a Coxeter triple with associated $2$-cover~$\TwoCov$ as well as cut paths $\kappa_f$ and $\kappa_g$  
	representing cut functions $f$ and $g$. 
	The cut paths $\kappa_f$ and $\kappa_g$ cross if and only if the cut functions~$f$ and~$g$ cross.
\end{lemma}

\begin{proof}
	Assume that the cut path~$\kappa_f$ crosses~$\kappa_g$ on~$\TwoCov$. Since $\kappa_f$ and $\kappa_g$ are crossing, they 
	have a common tile~$T$ and there are edges $e_1\in\kappa_g$ in the initial side of $\kappa_f$ and $e_2\in\kappa_g$ in 
	the final side of~$\kappa_f$. These edges are also in $\TwoCov\setminus \kappa_f$. The common tile~$T$ allows us to get 
	specific representatives $\widetilde f$ and $\widetilde g$ for $f$ and $g$ by going down to the first tiles of $\kappa_f$ 
	and $\kappa_g$. Indeed, consider the (unoriented) edge $\{\rho,\sigma\}\in\kappa_f$ such that $\mathsf g(\rho)=s_2$ and 
	$\mathsf g(\sigma)=s_1$ and use the $x$-coordinate $x_\rho$ of $\rho$ as the value 
	$\widetilde f(s_1)=\widetilde f(\mathsf{g}(\sigma)):= x_\rho \simeq f(s_1) \mod{2h}$. This determines $\widetilde f\simeq f$. 
	Proceed similarly to obtain~$\widetilde g$. Further, since we have edges $e_1$ and $e_2$ on distinct connected components 
	of $\TwoCov\setminus \big(\kappa_f \cup \kappa_f^*\big)$, there exist $s,t\in S$ such that
	\[
		\widetilde f(s) < \widetilde g(s) 
		\quad\text{and}\quad
		\widetilde f(t) > \widetilde g(t)
	\qquad\qquad\text{or}\qquad\qquad
		\widetilde f(s) > \widetilde g(s) 
		\quad\text{and}\quad
		\widetilde f(t) < \widetilde g(t)
	\]
	depending if the initial side of $\kappa_f$ is ``on the right'' or ``on the left'' of $\kappa_f$.

	Now suppose that $f$ and $g$ are crossing cut functions with equivalent cut functions $\widetilde f \simeq f$ 
	and $\widetilde g \simeq g$ such that $\widetilde f(s) < \widetilde g(s)$ and $\widetilde f(t) > \widetilde g(t)$
	for some $s,t\in S$. Since the Coxeter graph $\Gamma$ is a tree and since a cut function~$h$ satisfies $|h(a)-h(b))|=1$
	for non-commuting $a,b\in S$, 
	every integral value between $\widetilde f(s)$ and $\widetilde f(t)$ is in the image of $\widetilde f$. Because of the 
	latter inequalities for $\widetilde f(s)$ and $\widetilde g(s)$, this implies that there exists a generator $u\in S$ such that $\widetilde f(u)=\widetilde g(u)$.
	Following the procedure to obtain the cut path~$\kappa_h$ from a cut function~$h$, this implies that the representing 
	cut paths~$\kappa_f$ and $\kappa_g$ will have a common tile once we get to a tile with vertex label $u$. Further, the 
	inequalities guarantee that there will be one edge in the initial side of $\kappa_f$ and one edge in the final side 
	of $\kappa_f$ taken by~$\kappa_g$.
\end{proof}

\subsection{Upper and lower bounds for the cardinality of Cambrian acyclic domains}
\hfill\break
To obtain lower and upper bounds for $\mathsf S_c=|\mathsf{Acyc}_\c|$, the next lemma about the minimum and maximum number
of extrema of a cut function~$f$ is essential. Clearly, the maximum number of extrema is equal to~$n$ while the minimum number 
of extrema is equal to~$2$ if $(W,S)$ is path-like and equal to~$3$ if $(W,S)$ is of type~$D$ or~$E$. 

\begin{lemma}\label{lem:extremal}
	Let $f$ be a cut function of the finite irreducible Coxeter system~$(W,S)$.
	\begin{compactenum}[a)]
		\item The number of cut functions that cross~$f$ is minimum if and only if the number of extrema of~$f$ is maximum.
		\item If $(W,S)$ is path-like, then the number of cut functions that cross~$f$ is maximum if and only if the number 
			of extrema of~$f$ is minimum.\\
			If $(W,S)$ is of type $D$ or $E$, then the number of cut functions that cross~$f$ is maximum implies that the 
			number of extrema of~$f$ is minimum.
	\end{compactenum}
\end{lemma}

\begin{proof}
	a) 
	For $m\in\Z$, the \Dfn{$m$-reflection} $\mathcal{R}_m:\mathsf{CF}(W,S)\rightarrow\mathsf{CF}(W,S)$ is the bijection 
	\[
		\mathcal{R}_m(f)(m+\ell) := f(m-\ell) 
		\qquad\text{for all $\ell \in \Z$.}
	\]
	If~$f\in\mathsf{CF}(W,S)$ has less than~$n$ extrema then set
	\[
		f': S \longrightarrow \Z
		\qquad\text{via}\qquad
		s \longmapsto
		\begin{cases}
			f(s), 	& \text{if $f(s)<\max_{t\in S} f(t)$},\\
			f(s)-2, & \text{if $f(s)=\max_{t\in S} f(t)$}.
		\end{cases}
	\]
	Clearly, every extremum of~$f$ is an extremum of~$f'$ and there is at least one $s\in S$ that is extremal for~$f'$ but 
	not for~$f$, see Figure~\ref{fig:cut_width}. Now define
	\[
		\mathcal F := \set{ g \in\mathsf{CF}(W,S)}{g\text{ crosses }f}
		\quad\text{and}\quad
		\mathcal F':= \set{ g \in\mathsf{CF}(W,S)}{g\text{ crosses }f'}.
	\] 
	We first prove $|\mathcal F'| < |\mathcal F|$ by showing that $|\mathcal F'\setminus \mathcal F| < |\mathcal F\setminus \mathcal F'|$.
	To see this, let $\mu := \max_{t\in S} f(t)$ and notice that 
	\[
		\mathcal R_{\mu-1}(g) \in \mathcal F\setminus \mathcal F'
		\qquad
		\text{for all $g\in \mathcal F'\setminus \mathcal F$,}
	\]
	so the reflection~$\mathcal R_{\mu-1}$ is an injection 
	$\mathcal F'\setminus \mathcal F \hookrightarrow \mathcal F\setminus \mathcal F'$ which is not surjective because 
	$f \not\in \mathcal F'\setminus \mathcal F$ and $\mathcal R_{\mu-1}(f) \in \mathcal F\setminus \mathcal F'$. 
	Thus $|\mathcal F'| < |\mathcal F|$ and iterating this process yields a cut function where every generator of~$S$ 
	is extremal. To complete the proof, we show that $|\mathcal F|=|\mathcal G|$ if~$f$ and~$g$ are cut functions where 
	every $s\in S$ is extremal. This follows from the observation that two cut functions where every $s\in S$ is extremal 
	differ only by translation and reflection.

	\begin{figure}[!th]         
	\begin{center}
		\begin{tikzpicture}[
	scale=1,
	sommet_red/.style={inner sep=2pt,circle,draw=red!95!black,fill=red!95,thick,anchor=base},
	sommet_blue/.style={inner sep=2pt,circle,draw=blue!95!black,fill=blue!95,thick,anchor=base},
	sommet_red_blue/.style={inner sep=2pt,circle,draw=blue!95!black,fill=red!95,thick,anchor=base}]
		
	\def\delai{0.1}
	\def\xfactor{1.5}
	\def\yfactor{0.66}
		
	%axis
	\draw (0.5*\xfactor,0.5*\yfactor) --(7.1*\xfactor,0.5*\yfactor);
    	\draw (0.5*\xfactor,0.5*\yfactor) -- (0.5*\xfactor,7.1*\yfactor);
    	%ticks
    	\foreach \x in {1,...,7}
     		\draw (\x*\xfactor,0.5cm*\yfactor+1pt) -- (\x*\xfactor,(0.5cm*\yfactor-3pt)
			node[anchor=north] {$s_\x$};
    	\foreach \y in {1,...,7}
     		\draw (0.5cm*\xfactor+1pt,\y*\yfactor) -- (0.5cm*\xfactor-3pt,\y*\yfactor) 
     			node[anchor=east] {\y}; 
			
	\draw (0.5*\xfactor,4*\yfactor) -- (7.3*\xfactor,4*\yfactor) node[label=right:{${\mu-1}$}] {};

	\draw[red,densely dotted,thick] (1*\xfactor,3*\yfactor) -- (2*\xfactor,2*\yfactor) -- (5*\xfactor,5*\yfactor) -- (6*\xfactor,4*\yfactor) -- (7*\xfactor,5*\yfactor);
	\draw[red,thick] (1*\xfactor,3*\yfactor-\delai) -- (2*\xfactor,2*\yfactor-\delai) -- (4*\xfactor,4*\yfactor-\delai) -- (5*\xfactor,3*\yfactor-\delai) -- (6*\xfactor,4*\yfactor-\delai) -- (7*\xfactor,3*\yfactor-\delai);
	\draw[red,densely dashdotdotted,thick] (1*\xfactor,1*\yfactor+\delai) -- (5*\xfactor,5*\yfactor+\delai) -- (7*\xfactor,3*\yfactor+\delai);
	\draw[blue,dotted, thick] (1*\xfactor,5*\yfactor) -- (2*\xfactor,6*\yfactor) -- (5*\xfactor,3*\yfactor) -- (6*\xfactor,4*\yfactor) -- (7*\xfactor,3*\yfactor);
	\draw[blue,dashed,thick] (1*\xfactor,7*\yfactor+\delai) -- (5*\xfactor,3*\yfactor+\delai) -- (7*\xfactor,5*\yfactor+\delai);
	
	\draw[red,densely dotted,thick] (0.25*\xfactor,-\yfactor) -- (0.75*\xfactor,-\yfactor);
	\node[label=east:{$f=(3,2,3,4,5,4,5)$}] at (0.75*\xfactor,-\yfactor) {};

	\draw[red,thick] (0.25*\xfactor,-2*\yfactor) -- (0.75*\xfactor,-2*\yfactor);
	\node[label=east:{$f'=(3,2,3,4,3,4,3)$}] at (0.75*\xfactor,-2*\yfactor) {};

	\draw[red,densely dashdotdotted,thick] (0.25*\xfactor,-3*\yfactor) -- (0.75*\xfactor,-3*\yfactor);
	\node[label=east:{$g=(1,2,3,4,5,4,3)$}] at (0.75*\xfactor,-3*\yfactor) {};
		
	\draw[blue,dotted, thick] (4.0*\xfactor,-\yfactor) -- (4.5*\xfactor,-\yfactor);
	\node[label=east:{$R_{\mu-1}(f)=(5,6,5,4,3,4,3)$}] at (4.5*\xfactor,-\yfactor) {};
		
	\draw[blue,dashed,thick] (4.0*\xfactor,-3*\yfactor) -- (4.5*\xfactor,-3*\yfactor);
	\node[blue,label=east:{$R_{\mu-1}(g)=(7,6,5,4,3,4,5)$}] at (4.5*\xfactor,-3*\yfactor) {};
	
	\node[sommet_red] (s1r) at (1*\xfactor,3*\yfactor) {};
	
	\node[sommet_red] (s1r2) at (1*\xfactor,1*\yfactor) {};
	
	\node[sommet_red] (s2r) at (2*\xfactor,2*\yfactor) {};
	\node[sommet_red] (s3r) at (3*\xfactor,3*\yfactor) {};
	\node[sommet_red] (s5r) at (5*\xfactor,5*\yfactor) {};

	\node[sommet_red_blue] (s4r) at (4*\xfactor,4*\yfactor) {};
	\node[sommet_red_blue] (s5b) at (5*\xfactor,3*\yfactor) {};
	\node[sommet_red_blue] (s6r) at (6*\xfactor,4*\yfactor) {};
	\node[sommet_red_blue] (s7r) at (7*\xfactor,5*\yfactor) {};
	\node[sommet_red_blue] (s7b) at (7*\xfactor,3*\yfactor) {};
		
	\node[sommet_blue] (s1b) at (1*\xfactor,5*\yfactor) {};
	\node[sommet_blue] (s1b2) at (1*\xfactor,7*\yfactor) {};	
	\node[sommet_blue] (s2b) at (2*\xfactor,6*\yfactor) {};
	\node[sommet_blue] (s3b) at (3*\xfactor,5*\yfactor) {};
		
\end{tikzpicture}
		\caption{\label{fig:cut_width} $\mathcal F \supset \mathcal F'$ for cut functions~$f$ and~$f'$ as in the proof of Lemma~\ref{lem:extremal}.
		We have $R_{\mu-1}(f) \in \mathcal F\setminus \mathcal F'$ as well as $g\in \mathcal F'\setminus \mathcal F$ implies $R_{\mu-1}(g) \in \mathcal F\setminus \mathcal F'$.}	
	\end{center}
	\end{figure}
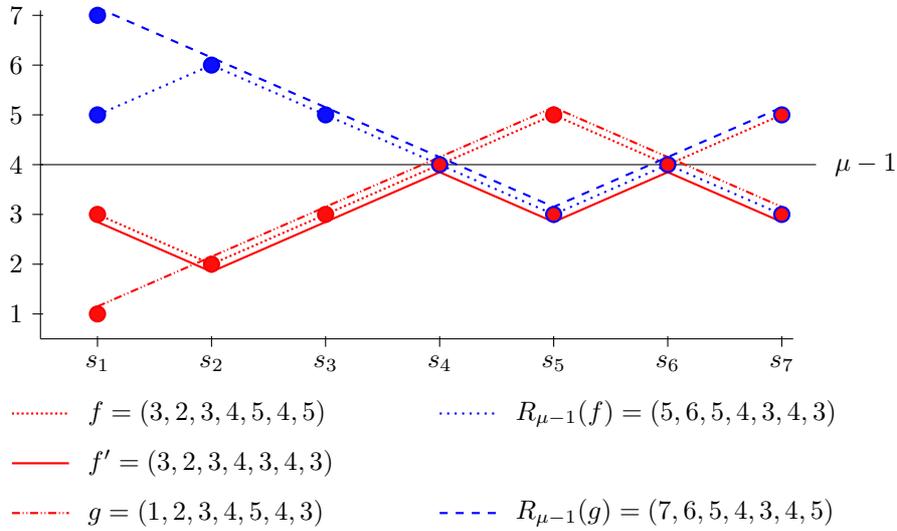

	b) For $n=1$ and $n=2$ there is nothing to prove, as all generators are extremal for each cut function. 
	We therefore assume $n\geq 3$ and prove the claim first if $(W,S)$ is path-like. A cut function~$f$ determines 
	the Coxeter element~$\c_f$ and the number of cut functions that cross~$f$ is equal to the 
	number~$\mathsf{Q}_{\operatorname{rev}(\c_f)}$ of cut paths that cross~$\kappa_{\operatorname{rev}(\c_f)}$ 
	in $\TwoCov$ by Lemma~\ref{lem:equiv_of_crossing}. By Corollary~\ref{cor:number_crossing_cut_paths}, we have
	\[
		\mathsf{Q}_{\operatorname{rev}(\c_f)}
		=
		\sum_{i\in[n-2]}2^{n-2-i}I\big(T^{\operatorname{rev}(\c_f)}_i\big)-\big(2^{n-2}-1\big),
	\]
	where $T^{\operatorname{rev}(\c_f)}_i$ is tile~$i$ of~$\kappa_{\operatorname{rev}(\c_f)}$ 
	and $I\big(T^{\operatorname{rev}(\c_f)}_i\big)$ is the number of distinct initial segments of cut paths~$\kappa$ 
	up to $T^{\operatorname{rev}(\c_f)}_i$ with edges contained
	in $\TwoCov\setminus\mathsf{out}\big(T^{\operatorname{rev}(\c_f)}_i\big)$. The reasoning of Example~\ref{expl:I(T_i)}   
	shows
	\[
		\operatorname{rev}(\c_f)
		\in
		\big\{
			s_1s_2\cdot \cdots\cdot s_n,\ 
			s_ns_{n-1}\cdot \cdots\cdot s_1
		\big\}
		\qquad\Longrightarrow\qquad
		I\big(T^{\operatorname{rev}(\c_f)}_i\big)=2^i
		\quad\text{for all $i\in[n-2]$}.
	\]
	These are clearly the only Coxeter elements	with~$I\big(T^{\operatorname{rev}(\c_f)}_i\big)=2^i$ for all $i\in[n-2]$ and these 
	values are maximum. Thus $I\big(T^{\operatorname{rev}(\c_f)}_i\big)$ attains its maximum value for each $i\leq n-2$ if and only 
	if~$f$ is strictly monotone on~$[i+2]$.
	In particular, $\mathsf{Q}_{\operatorname{rev}(\c_f)}$ is maximum if and only if the cut function~$f$ is strictly monotone. Thus, we
	conclude for any path-like Coxeter system~$(W,S)$ that~$f$ has precisely two extrema if and only if the number of cut 
	functions that cross~$f$ is maximum. 
	
	\medskip
	To analyze type~$D$, we first consider~$D_4$. Without loss of generality, we have to analyze the following four cases for~$c_f$ 
	and~$\kappa_{\operatorname{rev}(c_f)}$ as they represent all cut functions~$f$ in type $D_4$: 
	
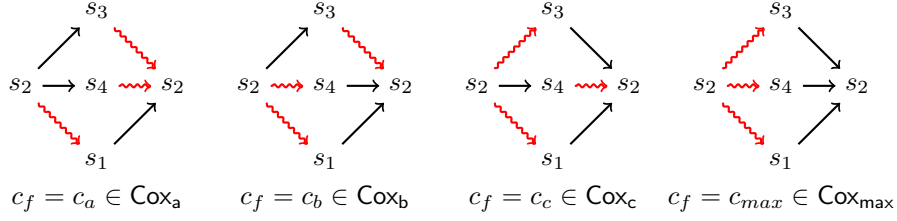
\begin{center}
		\begin{tikzpicture}[spec/.style={gray},decoration={snake,amplitude=.2mm,segment length=1mm},]

% Type Cox_a
	\node (s2-1) at (0,0) {$s_2$};
	\node (s1-1) at (1,-1) {$s_1$} edge[<-,decorate,red,thick] (s2-1);
	\node (s3-1) at (1,1) {$s_3$} edge[<-,thick] (s2-1);
	\node (s4-1) at (1,0) {$s_4$} edge[<-,thick] (s2-1);
	\node (s2-2) at (2,0) {$s_2$} edge[<-,thick] (s1-1) edge[<-,thick,decorate,red] (s3-1) edge[<-,thick,decorate,red] (s4-1);
	\node at (1,-1.5) {$c_f = c_a\in \mathsf{Cox_a}$};

% Type Cox_b
	\node (s2-1) at (3,0) {$s_2$};
	\node (s1-1) at (4,-1) {$s_1$} edge[<-,decorate,red,thick] (s2-1);
	\node (s3-1) at (4,1) {$s_3$} edge[<-,thick] (s2-1);
	\node (s4-1) at (4,0) {$s_4$} edge[<-,thick,decorate,red] (s2-1);
	\node (s2-2) at (5,0) {$s_2$} edge[<-,thick] (s1-1) edge[<-,thick,decorate,red] (s3-1) edge[<-,thick] (s4-1);
	\node at (4,-1.5) {$c_f = c_b \in \mathsf{Cox_b}$};

% Type Cox_c
	\node (s2-1) at (6,0) {$s_2$};
	\node (s1-1) at (7,-1) {$s_1$} edge[<-,decorate,red,thick] (s2-1);
	\node (s3-1) at (7,1) {$s_3$} edge[<-,thick,decorate,red] (s2-1);
	\node (s4-1) at (7,0) {$s_4$} edge[<-,thick] (s2-1);
	\node (s2-2) at (8,0) {$s_2$} edge[<-,thick] (s1-1) edge[<-,thick] (s3-1) edge[<-,thick,decorate,red] (s4-1);
	\node at (7,-1.5) {$c_f = c_c \in \mathsf{Cox_c}$};

% Type Cox_max
	\node (s2-1) at (9,0) {$s_2$};
	\node (s1-1) at (10,-1) {$s_1$} edge[<-,decorate,red,thick] (s2-1);
	\node (s3-1) at (10,1) {$s_3$} edge[<-,thick,decorate,red] (s2-1);
	\node (s4-1) at (10,0) {$s_4$} edge[<-,thick,decorate,red] (s2-1);
	\node (s2-2) at (11,0) {$s_2$} edge[<-,thick] (s1-1) edge[<-,thick] (s3-1) edge[<-,thick] (s4-1);
	\node at (10,-1.5) {$c_f = c_{max} \in \mathsf{Cox_{max}}$};

\end{tikzpicture}
	\end{center}
	Corollary~\ref{cor:number_crossing_cut_paths} implies
	\[
		\mathsf{Q}_{c_a}=6-3\qquad
		\mathsf{Q}_{c_b}=6-3\qquad
		\mathsf{Q}_{c_c}=6-3\qquad
		\mathsf{Q}_{c_{max}}=3-3.
	\]
	This shows that if the number of crossing cut functions of $f\in \mathsf{CF}(D_4,S)$ is maximum then~$f$ has three extrema which is the 
	minimum number of extrema in this situation. 
	
	We now consider an extension from $(D_n,S)$ to $(D_{n+1},\widetilde S)$ with $n\geq 4$ by adding a new vertex adjacent to the leaf~$s_1$ 
	of $\Gamma_{D_n}$ and appropriate relabeling of generators. Thus $\Gamma_{D_n}$ corresponds to the subgraph of~$\Gamma_{D_{n+1}}$
	induced by~$\widetilde S\setminus\{s_1\}$ and every Coxeter element~$c_n\in\mathsf{Cox}(D_n,S)$ can be extended in precisely
	two ways to a Coxeter element~$c_{n+1} \in \mathsf{Cox}(D_{n+1},\widetilde S)$. Clearly, we have 
	\[
		I(T^{c_{n+1}}_1) \in\{1,2\} 
		\qquad\text{as well as}\qquad
		I(T^{c_n}_i) \leq I(T^{c_{n+1}}_{i+1}) \leq 2 I(T^{c_n}_i)\text{ for $i\in[n-2]$}.
	\]
	Thus 
	\begin{align*}
		\mathsf Q_{c_{n+1}}
			&=
			\sum_{i\in [n-1]} 
				2^{n-1-i} I\big(T^{c_{n+1}}_i\big) 
			- \big(2^{n-1}-1\big)\\
			&\leq
			2^{n-1-1}\cdot 2
				+ \sum_{i\in [n-2]}2^{n-2-i}\cdot\Big(2I\big(T^{c_n}_i\big)\Big) 
				- 2\big(2^{n-2}-1\big) - 1\\
			&=
			2^{n-1} + 2\mathsf Q_{c_n} - 1
	\end{align*}
	with equality if $I(T^{c_{n+1}}_{i+1}) = 2 I(T^{c_n}_i)$ for all $i\in[n-2]$ and $I(T^{c_{n+1}}_{1}) = 2$ 
	which happens if and only if 
	$\mathsf{out}(T^{c_{n+1}}_k)$ and $\mathsf{out}(T^{c_{n+1}}_\ell)$ coincide for all $1\leq k,\ell \leq n-1$.
	Thus, if~$Q_{c_{n+1}}$ is maximum then $c_{n+1} \in \mathsf{Cox_a} \subseteq \mathsf{Cox_{min}}$.
	In other words, if the number of cut functions that cross~$f$ is maximum then~$f$ has three extrema.
	This proves the claim if $(W,S)$ is of type~$D$.	

	Finally, we prove the claim in type $E$. We first analyze $E_6$. Clearly, removing the vertex $s_p=s_5$ from $\Gamma_{E_6}$ yields 
	a Coxeter graph of type~$D_{5}$. Let $c$ be a Coxeter element for type $E_6$ and $\widetilde c$ be the corresponding Coxeter element 
	for $(\widetilde W, S\setminus\{s_5\})$ of type~$D_5$. Since $I(T^{c}_k)=I(T^{\widetilde c}_k)$ for $1 \leq k \leq 3$, 
	we obtain $\mathsf Q_{c} = 2\mathsf Q_{\widetilde c} + I(T^{c}_4) - 1$.

	A case analysis reveals that $\mathsf Q_{c}$ is maximum in type~$E_6$ if and only if $c\in\mathsf{Cox_a}\cup\mathsf{Cox_b}$. Thus, if $f$ is a cut
	function with the maximum number of crossing cut functions then $\operatorname c_f\in \mathsf{Cox_a}\cup\mathsf{Cox_b}$, that is, $f$ has three 
	extrema. To solve the remaining cases~$E_7$ and~$E_8$ we extend from type $E_6$ to $E_7$ and from type~$E_7$ to $E_8$ similarly to the induction 
	step from~$D_n$ to $D_{n+1}$. Again, we obtain $\mathsf Q_{c_{n+1}} \leq 2^{n-1} + 2\mathsf Q_{c_n} - 1$ with equality if and only if 
	$\mathsf{out}(T^{c_{n+1}}_k)$ and $\mathsf{out}(T^{c_{n+1}}_\ell)$ coincide for all $1\leq k,\ell \leq n-1$. Therefore, if~$Q_{c_{n+1}}$ is maximum 
	then $c_{n+1} \in \mathsf{Cox_a} \subseteq \mathsf{Cox_{min}}$. This proves the claim if $(W,S)$ is of type~$E_7$ and $E_8$.
\end{proof}
	
We now characterize the Coxeter elements~$c$ that maximize and minimize the cardinality~$\mathsf S_c$ of a Cambrian acyclic 
domain~$\mathsf{Acyc}_c$.
\begin{theorem}\label{thm:bounds_acyc_domain}
Let $(W,S)$ be a finite irreducible Coxeter system, $c\in \mathsf {Cox}(W,S)$ and $\mathsf S_c=|\mathsf{Acyc}_\c|$. 
\begin{compactenum}[a)]
	\item The cardinality~$\mathsf S_c$ of~$\mathsf{Acyc}_c$ is maximum if and only if $c\in\mathsf{Cox_{max}}$.
	\item The cardinality~$\mathsf S_c$ of~$\mathsf{Acyc}_c$ is minimum if and only if
		\begin{compactenum}[i)]
			\item $c\in \mathsf{Cox_{min}}$ and $(W,S)$ is path-like or of type~$D_4$;
			\item $c\in \mathsf{Cox_a}\cup\mathsf{Cox_b}$ and $(W,S)$ is of type~$E_6$;
			\item $c\in \mathsf{Cox_a}$ and $(W,S)$ is of type~$E_7$, $E_8$ or $D_n$ for $n\geq 5$.
		\end{compactenum}
\end{compactenum}
\end{theorem}

\begin{proof}
	a) This is a consequence of Theorem~\ref{thm:singleton_cuts} and Corollary~\ref{cor:number_crossing_cut_paths}
	combined with Lemma~\ref{lem:extremal}.
	
	b) When $(W,S)$ is path-like or of type $D_4$, it follows immediately from Lemma~\ref{lem:extremal}.
	Otherwise, to decide the remaining cases for type~$D$ and~$E$, we use the relevant values for $I\big(T^\c_i\big)$ from the 
	proof of Proposition~\ref{prop:S_c_for_Cox_min}. We have to analyze $\c_{\mathsf a}\in\mathsf {Cox_a}$,
	$\c_{\mathsf b}\in\mathsf {Cox_b}$ and $\c_{\mathsf c}\in\mathsf {Cox_c}$. If $(W,S)$ is of type~$D$, we 
	obtain
	\[
		\mathsf{Q}_{\c_{\mathsf a}} = (n-4)2^{n-1}+2^{n-3}+1,
		\qquad
		\mathsf{Q}_{\c_{\mathsf b}} = (n-4)2^{n-1}+3
		\qquad\text{and}\qquad
		\mathsf{Q}_{\c_{\mathsf c}} = (n-4)2^{n-1}+3.
	\]
	The maximum is achieved by $\c_{\mathsf a}$, $\c_{\mathsf b}$ and $\c_{\mathsf c}$ if $n=4$ and only by~$\c_{\mathsf a}$ 
	if $n\geq 5$. If $(W,S)$ is of type~$E$, we similarly obtain
	\[
		\mathsf{Q}_{\c_{\mathsf a}} = (n-5)2^{n-2}+5\cdot 2^{n-4}+1,
		\hspace{0.5cm}
		\mathsf{Q}_{\c_{\mathsf b}} = (n-5)2^{n-2}+2^{n-2}+5
		\hspace{0.25cm}\text{and}\hspace{0.25cm}
		\mathsf{Q}_{\c_{\mathsf c}} = (n-5)2^{n-2}+2n+1.
	\]
	The maximum is achieved by $\c_{\mathsf a}$ and $\c_{\mathsf b}$ if $n=6$ and by $\c_{\mathsf a}$ if $n\in\{7,8\}$.
	In particular, this shows that the number of cut functions that cross~$f$ is not always maximized if the
	number of extrema of~$f$ is minimized.
\end{proof}

\section*{Acknowledgements}

The authors would like to thank Vic Reiner for pointing out his article with Galambos which initiated this work, and 
Cesar Ceballos and Vincent Pilaud for helpful discussions and their hospitality in Paris and Toronto.

%\bibliographystyle{amsalpha}
%\bibliography{biblio}

\providecommand{\bysame}{\leavevmode\hbox to3em{\hrulefill}\thinspace}
\providecommand{\MR}{\relax\ifhmode\unskip\space\fi MR }
% \MRhref is called by the amsart/book/proc definition of \MR.
\providecommand{\MRhref}[2]{%
  \href{http://www.ams.org/mathscinet-getitem?mr=#1}{#2}
}
\providecommand{\href}[2]{#2}

\end{document}